\documentclass[a4paper,10pt]{article}
\usepackage[utf8]{inputenc}
\usepackage[a4paper, total={7in, 10in}]{geometry}
\usepackage{microtype}
\usepackage{amssymb}
\usepackage{amsmath}
\usepackage{caption}
\usepackage{subcaption}
\usepackage{graphicx,amsthm,color}
\usepackage{diagbox}
\usepackage{float}
\usepackage[shortlabels]{enumitem}
\usepackage{hyperref,cleveref,cite,url}
\hypersetup{
 colorlinks,
 citecolor=green!50!black,
 linkcolor=red,
 urlcolor=blue!50!black}
\usepackage{makecell}
\usepackage{pgf,tikz}

\usepackage{tkz-graph}
\usetikzlibrary{arrows.meta,decorations.markings,bbox,snakes,shapes.misc}
\tikzset{decoration={snake}}
\tikzset{cross/.style={cross out, draw=black, minimum size=2*(#1-\pgflinewidth), inner sep=0pt, outer sep=0pt}}
\tikzset{->-/.style={decoration={
  markings,
  mark=at position .5 with {\arrow{>}}},postaction={decorate}}}
\usepackage{authblk}

\newtheorem{theorem}{Theorem}
\newtheorem{question}[theorem]{Question}
\newtheorem{proposition}[theorem]{Proposition}
\newtheorem{observation}[theorem]{Observation}
\newtheorem{corollary}[theorem]{Corollary}
\newtheorem{lemma}[theorem]{Lemma}
\newtheorem{claim}[theorem]{Claim}
\theoremstyle{definition}

\newtheorem{definition}[theorem]{Definition}
\newtheorem{conjecture}[theorem]{Conjecture}

\newcommand*{\myproofname}{Proof}

\newcommand{\A}{\mathcal{A}}
\newcommand{\I}{\mathcal{I}}
\newcommand{\F}{\mathcal{F}}
\newcommand{\CAI}{\rm{CAI}}
\newcommand{\AI}{\ensuremath{(\A,\I)}}
\newcommand{\ru}[1]{\textbf{R#1}}
\newcommand{\ora}[1]{\overrightarrow{#1}}
\renewcommand{\setminus}{-}

\setenumerate{label=\textup{(\arabic*)}, noitemsep, topsep=-\parskip, labelindent=.2em, leftmargin=*}
\setitemize{noitemsep, topsep=-\parskip, labelindent=.2em, leftmargin=*}

\title{Partitions of planar (oriented) graphs into a connected acyclic and an independent set}
\author[1]{Stijn Cambie}
\author[2]{François Dross}
\author[3]{Kolja Knauer}
\author[4]{Hoang La}
\author[5]{Petru Valicov}
\affil[1]{Department of Computer Science, KU Leuven Campus Kulak-Kortrijk, 8500 Kortrijk, Belgium.}
\affil[2]{LaBRI, CNRS, Université de Bordeaux, Bordeaux, France.}
\affil[3]{School of Mathematical Sciences, Hebei Key Laboratory of Computational Mathematics and Applications, Hebei Normal University, Shijiazhuang 050024, P. R. China and Departament de Matem\`atiques i Inform\`atica,
Universitat de Barcelona (UB), Barcelona, Spain.}
\affil[4]{LISN, Université Paris-Saclay, CNRS, Gif-sur-Yvette, France.}
\affil[5]{LIRMM, Université de Montpellier, CNRS, Montpellier, France.}

\begin{document}

\maketitle

\begin{abstract}
A question at the intersection of Barnette's Hamiltonicity and Neumann-Lara's dicoloring conjecture is: \emph{Can every Eulerian oriented planar graph be vertex-partitioned into two acyclic sets?}
A \CAI-partition of an undirected/oriented graph is a partition into a tree/connected acyclic subgraph and an independent set.
Consider any plane Eulerian oriented triangulation together with its unique tripartition, i.e. partition into three independent sets.
If two of these three sets induce a subgraph $G$ that has a \CAI-partition, then the above question has a positive answer.
We show that if $G$ is subcubic, then it has a \CAI-partition, i.e. oriented planar bipartite subcubic 2-vertex-connected graphs admit \CAI-partitions.
We also show that series-parallel 2-vertex-connected graphs admit \CAI-partitions.
Finally, we present a Eulerian oriented triangulation such that no two sets of its tripartition induce a graph with a \CAI-partition. This generalizes a result of Alt, Payne, Schmidt, and Wood to the oriented setting.
\end{abstract}

 \hfill  \textit{ “A problem worthy of attack, proves its worth by fighting back!” (Piet Hein)}

\section{Introduction}
A famous and widely open conjecture of Barnette says:

\begin{conjecture}[Barnette's Hamiltonicity Conjecture, 1969~\cite{B68}]\label{conj:primalBarnette}
Every $3$-connected cubic planar bipartite graph is Hamiltonian.
\end{conjecture}

Bipartiteness is important here, because if it is dropped, then the statement corresponds to Tait's Conjecture~\cite{Tait1884}, disproved by Tutte~\cite{Tutte46}. On the other hand, planarity is also essential, as shown by Horton~\cite{Hor82} who disproved a corresponding conjecture of Tutte~\cite{Tut71}.

Many partial and related results are available~\cite{APSW16,CO02,F24,Flo10,Lu11,BCD24,BFFS22,Flo16,Flo20a,Flo20b,Har13,K20}. In particular, \Cref{conj:primalBarnette} holds on graphs on up to 90 vertices~\cite{HMM85, BGM22}.

It is well-known and easy to see that the planar dual of a $3$-connected cubic planar bipartite graph is \emph{Eulerian}, i.e., it is connected and all its vertices have even degree. Moreover, the dual will be a planar \emph{triangulation}, i.e., all its faces are triangles. A subset of the vertices of an undirected graph is called \emph{acyclic} if it induces a forest. Finally, one can observe that after dualization one obtains the following equivalent statement of \Cref{conj:primalBarnette}.

\begin{conjecture}[Dual Barnette]\label{conj:Barnette}
Every Eulerian planar triangulation can be vertex-partitioned into two acyclic sets.
\end{conjecture}

The statement does not hold for general planar triangulations, because then it corresponds to Tait's Conjecture~\cite{Tait1884}. Indeed, there is a rich literature about decompositions of planar graphs into graphs close to forests, see e.g.~\cite{Tho95,knauer2024partitioning,RW08,KT09}.

We are now switching to \emph{oriented graphs}, i.e., directed graphs without cycles of length $1$ or $2$. A subset of the vertices of a directed graph is called \emph{acyclic} if it induces a subdigraph without directed cycles. Another relaxation of Tait's Conjecture is due to Neumann-Lara.

\begin{conjecture}[Neumann-Lara Dicoloring Conjecture, 1985~\cite{NL85}]\label{conj:NL}
Every oriented planar triangulation can be vertex-partitioned into two acyclic sets.
\end{conjecture}

\Cref{conj:NL} is settled in the absence of directed triangles~\cite{LM17} and for oriented graphs on at most $26$ vertices~\cite{KV19} but remains widely open. Note that in the primal setting, i.e. in the language of Hamiltonicity of $3$-connected graphs, also~\Cref{conj:NL} has a formulation and can be seen as a special case of a conjecture of Hochst\"attler~\cite{Hoc17} which has been disproved in~\cite{KV19}, where a detailed overview of the interplay of these conjectures have been given\footnote{See also \url{http://www.cs.toronto.edu/~ahertel/WebPageFiles/Papers/StrengtheningBarnette'sConjecture10.pdf}}. Together with results of Steiner~\cite[Corollary 5.40]{steiner2018} it follows that the largest open common special case of these conjectures is equivalent to:

\begin{conjecture}[Eulerian Neumann-Lara]\label{conj:EulerianNL}
Every Eulerian oriented planar graph can be vertex-partitioned into two acyclic sets.
\end{conjecture}

Here, a connected oriented graph is \emph{Eulerian} if for each of its vertices its out-degree and in-degree are equal. Note that when forgetting the orientations of a Eulerian oriented graph, one obtains a Eulerian undirected graph, but not vice versa (not every orientation is Eulerian). The main definition for the present paper is
\begin{definition}[\CAI-partition]
A partition $\A\cup\I=V$ of the vertices of an (oriented) graph $G=(V,E)$ is a \CAI\emph{-partition} if $\A$ induces a connected acyclic sub(di)graph and $\I$ is independent.
\end{definition}

The connection of \CAI-partitions to the above conjectures and simultaneously our central interest in their study is the following observation. For this, recall that every Eulerian planar triangulation has a unique \emph{tripartition}, i.e., a vertex-partitioning into three independent sets (see~\cite{TW11} for a new proof and a history of this result).

\begin{observation}\label{obs:main}
Let $G$ be a Eulerian (oriented) planar triangulation with tripartition $I_1,I_2,I_3$. If there exists $1\leq i\leq 3$ such that $G-I_i$ has a \CAI-partition, then $G$ can be vertex-partitioned into two acyclic sets $A_1, A_2$. Moreover, $A_1$ is connected and $A_2$ is a forest containing $I_i$ for some $1\leq i\leq 3$.
\end{observation}

Observation~\ref{obs:main} suggests a way of attacking the notoriously hard~\Cref{conj:Barnette} and~\Cref{conj:EulerianNL}. But what kind of graphs can appear and hence would need to be given a \CAI-partition?

\begin{observation}\label{obs:2-connected}
 An (oriented) graph $H$ is induced by two parts of the tripartition of a Eulerian (oriented) planar triangulation if and only if $H$ is a $2$-vertex-connected bipartite planar (oriented) graph.
\end{observation}

\subsubsection*{Related work}
To our knowledge \CAI-partitions have been studied only for undirected graphs. 

In~\cite{APSW16} the authors call a subtree of a Eulerian plane triangulation $G$ \emph{permeating} if it intersects every face and study the case where the tree avoids one class of the tripartition of $G$. More generally, let us call an acyclic connected subgraph $\mathcal{A}$ of a plane (oriented) $G$ \emph{permeating} if $\mathcal{A}$ intersects every face of $G$.  The following observation makes the connection with \CAI-partitions:
\begin{observation}\label{obs:permeat}
    Let $G$ be an undirected Eulerian triangulation with tripartition $I_1, I_2, I_3$. If $\mathcal{A}\cup\mathcal{I}$ is a \CAI-partition of $G-I_i$, then $\mathcal{A}$ is a permeating acyclic connected subgraph of $G$ and every permeating acyclic connected subgraph of $G$ that avoids $I_i$ arises like this.
\end{observation}

The negative result~\cite[Theorem 4]{APSW16} says that for every integer $k$ there is a properly $3$-coloured undirected Eulerian planar
triangulation $G$ such that every permeating tree of $G$ contains at least $k$ vertices
from each colour class. In particular, there are Eulerian triangulations $G$ with tripartition $I_1, I_2, I_3$ such that no $G\setminus I_i$ admits a \CAI-partition. 
With~\Cref{obs:2-connected} and~\Cref{obs:permeat} the positive result~\cite[Corollary 2]{APSW16} reads: $2$-vertex connected bipartite planar undirected graphs in which every cycle contains a vertex of degree $2$ have a \CAI-partition. 

\CAI-partitions have also been studied in non-planar graphs. Payan and Sakarovitch~\cite{PS75} show that cubic, 2-connected, cyclically $4$-edge connected graphs have a \CAI-partition if their order is not divisible by 4, but also give examples of order divisible by 4 without \CAI-partition. The case of cubic, 2-connected, cyclically $4$-edge connected graphs without \CAI-partition remains active, see~\cite{NS22,NSS23}.
In~\cite{CBOS24} it is shown NP-hard to decide if a graph (of diameter at most $3$) has a \CAI-partition.

\subsubsection*{Our results}
Our first and main positive result can be translated via~\Cref{obs:2-connected} and~\Cref{obs:main} into further evidence for~\Cref{conj:EulerianNL}. 

\begin{theorem}\label{thm:subcubic}
Every planar bipartite $2$-vertex-connected subcubic oriented graph has a \CAI-partition.
\end{theorem}

Our second positive result can be seen as a general contribution to \CAI-partitions in undirected graphs and when restricted to bipartite graphs it yields further positive evidence for~\Cref{conj:Barnette} via~\Cref{obs:2-connected} and~\Cref{obs:main}.

\begin{theorem}\label{thm:tw2}
Every $2$-vertex-connected simple series-parallel graph has a \CAI-partition.
\end{theorem}

We (almost) show the tightness of our positive results by showing that none of the conditions except possibly planarity in \Cref{thm:subcubic} can be dropped, see~\Cref{lem:no-AI-graphs}. See also~\Cref{quest}.

Finally, in \Cref{sec:neg}, we show that the strategy suggested by~\Cref{obs:main} is doomed to fail for resolving~\Cref{conj:EulerianNL} and thus its generalization~\Cref{conj:Barnette}. 

\begin{theorem}\label{thm:finalboss}
There exists a Eulerian oriented planar triangulation $G$ such that for any $I$ of its tripartition, the induced subgraph $H=G-I$ admits no \CAI-partition.
\end{theorem}

As a consequence of~\Cref{thm:finalboss} we obtain an oriented strengthening of~\cite[Theorem 4]{APSW16}:
\begin{corollary}\label{cor:finalboss}
 For every integer $k$ there is a properly $3$-coloured Eulerian oriented planar
triangulation $G$ such that every permeating acyclic connected subgraph $\mathcal{A}$ of $G$ contains at least $k$ vertices
from each colour class.
\end{corollary}

\subsection*{Definitions and notation}

Let $G=(V,E)$ be a (directed) graph.
We define the degree $d_G(u)$, the in-degree $d^-_G(u)$, and out-degree $d^+_G(u)$. 
We will drop the subscript $_G$ when the graph is clear from the context. A $k$-vertex (resp. $k^-$-vertex, $k^+$-vertex) is a vertex of degree $k$ (resp. at most $k$, at least $k$).
Let $G$ be a planar graph.
The degree of a face $f$ in $G$ is the number of edges of the face. The set of faces of $G$ is denoted by $F(G)$. A $k$-face is an induced cycle $C_k$.

For every set $S\subseteq V$, we denote by $G-S$ the graph $G$ where we removed the vertices of $S$ along with their incident edges. 


A \emph{bridge} is an edge whose removal disconnects the graph. A graph with no bridge is \emph{$2$-edge-connected}.

A \emph{cut-vertex} is a vertex whose removal disconnects the graph. A graph with no cut-vertex is \emph{$2$-vertex-connected}.

Note that a subcubic graph is $2$-vertex-connected if and only if it is $2$-edge-connected.

A set of vertices is \emph{separating} if its removal disconnects the graph. 

A \emph{cut-set} is a set of vertices that is separating.

Two vertices in $G$ are said to be at \emph{facial distance $d$ on a face $f$} if they are on the same face $f$ and their distance is $d$ in the induced subgraph $G[f]$.

When a graph $G$ is planar, we associate it with one of its plane drawings for simplicity. 
A triangulation is a maximal planar graph, i.e. a planar graph for adding an edge results into a non-planar graph, or equivalently a planar graph for which every face (also the outerface) is a triangle.

\section{Proofs of Observations}
\begin{proof}[Proof of~\Cref{obs:main}]
Take a \CAI-partition of $G-I_i$. Clearly $\A$ is a connected acyclic sub(di)graph of $G$. Now suppose for a contradiction that $\I\cup I_i$ induces a (not necessarily directed) cycle $C$ in $G$. 

Consider a planar embedding of $G$. Since $\A$ is connected and disjoint from $C$, we may assume without loss of generality that all vertices in $\A$ are outside of $C$ in the embedding. Let $v \in C$. By assumption, the vertex $v$ and all of its neighbors in $C$ or inside of $C$ belong to $V(G) \setminus \A = \I \cup I_i$. Note that any two consecutive neighbors of $v$ are adjacent in $G$, since $G$ is a triangulation. Since $C$ is a cycle, the vertex $v$ has at least two neighbors in $C$ or inside of $C$, hence $G[\I \cup I_i]$ contains a triangle, a contradiction.

Thus, $A_1=\A$ and $A_2=\I\cup I_i$ partition $G$ into a connected acyclic set and a forest containing $I_i$.
\end{proof}

\begin{proof}[Proof of~\Cref{obs:2-connected}]

We use the following well-known fact: a planar graph is $2$-vertex-connected if and only if all its faces are simple cycles, see e.g.~\cite[Chapter 2]{MT01}.
Let $G$ be a Eulerian (oriented) planar triangulation and tripartition $I_1,I_2,I_3$ and $H=G-I_i$ for some $1\leq i\leq 3$. Clearly, $H$ is a bipartite planar (oriented) graph. To see that it is $2$-vertex-connected, observe that every face of $H$ consists of the neighbors of a vertex of $I_i$ in their cyclic ordering. No vertex can appear twice in such a face by simplicity of $G$, hence all faces are simple cycles and $H$ is $2$-vertex-connected by the above result.

Conversely, if $H$ is a planar bipartite $2$-vertex-connected graph (let us for a moment forget about orientations), then by adding a vertex $v_f$ for each face $f$ of $H$ and edges between $v_f$ and the vertices of $f$, we obtain a planar triangulation $G$, which is simple because all faces are cycles. Moreover, each added $v_f$ will have even degree since $H$ is bipartite. For any vertex $v\in H$ its degree equals the number of faces incident to $v$ since $H$ is $2$-vertex-connected, so the degree of $v$ in $G$ is even. Thus, $G$ is a Eulerian planar and the added vertices form one of the independent sets in the tripartition of $G$. Finally, orient the new edges from $v_f$ towards an old vertex $v$ if $v$ is a source on $f$ and towards $v_f$ if $v$ is a sink on $f$. Since on each face the number of sinks and sources is equal, without the still unoriented edges every vertex has now indegree equal to outdegree. It is easy to see that the still unoriented edges form a subgraph all of whose vertices have even degree, hence we can give it a Eulerian orientation to satisfy the statement of the observation.
\end{proof}

\begin{proof}[Proof of~\Cref{obs:permeat}]
    If $\mathcal{A}\cup\mathcal{I}$ be a \CAI-partition of $G\setminus I_i$, then by~\Cref{obs:main} $\mathcal{I}\cup I_i$ is a forest in $G$ and in particular it cannot contain any face of $G$. Hence, $\mathcal{A}$ is a permeating acyclic connected subgraph of $G$ that avoids $I_i$.

    Conversely, if $\mathcal{A}$ is a permeating acyclic connected subgraph of $G$ that avoids $I_i$. Let $B=G\setminus\mathcal{A}$ be the remaining vertices of $G$. If $B\setminus I_i$ had an edge $e$, then since $G$ is a Eulerian triangulation and $I_1, I_2, I_3$ its tripartion, the triangles containing $e$ would have its third vertex in $I_i\subseteq B$. Hence, $B$ would contain a face. Thus, $\mathcal{I}=B\setminus I_i$ is independent. 
\end{proof}

\section{Proof of \texorpdfstring{\Cref{thm:subcubic}}{Theorem 8}}

We will prove~\Cref{thm:subcubic} using a discharging argument. Suppose by contradiction that there exists a counter-example $G$ of~\Cref{thm:subcubic} that minimizes the number of edges and vertices.

We call a $2$-vertex \emph{bad} if it is incident to a $6$-face, and \emph{good} otherwise. 
In order to prove the result, we will use the following proposition. Its proof will be given later.

\begin{proposition}\label{prop:reducible} 
The graph $G$ must have the following structural properties.
\begin{enumerate}[(i)]

\item\label{itm:2-vertices-distance-4} Two $2$-vertices are at facial distance at least $4$ (\Cref{lem:23b}).

\item\label{itm:no-4-faces} There are no $4$-faces (\Cref{lem:C4}).


\item\label{itm:2-vertices-8-faces} If an $8$-face contains two $2$-vertices, then none of them is bad (\Cref{lem:C8}).


\item\label{itm:2-vertices-6-faces} A $2$-vertex cannot be incident to two $6$-faces (\Cref{lem:2C6}).

\end{enumerate}
\end{proposition}

\begin{proof}[Proof of~\Cref{thm:subcubic}]
By Euler's formula, we have 
\begin{equation}\label{eq:euler}
\sum_{v\in V(G)}(2d(v)-6)+\sum_{f\in F(G)}(d(f)-6)=-12<0.
\end{equation}

We assign the charges $\mu(v)=2d(v)-6$ to each vertex $v\in V(G)$ and  $\mu(f)=d(f)-6$ to each face $f\in F(G)$. Now, we apply the following discharging rule. 

\textbf{Discharging rule}:
\begin{itemize}
\item[\ru0] Each $8^+$-face gives 2 to its bad $2$-vertices and 1 to its good $2$-vertices.  
\end{itemize}

\medskip

If~\Cref{prop:reducible} holds, then after applying \ru0, we will prove that the remaining charge $\mu^*$ on each face and each vertex is nonnegative, reaching a contradiction with~\Cref{eq:euler}. 

\textbf{Faces}: Recall that $G$ is bipartite. So $d(f)$ is even and $d(f)\geq 6$ for every $f\in F(G)$ by~\Cref{prop:reducible}\ref{itm:no-4-faces}.
\begin{itemize}
    \item Let $f$ be a $6$-face. Its charge is unchanged so $\mu^*(f)=\mu(f)=d(f)-6=0$.
    \item Let $f$ be an $8$-face. By \Cref{prop:reducible}\ref{itm:2-vertices-distance-4}, $f$ is incident to at most two $2$-vertices.
    By \Cref{prop:reducible}\ref{itm:2-vertices-8-faces}, if $f$ is incident to exactly two $2$-vertices, then none of them is bad. Therefore, $\mu^*(f)\geq 8-6-\max\{2\cdot 1,1\cdot 2\}=0$.



     \item Let $f$ be a $10^+$-face. By~\Cref{prop:reducible}\ref{itm:2-vertices-distance-4}, $f$ is incident to at most $\left\lfloor\frac{d(f)}{4}\right\rfloor$ $2$-vertices.
     Therefore, $\mu^*(f)=d(f)-6-2\left\lfloor\frac{d(f)}{4}\right\rfloor \geq 0$.
\end{itemize}
\medskip
\textbf{Vertices}: Let $v\in V(G)$, $v$ is a $2^+$-vertex since $G$ is $2$-vertex-connected. 
\begin{itemize}
    \item Let $v$ be a $2$-vertex. Recall that $\mu(v)=2d(v)-6=-2$.
    Since $v$ cannot be incident with two $6$-faces by \Cref{prop:reducible}\ref{itm:2-vertices-6-faces}, one of the following two cases occur.
\begin{itemize}
    \item If $v$ is incident to a $6$-face and an $8^+$-face, then it is a bad 2-vertex and it receives 2 from the $8^+$-face.
    \item If $v$ is incident to two $8^+$-faces, then it is a good 2-vertex and it receives 1 from each incident $8^+$-face.
\end{itemize}
Therefore, $\mu^*(v)=-2 + 1\cdot 2 = -2 + 2\cdot 1 = 0$.

\item Let $v$ be a $3$-vertex. Its charge is unchanged so $\mu^*(v)=\mu(v)=2d(v)-6= 2\cdot 3 - 6=0$. \qedhere
\end{itemize}
\end{proof}

\subsection*{Structural properties of $G$}

To prove~\Cref{prop:reducible}, we will study the structural properties of $G$ in greater detail. For conciseness, we will call the class of oriented planar bipartite 2-vertex-connected subcubic graphs $\F$ and when we talk about decompositions, we implicitly imply that it must be a partition into a connected acyclic set and an independent set.

\paragraph{Proof sketch.}
Every proof in this section will be by contradiction with the following scheme.
\begin{itemize}
    \item We build one (or two) graph(s) $H$ in $\F$ from $G$ such that $|E(H)|+|V(H)|<|E(G)|+|V(G)|$.
    \item We use the minimality of $G$ to obtain a \CAI-partition of $H$.
    \item We modify this \CAI-partition of $H$ to obtain a partition $(\A,\I)$ of $G$ that we claim is a \CAI-partition, thus obtaining a contradiction.
    \item The proofs that this new partition $(\A,\I)$ of $G$ is a \CAI-partition will consist in
    \begin{itemize}
        \item verifying that vertices in $\I$ form an independent set;
        \item verifying that new connections between vertices in $\A$ in $G$ will not create a directed cycle;
        \item if some connections between vertices in $\A$ in $H$ are not present in $G$ or if there were two disconnected graph $H_1$ and $H_2$, then we verify that $\A$ is connected.
    \end{itemize}
\end{itemize}
To avoid repetitions in this section, we will only argue that $H\in\F$ for restrictions that are not straightforward from the definition of $H$, which most of the time will be $2$-vertex-connectivity. Moreover, to help the reader see how the modification of $G$ to obtain $H$ preserves the bipartition, the vertices of one part will be labeled $a_i$ for some indices $i$, and the vertices in the other part with $b_j$ for the other indices $j$. We also often use the two following easy observations.

\begin{observation}
Let $v\in\A$. If $v$ has exactly one neighbor in $\A$, then $\A\setminus\{v\}$ is a connected acyclic set. 
\end{observation}

\begin{observation}
Let $v\notin \A$. If $v$ has exactly one neighbor in $\A$, then $\A\cup\{v\}$ is a connected acyclic set.
\end{observation}

We use edge (resp. path, cycle) instead of arc (resp. directed path, directed cycle), whenever the orientation can be omitted in the proof. We define an \emph{$\A$-path between $u$ and $v$} as a path between $u$ and $v$, where every vertex on this path is in $\A$, $u,v$ included. We define an \emph{$\A$-cycle} similarly. 

Proofs will also come with figures to illustrate the extension of the \CAI-partition of $H$ to $G$. Vertices and edges removed from $G$ to obtain $H$ will be in red. Vertices and edges added in $H$ will be in blue. Next to the vertices, we add labels $\A$ and $\I$ in blue according to the \CAI-partition in $H$ and in red for the extension to the \CAI-partition in $G$. The presence of (directed) $\A$-paths highlighted by the proof will be in the figures as (directed) squiggly lines between vertices in $\A$.

\begin{lemma}\label{lem:22}
There are no adjacent $2$-vertices in $G$.
\end{lemma}

\begin{proof}
Suppose by contradiction that we have a path $a_0b_1a_2b_3$ where $d(b_1)=d(a_2)=2$ in $G$. If $a_0$ and $b_3$ are not adjacent, then let $H=G-\{b_1,a_2\}+\ora{a_0b_3}$ when $\ora{a_0b_1}$ is an arc of $G$, otherwise let $H=G-\{b_1,a_2\}+\ora{b_3a_0}$. If $a_0$ and $b_3$ are adjacent, then let $H=G-\{b_1,a_2\}$. The resulting graph remains subcubic and bipartite. We check that $H$ is $2$-vertex-connected. Indeed, when $a_0$ and $b_3$ are not adjacent, replacing the path $a_0b_1a_2b_3$ by the edge $a_0b_3$ preserves the connectivity. In the case where $a_0$ and $b_3$ are adjacent in $G$, if removing $\{b_1,a_2\}$ creates a bridge in $H$, then this bridge along with $b_1a_2$ must be an edge-cut in $G$. We deduce that this edge cut must be $\{b_1a_2,a_0b_3\}$. This implies that $a_0$ or $b_3$ is a cut-vertex in $G$, or that $G$ is a cycle, a contradiction since $G$ is $2$-vertex-connected and cycles have a decomposition. 

Now, let \AI~be a \CAI-partition of $H$. Since $a_0$ and $b_3$ are adjacent in $H$, at most one of them can be in $\I$.\\ 
\textbf{Case 1:} $a_0$ and $b_3$ are not adjacent in $G$. See~\Cref{fig:lem:22-1}.\\
We claim that $(\A',\I')=(\A\cup\{a_2,b_1\},\I)$ is a \CAI-partition of $G$. Indeed, it is the case if either $a_0$ or $b_3$ is in $\I$. If they are both in $\A$, then the connectivity of $\A$ is preserved in $G$. Moreover, if there exists a directed $\A'$-cycle in $G$, then it must also exist in $H$ thanks to the added arc between $a_0$ and $b_3$.

\textbf{Case 2:} $a_0$ and $b_3$ are adjacent in $G$. See~\Cref{fig:lem:22-2}.\\
If either $a_0\in\I$ or $b_3\in\I$, then $(\A\cup\{b_1,a_2\},\I)$ is a \CAI-partition of $G$. If they are both in $\A$, then $(\A\cup\{b_1\},\I\cup\{a_2\})$ is a \CAI-partition of $G$.\qedhere
\end{proof}

\begin{figure}[!htbp]
\centering
\begin{subfigure}[b]{0.49\textwidth}
\begin{minipage}[b]{0.49\textwidth}
\centering
\begin{tikzpicture}[scale=0.55]
\begin{scope}[every node/.style={circle,draw,minimum size=0.65cm,inner sep = 2}]
    \node[label={below:\textcolor{blue}{$\A$}}] (a0) at (0,0) {$a_0$};
    \node[red, label={below:\textcolor{red}{$\A$}}] (b1) at (2,0) {$b_1$};
    \node[red, label={below:\textcolor{red}{$\A$}}] (a2) at (4,0) {$a_2$};
    \node[label={below:\textcolor{blue}{$\I$}}] (b3) at (6,0) {$b_3$};
\end{scope}

\begin{scope}[every edge/.style={draw,minimum width = 0.04cm}]
    \path[red] (a0) edge[->-] (b1);
    \path[red] (b1) edge (a2);
    \path[red] (a2) edge (b3);
    \path[blue] (a0) edge[->-,bend left] (b3);
\end{scope}
\end{tikzpicture}    
\end{minipage}
\begin{minipage}[b]{0.49\textwidth}
\centering
\begin{tikzpicture}[scale=0.55]
\begin{scope}[every node/.style={circle,draw,minimum size=0.65cm,inner sep = 2}]
    \node[label={below:\textcolor{blue}{$\A$}}] (a0) at (0,0) {$a_0$};
    \node[red, label={below:\textcolor{red}{$\A$}}] (b1) at (2,0) {$b_1$};
    \node[red, label={below:\textcolor{red}{$\A$}}] (a2) at (4,0) {$a_2$};
    \node[label={below:\textcolor{blue}{$\A$}}] (b3) at (6,0) {$b_3$};
\end{scope}

\begin{scope}[every edge/.style={draw,minimum width = 0.04cm}]
    \path[red] (a0) edge[->-] (b1);
    \path[red] (b1) edge (a2);
    \path[red] (a2) edge (b3);
    \path[blue] (a0) edge[->-,bend left] (b3);
\end{scope}
\end{tikzpicture} 
\end{minipage}
\caption{Case 1.}
\label{fig:lem:22-1}
\end{subfigure}
\begin{subfigure}[b]{0.49\textwidth}
\begin{minipage}[b]{0.49\textwidth}
\centering
\begin{tikzpicture}[scale=0.55]
\begin{scope}[every node/.style={circle,draw,minimum size=0.65cm,inner sep = 2}]
    \node[label={below:\textcolor{blue}{$\A$}}] (a0) at (0,0) {$a_0$};
    \node[red, label={below:\textcolor{red}{$\A$}}] (b1) at (2,0) {$b_1$};
    \node[red, label={below:\textcolor{red}{$\A$}}] (a2) at (4,0) {$a_2$};
    \node[label={below:\textcolor{blue}{$\I$}}] (b3) at (6,0) {$b_3$};
\end{scope}

\begin{scope}[every edge/.style={draw,minimum width = 0.04cm}]
    \path[red] (a0) edge (b1);
    \path[red] (b1) edge (a2);
    \path[red] (a2) edge (b3);
    \path (a0) edge[bend left] (b3);
\end{scope}
\end{tikzpicture}    
\end{minipage}
\begin{minipage}[b]{0.49\textwidth}
\centering
\begin{tikzpicture}[scale=0.55]
\begin{scope}[every node/.style={circle,draw,minimum size=0.65cm,inner sep = 2}]
    \node[label={below:\textcolor{blue}{$\A$}}] (a0) at (0,0) {$a_0$};
    \node[red, label={below:\textcolor{red}{$\A$}}] (b1) at (2,0) {$b_1$};
    \node[red, label={below:\textcolor{red}{$\I$}}] (a2) at (4,0) {$a_2$};
    \node[label={below:\textcolor{blue}{$\A$}}] (b3) at (6,0) {$b_3$};
\end{scope}

\begin{scope}[every edge/.style={draw,minimum width = 0.04cm}]
    \path[red] (a0) edge (b1);
    \path[red] (b1) edge (a2);
    \path[red] (a2) edge (b3);
    \path (a0) edge[bend left] (b3);
\end{scope}
\end{tikzpicture} 
\end{minipage}
\caption{Case 2.}
\label{fig:lem:22-2}
\end{subfigure}
\caption{\Cref{lem:22}.}
\end{figure}

Since $G$ is bipartite, containing a 4-cycle as a subgraph is the same as containing it as an induced subgraph, so there is no ambiguity in the statements that will follow.

\begin{lemma}\label{lem:2vC4}
There are no $2$-vertices on a 4-cycle in $G$.
\end{lemma}

\begin{proof}
Suppose by contradiction that there exists a cycle $C=a_0b_1a_2b_3$ where $d(a_0)=2$. By~\Cref{lem:22}, $d(b_1)=d(b_3)=3$.
Let $H=G-\{a_0\}$. See~\Cref{fig:lem:2vC4}.
Observe that $H$ is $2$-vertex-connected. 
Indeed, if $H$ is not $2$-vertex-connected, then there is a cut-vertex $v$ in $H$ such that $\{v,a_0\}$ is a cut-set in $G$.
Since removing $a_0$ could only separate $b_1$ and $b_3$, $v$ must be $a_2$. However, this implies that $b_1$ or $b_3$ is a cut-vertex in $G$, a contradiction.

Let $(\A,\I)$ be a \CAI-partition of $H$. See~\Cref{fig:lem:2vC4}.
If $b_1$ and $b_3$ are in $\A$, then $(\A,\I\cup\{a_0\})$ is a \CAI-partition of $G$. If only one of $b_1$ and $b_3$ is in $\A$, then $(\A\cup\{a_0\},\I)$ is a \CAI-partition of $G$. Finally, suppose $b_1\in\I$ and $b_3\in\I$. Since $(\A,\I)$ is a \CAI-partition of $H$, then $a_2$ must be in $\A$ and also must have a third neighbor in $\A$. 
Note that both $b_1$ and $b_3$ have degree three by Lemma \ref{lem:22}, and thus each has a third neighbor in $\A$, which is connected to the rest of $\A$.
Therefore, $((\A\setminus\{a_2\})\cup\{b_1,b_3\},(\I\setminus\{b_1,b_3\})\cup\{a_0,a_2\})$ is a \CAI-partition of $G$.\qedhere

\begin{figure}[!htbp]
\begin{minipage}[b]{0.33\textwidth}
\centering
\begin{tikzpicture}
\begin{scope}[every node/.style={circle,draw,minimum size=0.65cm,inner sep = 2}]
    \node[red,label={above:\textcolor{red}{$\I$}}] (a0) at (1,1) {$a_0$};
    \node[label={above:\textcolor{blue}{$\A$}}] (b1) at (0,0) {$b_1$};
    \node (a2) at (1,-1) {$a_2$};
    \node[label={above:\textcolor{blue}{$\A$}}] (b3) at (2,0) {$b_3$};
    \node[draw=none] (b'2) at (1,-2) {};
\end{scope}

\begin{scope}[every edge/.style={draw,minimum width = 0.04cm}]
    \path[red] (a0) edge (b1);
    \path (b1) edge (a2);
    \path (a2) edge (b3);
    \path[red] (a0) edge (b3);
\end{scope}
\end{tikzpicture}    
\end{minipage}
\begin{minipage}[b]{0.33\textwidth}
\centering
\begin{tikzpicture}
\begin{scope}[every node/.style={circle,draw,minimum size=0.65cm,inner sep = 2}]
    \node[red,label={above:\textcolor{red}{$\A$}}] (a0) at (1,1) {$a_0$};
    \node[label={above:\textcolor{blue}{$\A$}}] (b1) at (0,0) {$b_1$};
    \node (a2) at (1,-1) {$a_2$};
    \node[label={above:\textcolor{blue}{$\I$}}] (b3) at (2,0) {$b_3$};
    \node[draw=none] (b'2) at (1,-2) {};
\end{scope}

\begin{scope}[every edge/.style={draw,minimum width = 0.04cm}]
    \path[red] (a0) edge (b1);
    \path (b1) edge (a2);
    \path (a2) edge (b3);
    \path[red] (a0) edge (b3);
\end{scope}
\end{tikzpicture}
\end{minipage}
\begin{minipage}[b]{0.33\textwidth}
\centering
\begin{tikzpicture}
\begin{scope}[every node/.style={circle,draw,minimum size=0.65cm,inner sep = 2}]
    \node[red,label={above:\textcolor{red}{$\I$}}] (a0) at (1,1) {$a_0$};
    \node[label={[label distance = -8pt]above:\textcolor{blue}{$\I$}$\rightarrow$\textcolor{red}{$\A$}}] (b1) at (0,0) {$b_1$};
    \node[label={[label distance = -8pt]above:\textcolor{blue}{$\A$}$\rightarrow$\textcolor{red}{$\I$}}] (a2) at (1,-1) {$a_2$};
    \node[label={[label distance = -8pt]above:\textcolor{blue}{$\I$}$\rightarrow$\textcolor{red}{$\A$}}] (b3) at (2,0) {$b_3$};
    \node[draw=none] (b'2) at (1,-2) {\textcolor{blue}{$\A$}};
\end{scope}

\begin{scope}[every edge/.style={draw,minimum width = 0.04cm}]
    \path[red] (a0) edge (b1);
    \path (b1) edge (a2);
    \path (a2) edge (b3);
    \path[red] (a0) edge (b3);
    \path (a2) edge (b'2);
\end{scope}
\end{tikzpicture} 
\end{minipage}
\caption{\Cref{lem:2vC4}.}
\label{fig:lem:2vC4}
\end{figure}
\end{proof}

\begin{lemma}\label{lem:232}
Two $2$-vertices are at distance at least $3$ in $G$. 
\end{lemma}

\begin{proof}
Suppose by contradiction that the underlying undirected graph of $G$ has a path $a_0b_1a_2b_3a_4$ where $d(b_1)=d(b_3)=2$ in $G$. By~\Cref{lem:22}, we know that $d(a_2)=3$ so let $b'_2\notin\{b_1,b_3\}$ be its third neighbor. By~\Cref{lem:2vC4} we know that $a_0\neq a_4$. 
Let $H=G-\{b_3\}+\ora{b_1a_4}$. By adding the edge $b_1a_4$, we ensure the 2-connectivity of $H$, otherwise $b_1a_2$ is a bridge in $H$ and thus $a_2b_2'$ is a bridge in $G$. Let \AI~be a \CAI-partition of $H$. We have to distinguish several cases:

\textbf{Case 1:} Suppose that $a_2$ and $a_4$ are in $\A$. See~\Cref{fig:lem:232-1}.\\
If there is an $\A$-path between $a_2$ and $a_4$ in $G-\{b_3\}$, then $(\A,\I\cup\{b_3\})$ is a \CAI-partition of $G$. Otherwise, $(\A\cup\{b_3\},\I)$ is a \CAI-partition of $G$.
 
\textbf{Case 2:} Suppose that $a_2\in\A$ and $a_4\in\I$. See~\Cref{fig:lem:232-2}.\\
In this case, $(\A\cup\{b_3\},\I)$ is a \CAI-partition of $G$.
 
\textbf{Case 3:} Suppose that $a_2\in\I$ and $a_4\in\A$. See~\Cref{fig:lem:232-3}.\\
Since $a_2\in\I$, we must have $b_1$ and $b'_2$ in $\A$.
\begin{itemize}
\item If there is an $\A$-path between $a_4$ and $b_1$ in $G-\{b_3\}$, then $(\A\cup\{b_3\},\I)$ is a \CAI-partition of $G$.
\item Otherwise, if there is an $\A$-path between $b'_2$ and $b_1$ (which must go through $a_0$) in $G-\{b_3\}$, then $((\A\setminus\{b_1\})\cup\{a_2,b_3\},(\I\setminus\{a_2\})\cup\{b_1\})$ is a \CAI-partition of $G$.
\item If both of the previous conditions do not hold, then there must be an $\A$-path between $b'_2$ and $a_4$ in $G-\{b_3\}$ since $\A$ is connected in $H$. In this case, $(\A\cup\{a_2\},(\I\setminus\{a_2\})\cup\{b_3\})$ is a \CAI-partition of $G$.\\
\end{itemize} 
 
\textbf{Case 4:} Suppose that $a_2\in\I$ and $a_4\in\I$. See~\Cref{fig:lem:232-4}.\\
Since $a_2\in\I$, we must have $b_1$ and $b'_2$ in $\A$. Moreover, there must be an $\A$-path between $b_1$ and $b'_2$ since $\A$ is connected in $H$. Therefore, $((\A\setminus\{b_1\})\cup\{a_2,b_3\},(\I\setminus\{a_2\})\cup\{b_1\})$ is a \CAI-partition of $G$. \qedhere
\end{proof}

\begin{figure}[!htbp]
\centering
\begin{subfigure}[b]{0.66\textwidth}
\centering
\begin{minipage}[b]{0.49\textwidth}
\centering
\begin{tikzpicture}[scale=0.75]
\begin{scope}[every node/.style={circle,draw,minimum size=0.65cm,inner sep = 2}]
    \node (b1) at (2,0) {$b_1$};
    \node[label={below:\textcolor{blue}{$\A$}}] (a2) at (4,0) {$a_2$};
    \node[red,label={below:\textcolor{red}{$\I$}}] (b3) at (6,0) {$b_3$};
    \node[label={below:\textcolor{blue}{$\A$}}] (a4) at (8,0) {$a_4$};
    
\end{scope}

\begin{scope}[every edge/.style={draw,minimum width = 0.04cm}]
    \path (b1) edge (a2);
    \path[red] (a2) edge (b3);
    \path[red] (b3) edge (a4);
    
    \path[blue] (b1) edge[->-,bend left] (a4);

    \path (a2) edge[bend left=50,decorate] (a4); 
\end{scope}
\end{tikzpicture}    
\end{minipage}
\begin{minipage}[b]{0.49\textwidth}
\centering
\begin{tikzpicture}[scale=0.75]
\begin{scope}[every node/.style={circle,draw,minimum size=0.65cm,inner sep = 2}]
    \node (b1) at (2,0) {$b_1$};
    \node[label={below:\textcolor{blue}{$\A$}}] (a2) at (4,0) {$a_2$};
    \node[red,label={below:\textcolor{red}{$\A$}}] (b3) at (6,0) {$b_3$};
    \node[label={below:\textcolor{blue}{$\A$}}] (a4) at (8,0) {$a_4$};
    
\end{scope}

\begin{scope}[every edge/.style={draw,minimum width = 0.04cm}]
    \path (b1) edge (a2);
    \path[red] (a2) edge (b3);
    \path[red] (b3) edge (a4);
    
    \path[blue] (b1) edge[->-,bend left] (a4);

\end{scope}
\end{tikzpicture}    
\end{minipage}
\caption{Case 1.}
\label{fig:lem:232-1}
\end{subfigure}
\begin{subfigure}[b]{0.33\textwidth}
\centering
\begin{minipage}[b]{0.99\textwidth}
\centering
\begin{tikzpicture}[scale=0.75]
\begin{scope}[every node/.style={circle,draw,minimum size=0.65cm,inner sep = 2}]
    \node (b1) at (2,0) {$b_1$};
    \node[label={below:\textcolor{blue}{$\A$}}] (a2) at (4,0) {$a_2$};
    \node[red,label={below:\textcolor{red}{$\A$}}] (b3) at (6,0) {$b_3$};
    \node[label={below:\textcolor{blue}{$\I$}}] (a4) at (8,0) {$a_4$};
    
\end{scope}

\begin{scope}[every edge/.style={draw,minimum width = 0.04cm}]
    \path (b1) edge (a2);
    \path[red] (a2) edge (b3);
    \path[red] (b3) edge (a4);
    
    \path[blue] (b1) edge[->-,bend left] (a4);

\end{scope}
\end{tikzpicture}    
\end{minipage}
\caption{Case 2.}
\label{fig:lem:232-2}
\end{subfigure}

\begin{subfigure}[b]{0.64\textwidth}
\centering
\begin{minipage}[b]{0.99\textwidth}
\centering
\begin{tikzpicture}[scale=0.66]
\begin{scope}[every node/.style={circle,draw,minimum size=0.65cm,inner sep = 2}]
    \node[label={below:\textcolor{blue}{$\A$}}] (b1) at (2,0) {$b_1$};
    \node[label={below:\textcolor{blue}{$\I$}}] (a2) at (4,0) {$a_2$};
    \node[red,label={below:\textcolor{red}{$\A$}}] (b3) at (6,0) {$b_3$};
    \node[label={below:\textcolor{blue}{$\A$}}] (a4) at (8,0) {$a_4$};
    
\end{scope}

\begin{scope}[every edge/.style={draw,minimum width = 0.04cm}]
    \path (b1) edge (a2);
    \path[red] (a2) edge (b3);
    \path[red] (b3) edge (a4);
    
    \path[blue] (b1) edge[->-,bend left] (a4);

    \path (b1) edge[bend left=50,decorate] (a4); 
\end{scope}
\end{tikzpicture}    
\end{minipage}

\begin{minipage}[b]{0.52\textwidth}
\centering
\begin{tikzpicture}[scale=0.66]
\begin{scope}[every node/.style={circle,draw,minimum size=0.65cm,inner sep = 2}]
    \node[label={below:\textcolor{blue}{$\A$}}] (a0) at (0,0) {$a_0$};
    \node[label={[label distance=-6pt]below:\textcolor{blue}{$\A$}$\rightarrow$\textcolor{red}{$\I$}}] (b1) at (2,0) {$b_1$};
    \node[label={[label distance=-6pt]below:\textcolor{blue}{$\I$}$\rightarrow$\textcolor{red}{$\A$}}] (a2) at (4,0) {$a_2$};
    \node[red,label={below:\textcolor{red}{$\A$}}] (b3) at (6,0) {$b_3$};
    \node[label={below:\textcolor{blue}{$\A$}}] (a4) at (8,0) {$a_4$};
    
    \node[label={below:\textcolor{blue}{$\A$}}] (b'2) at (4,-2) {$b'_2$};
\end{scope}

\begin{scope}[every edge/.style={draw,minimum width = 0.04cm}]
    \path (a0) edge (b1);
    \path (b1) edge (a2);
    \path[red] (a2) edge (b3);
    \path[red] (b3) edge (a4);
    \path (a2) edge (b'2);
    
    \path[blue] (b1) edge[->-,bend left] (a4);

    \path (a0) edge[bend right,decorate] (b'2); 
\end{scope}
\end{tikzpicture}    
\end{minipage}
\begin{minipage}[b]{0.37\textwidth}
\centering
\begin{tikzpicture}[scale=0.66]
\begin{scope}[every node/.style={circle,draw,minimum size=0.65cm,inner sep = 2}]
    \node[label={below:\textcolor{blue}{$\A$}}] (b1) at (2,0) {$b_1$};
    \node[label={[label distance=-6pt]below:\textcolor{blue}{$\I$}$\rightarrow$\textcolor{red}{$\A$}}] (a2) at (4,0) {$a_2$};
    \node[red,label={below:\textcolor{red}{$\I$}}] (b3) at (6,0) {$b_3$};
    \node[label={below:\textcolor{blue}{$\A$}}] (a4) at (8,0) {$a_4$};
    
    \node[label={below:\textcolor{blue}{$\A$}}] (b'2) at (4,-2) {$b'_2$};
\end{scope}

\begin{scope}[every edge/.style={draw,minimum width = 0.04cm}]
    \path (b1) edge (a2);
    \path[red] (a2) edge (b3);
    \path[red] (b3) edge (a4);
    \path (a2) edge (b'2);
    
    \path[blue] (b1) edge[->-,bend left] (a4);

    \path (b'2) edge[bend right,decorate] (a4); 
\end{scope}
\end{tikzpicture}    
\end{minipage}
\caption{Case 3.}
\label{fig:lem:232-3}
\end{subfigure}
\begin{subfigure}[b]{0.35\textwidth}
\centering
\begin{minipage}[b]{0.99\textwidth}
\centering
\begin{tikzpicture}[scale=0.66]
\begin{scope}[every node/.style={circle,draw,minimum size=0.65cm,inner sep = 2}]
    \node[label={below:\textcolor{blue}{$\A$}}] (a0) at (0,0) {$a_0$};
    \node[label={[label distance = -6pt]below:\textcolor{blue}{$\A$}$\rightarrow$\textcolor{red}{$\I$}}] (b1) at (2,0) {$b_1$};
    \node[label={[label distance = -6pt]below:\textcolor{blue}{$\I$}$\rightarrow$\textcolor{red}{$\A$}}] (a2) at (4,0) {$a_2$};
    \node[red,label={below:\textcolor{red}{$\A$}}] (b3) at (6,0) {$b_3$};
    \node[label={below:\textcolor{blue}{$\I$}}] (a4) at (8,0) {$a_4$};
    
    \node[label={below:\textcolor{blue}{$\A$}}] (b'2) at (4,-2) {$b'_2$};
\end{scope}

\begin{scope}[every edge/.style={draw,minimum width = 0.04cm}]
    \path (a0) edge (b1);
    \path (b1) edge (a2);
    \path[red] (a2) edge (b3);
    \path[red] (b3) edge (a4);
    \path (a2) edge (b'2);
    
    \path[blue] (b1) edge[->-,bend left] (a4);

    \path (a0) edge[bend right,decorate] (b'2); 
\end{scope}
\end{tikzpicture}    
\end{minipage}
\caption{Case 4.}
\label{fig:lem:232-4}
\end{subfigure}
\caption{\Cref{lem:232}.}
\label{fig:lem:232}
\end{figure}

To prove that $G$ contains no $4$-faces (\Cref{prop:reducible}\ref{itm:no-4-faces}), we need to prove~\Cref{lem:3C4,lem:2C4,lem:separating} first.

\begin{lemma}\label{lem:3C4}
There are no three distinct 4-cycles in $G$, each sharing at least one edge with each other.
\end{lemma}

\begin{proof}
Suppose that such a configuration exists by contradiction. 
Due to~\Cref{lem:2vC4} and the fact that $G$ is planar, bipartite, subcubic, and $2$-vertex-connected, the only possible drawing of such a configuration is presented in~\Cref{fig:lem:3C4} along with the name of the vertices. Note that not every $b'_i$ is necessarily distinct from each other.
The three 4-cycles cannot be all directed so let $C$ be the set of vertices of a non-directed 4-cycle. 
If every $b'_i$ is the same vertex, then $G$ is an orientation of the cube. We can put all of the vertices of $C$ in $\A$, along with two non-adjacent vertices among the remaining ones, and the last two vertices in $\I$, to get a \CAI-partition of $G$. Therefore we may assume that not every $b'_i$ is the same vertex.

Let $H$ be $G$ where we identify $a_0,a_1,b_2,a_3,b_4,a_5,b_6$ into one vertex $a^*$. If this causes two arcs to be merged into one, we orient it in the opposite direction to the one of the other arc incident to $a^*$.
Observe that if $H$ has a bridge, then it must be one that is incident to $a^*$, and both arcs incident to $a^*$ if $a^*$ has degree $2$.
But then, one of those arcs would also be a bridge in $G$, a contradiction. Therefore, $H\in\mathcal{F}$. Let \AI~be a \CAI-partition of $H$. In what follows, we give a \CAI-partition of $G$ in every possible case up to the symmetry of the configuration. 

Observe that $b'_1$, $b'_3$, and $b'_5$ cannot all be in $\I$, otherwise $a^*$ would be an isolated vertex in $\A$. Therefore, we have the following cases.

\textbf{Case 1:} $a^*\in\I$. See~\Cref{fig:lem:3C4-1}.\\
We must have $\{b'_1,b'_3,b'_5\}\subseteq \A$. By the pigeonhole principle and w.l.o.g. we assume the existence of arcs $\ora{a_1b'_1}$ and $\ora{a_3b'_3}$. In that case, $(\A\cup\{a_1,b_2,a_3,b_4,b_6\},(\I\setminus\{a^*\})\cup\{a_0,a_5\})$ is a \CAI-partition of $G$. 

\textbf{Case 2:} $a^*\in\A$. See~\Cref{fig:lem:3C4-2}.\\
Let $b\in\{b_2,b_4,b_6\}\setminus C$.

Suppose first that every $b'_i$ is distinct. We claim that $(\A',\I')=((\A\setminus\{a^*\})\cup\{a_0,a_1,b_2,a_3,b_4,a_5,b_6\}\setminus\{b\},\I\cup\{b\})$ is a \CAI-partition of $G$. The only possible problem with this decomposition is a directed $\A'$-cycle.
However, any such cycle in $G$ that contains two of the $b'_i$s will be a directed $\A$-cycle in $H$ that goes through $a^*$.
Moreover, the only other possible directed $\A'$-cycle is the 4-cycle that does not contain $b$. This is impossible since it is $C$ which is not directed.

Now suppose not every $b'_i$ is distinct, say $b'_1 = b'_3$ without loss of generality.
Then we can put another vertex $\hat{b}$ of $\{b_2,b_4,b_6\}$ in $I'$ without disconnecting $\A'$. We choose $\hat{b}$ so that $b$ and $\hat{b}$ are not both adjacent to $a_5$. Now $(\A',\I')=((\A\setminus\{a^*\})\cup\{a_0,a_1,b_2,a_3,b_4,a_5,b_6\}\setminus\{b,\hat{b}\},\I\cup\{b,\hat{b}\})$ is a \CAI-partition of $G$. The only additional potential directed $\A'$-cycle that could appear compared to the previous paragraph is a directed cycle containing $b'_1 = b'_3$ and not $b'_5$. But any such cycle contains either $b$ or $\hat{b}$, which is in $\I'$.
\end{proof}

\begin{figure}[!htbp]
\centering
\begin{subfigure}[b]{0.59\textwidth}
\centering
\begin{tikzpicture}[scale=0.7]
\begin{scope}[every node/.style={circle,draw,minimum size=0.65cm,inner sep = 2}]
    \node[red,label={above:\textcolor{red}{$\I$}}] (a0) at (2,2) {$a_0$};
    \node[red,label={below:\textcolor{red}{$\A$}}] (a1) at (0,1) {$a_1$};
    \node[red,label={below:\textcolor{red}{$\A$}}] (b2) at (2,0) {$b_2$};
    \node[red,label={below:\textcolor{red}{$\A$}}] (a3) at (4,1) {$a_3$};
    \node[red,label={above:\textcolor{red}{$\A$}}] (b4) at (4,3) {$b_4$};
    \node[red,label={above right:\textcolor{red}{$\I$}}] (a5) at (2,4) {$a_5$};
    \node[red,label={above:\textcolor{red}{$\A$}}] (b6) at (0,3) {$b_6$};
    
    \node[label={below:\textcolor{blue}{$\A$}}] (b'1) at (-2,0) {$b'_1$};
    \node[label={below:\textcolor{blue}{$\A$}}] (b'3) at (6,0) {$b'_3$};
    \node[label={above:\textcolor{blue}{$\A$}}] (b'5) at (2,6) {$b'_5$};

    \node[draw=none] (arrow) at (7,2) {$\longleftrightarrow$};
    
    \node[blue,label={above right:\textcolor{blue}{$\I$}}] (a*) at (10,2) {$a^*$};
    \node[label={below:\textcolor{blue}{$\A$}}] (hb'1) at (8,1) {$b'_1$};
    \node[label={below:\textcolor{blue}{$\A$}}] (hb'3) at (12,1) {$b'_3$};
    \node[label={above:\textcolor{blue}{$\A$}}] (hb'5) at (10,4) {$b'_5$};
\end{scope}

\begin{scope}[every edge/.style={draw,minimum width = 0.04cm}]
    \path[red] (a1) edge (b2);
    \path[red] (b2) edge (a3);
    \path[red] (a3) edge (b4);
    \path[red] (b4) edge (a5);
    \path[red] (a5) edge (b6);
    \path[red] (b6) edge (a1);
    
    \path[red] (a0) edge (b2);
    \path[red] (a0) edge (b4);
    \path[red] (a0) edge (b6);

    \path[red] (a1) edge[->-] (b'1);
    \path[red] (a3) edge[->-] (b'3);
    \path[red] (b'5) edge (a5);
    
    \path[blue] (a*) edge[->-] (hb'1);
    \path[blue] (a*) edge[->-] (hb'3);
    \path[blue] (hb'5) edge (a*);
\end{scope}
\end{tikzpicture}
\caption{\label{fig:lem:3C4-1}Case 1.}    
\end{subfigure}
\begin{subfigure}[b]{0.39\textwidth}
\centering
\begin{tikzpicture}[scale=0.7]
\begin{scope}[every node/.style={circle,draw,minimum size=0.65cm,inner sep = 2}]
    \node[red,label={above:\textcolor{red}{$\A$}}] (a0) at (2,2) {$a_0$};
    \node[red,label={below:\textcolor{red}{$\A$}}] (a1) at (0,1) {$a_1$};
    \node[red,label={below:\textcolor{red}{$\A$}}] (b2) at (2,0) {$b_2$};
    \node[red,label={below:\textcolor{red}{$\A$}}] (a3) at (4,1) {$a_3$};
    \node[red,label={above:\textcolor{red}{$\A$}}] (b4) at (4,3) {$b_4$};
    \node[red,label={above right:\textcolor{red}{$\A$}}] (a5) at (2,4) {$a_5$};
    \node[red,label={above:\textcolor{red}{$\I$}}] (b6) at (0,3) {$b_6$};
    
    \node (b'1) at (-2,0) {$b'_1$};
    \node (b'3) at (6,0) {$b'_3$};
    \node (b'5) at (2,6) {$b'_5$};

    
\end{scope}

\begin{scope}[every edge/.style={draw,minimum width = 0.04cm}]
    \path[red] (a1) edge (b2);
    \path[red] (b2) edge (a3);
    \path[red] (a3) edge (b4);
    \path[red] (b4) edge (a5);
    \path[red] (a5) edge (b6);
    \path[red] (b6) edge (a1);
    
    \path[red] (a0) edge (b2);
    \path[red] (a0) edge (b4);
    \path[red] (a0) edge (b6);

    \path[red] (a1) edge (b'1);
    \path[red] (a3) edge (b'3);
    \path[red] (b'5) edge (a5);
    
\end{scope}
\end{tikzpicture}
\caption{\label{fig:lem:3C4-2}Case 2 where $C=\{a_0,b_2,a_3,b_4\}$ and $b=b_6$.}
\end{subfigure}
\caption{\Cref{lem:3C4}.}
\label{fig:lem:3C4}
\end{figure}

Using~\Cref{lem:3C4}, we can prove~\Cref{lem:2C4}.

\begin{lemma}\label{lem:2C4}
There are not two $4$-cycles sharing an edge in $G$.
\end{lemma}

\begin{proof}
Suppose that such a configuration exists by contradiction. Note that since $G$ is $2$-connected and planar, there cannot be two $4$-cycles that share two edges. We give a drawing of such a configuration in~\Cref{fig:lem:2C4} along with the name of the vertices. Let $H$ be obtained from $G$ by identifying $a_1,b_2,a_3$ into a vertex $a^*$ and $b_4,a_5,b_6$ into one vertex $b^*$, where the direction of the arc between $a^*$ and $b^*$ will be chosen later depending on the orientations in $G$. By contracting these vertices, we do not create digons due to~\Cref{lem:3C4}. Moreover, if we create a bridge, then it is exactly $a^*b^*$ since otherwise, the same bridge would exist in $G$, a contradiction. Therefore, we distinguish two cases.

\textbf{Case 1:} $a^*b^*$ is a bridge. See~\Cref{fig:lem:2C4-1}.\\
In this case, each component $H_i$ of $H-a^*b^*$ is in $\mathcal{F}$ for $i\in\{1,2\}$ since a bridge in $H_i$ would also exist in $G$. Let $(\A_i,\I_i)$ be a \CAI-partition of $H_i$ for $i\in\{1,2\}$. Now, we have the following cases up to symmetry.
\begin{itemize}
 \item Suppose that $a^*\in\A_1$ and $b^*\in\A_2$. In this case, $(\A,\I)=((\A_1\setminus\{a^*\})\cup(\A_2\setminus\{b^*\})\cup\{a_1,b_2,a_3,b_4,b_6\},\I_1\cup\I_2\cup\{a_5\})$ is a \CAI-partition of $G$ since $\A$ is connected and any potential directed $\A$-cycle would have existed in $H_1$ or $H_2$ by going through either $a^*$ or $b^*$.
 \item Suppose that $a^*\in\A_1$ and $b^*\in\I_2$. By pigeonhole principle and w.l.o.g., there must be at most one edge $uv\in\{a_1b_6,b_2a_5,a_3b_4\}$ that is not directed from $H_1$ towards $H_2$. Say that $v$ is in $H_2$. Observe that $a'_4,a'_6\in\A_2$ since $b^*\in\I_2$. In this case, $((\A_1\setminus\{a^*\})\cup\A_2\cup\{a_1,b_2,a_3,b_4,a_5,b_6\}\setminus\{v\},\I_1\cup(\I_2\setminus\{b^*\})\cup\{v\})$ is a \CAI-partition of $G$ since $\A$ is connected. 
 \item Suppose that $a^*\in\I_1$ and $b^*\in\I_2$. Observe that there exists an $\A_1$-path between $b'_1$ and $b'_3$ and an $\A_2$-path between $a'_4$ and $a'_6$ since $a^*\in\I_1$, $b^*\in\I_2$ and $\A_1$ and $\A_2$ are connected. In this case, $(\A_1\cup\A_2\cup\{a_1,b_2,b_4,a_5\},(\I_1\setminus\{a^*\})\cup(\I_2\setminus\{b^*\})\cup\{a_3,b_6\})$ is a \CAI-partition of $G$.\\
\end{itemize}

\textbf{Case 2:} $a^*b^*$ is not a bridge. See~\Cref{fig:lem:2C4-2}.\\
In this case, $H\in\mathcal{F}$. Let \AI~be a \CAI-partition of $H$. 
\begin{itemize}
 \item Suppose that $a^*\in\A$ and $b^*\in\I$. Observe that $a'_4,a'_6\in\A$ and therefore $((\A\setminus\{a^*\})\cup\{a_1,b_2,a_3,a_5\},(\I\setminus\{b^*\})\cup\{b_4,b_6\})$ is a \CAI-partition of $G$. The same idea holds by symmetry when $a^*\in\I$ and $b^*\in\A$.

 \item Suppose that $a^*\in\A$ and $b^*\in\A$. We can assume w.l.o.g. that $\ora{a_1b_6}$ is an arc in $G$.
 \begin{itemize}
 \item Suppose that $\ora{a_3b_4}$ is an arc in $G$. In this case, we choose $\ora{a^*b^*}$ in $H$. Therefore, $(\A',\I')=((\A\setminus\{a^*,b^*\})\cup\{a_1,b_2,a_3,b_4,b_6\},\I\cup\{a_5\})$ is a \CAI-partition of $G$ since any potential directed $\A'$-cycle would have been a directed $\A$-cycle in $H$ by going through $a^*$ or $b^*$. 

 \item Suppose that $\ora{b_4a_3}$ is an arc in $G$. W.l.o.g. we assume that $\ora{a_5b_2}$ is also an arc in $G$. In this case, we choose $\ora{b^*a^*}$ in $H$.
 If there are no $\A$-paths between $a'_6$ and $b'_1$, $b'_3$, or $a'_4$ in $G-\{a_1,b_2,a_3,b_4,a_5,b_6\}$, then $((\A\setminus\{a^*,b^*\})\cup\{a_1,b_2,a_3,b_4,b_6\},\I\cup\{a_5\})$ is a \CAI-partition of $G$. 
 Otherwise, $((\A\setminus\{a^*,b^*\})\cup\{a_1,b_2,a_3,b_4,a_5\},\I\cup\{b_6\})$ is a \CAI-partition of $G$. \qedhere
 \end{itemize} 
\end{itemize} 
\end{proof}

\begin{figure}[!htbp]
\centering
\begin{subfigure}[b]{\textwidth}
\centering
\begin{minipage}[b]{0.32\textwidth}
\centering
\begin{tikzpicture}[scale=0.75]
\begin{scope}[every node/.style={circle,draw,minimum size=0.65cm,inner sep = 2}]
    \node[red,label={above:\textcolor{red}{$\A$}}] (a1) at (2,4) {$a_1$};
    \node[red,label={left:\textcolor{red}{$\A$}}] (b2) at (2,2.5) {$b_2$};
    \node[red,label={below:\textcolor{red}{$\A$}}] (a3) at (2,1) {$a_3$};
    \node[red,label={below:\textcolor{red}{$\A$}}] (b4) at (4,1) {$b_4$};
    \node[red,label={right:\textcolor{red}{$\I$}}] (a5) at (4,2.5) {$a_5$};
    \node[red,label={above:\textcolor{red}{$\A$}}] (b6) at (4,4) {$b_6$};
    
    \node (b'1) at (0,5) {$b'_1$};
    \node (b'3) at (0,0) {$b'_3$};
    \node (a'4) at (6,0) {$a'_4$};
    \node (a'6) at (6,5) {$a'_6$};

    \node[draw=none] (arrow) at (3,-0.5) {$\big\updownarrow$};

    \node[blue,label={below:\textcolor{blue}{$\A$}}] (a*) at (2,-2) {$a^*$};
    \node[blue,label={below:\textcolor{blue}{$\A$}}] (b*) at (4,-2) {$b^*$};
    \node (hb'1) at (0,-1) {$b'_1$};
    \node (hb'3) at (0,-3) {$b'_3$};
    \node (ha'4) at (6,-3) {$a'_4$};
    \node (ha'6) at (6,-1) {$a'_6$};
\end{scope}

\begin{scope}[every edge/.style={draw,minimum width = 0.04cm}]
    \path[red] (a1) edge (b2);
    \path[red] (b2) edge (a3);
    \path[red] (a3) edge (b4);
    \path[red] (b4) edge (a5);
    \path[red] (a5) edge (b6);
    \path[red] (b6) edge (a1);
    \path[red] (b2) edge (a5);
    
    \path[red] (b'1) edge (a1);
    \path[red] (a3) edge (b'3);
    \path[red] (b4) edge (a'4);
    \path[red] (b6) edge (a'6); 

    \path[blue] (hb'1) edge (a*);
    \path[blue] (hb'3) edge (a*);
    \path[blue] (ha'4) edge (b*);
    \path[blue] (ha'6) edge (b*);
\end{scope}
\end{tikzpicture}    
\end{minipage}
\begin{minipage}[b]{0.32\textwidth}
\centering
\begin{tikzpicture}[scale=0.75]
\begin{scope}[every node/.style={circle,draw,minimum size=0.65cm,inner sep = 2}]
    \node[red,label={above:\textcolor{red}{$\A$}}] (a1) at (2,4) {$a_1$};
    \node[red,label={left:\textcolor{red}{$\A$}}] (b2) at (2,2.5) {$b_2$};
    \node[red,label={below:\textcolor{red}{$\A$}}] (a3) at (2,1) {$a_3$};
    \node[red,label={below:\textcolor{red}{$\A$}}] (b4) at (4,1) {$b_4$};
    \node[red,label={right:\textcolor{red}{$\I$}}] (a5) at (4,2.5) {$a_5$};
    \node[red,label={above:\textcolor{red}{$\A$}}] (b6) at (4,4) {$b_6$};
    
    \node (b'1) at (0,5) {$b'_1$};
    \node (b'3) at (0,0) {$b'_3$};
    \node[label={left:\textcolor{blue}{$\A$}}] (a'4) at (6,0) {$a'_4$};
    \node[label={left:\textcolor{blue}{$\A$}}] (a'6) at (6,5) {$a'_6$};

    \node[draw=none] (arrow) at (3,-0.5) {$\big\updownarrow$};

    \node[blue,label={below:\textcolor{blue}{$\A$}}] (a*) at (2,-2) {$a^*$};
    \node[blue,label={below:\textcolor{blue}{$\I$}}] (b*) at (4,-2) {$b^*$};
    \node (hb'1) at (0,-1) {$b'_1$};
    \node (hb'3) at (0,-3) {$b'_3$};
    \node[label={left:\textcolor{blue}{$\A$}}] (ha'4) at (6,-3) {$a'_4$};
    \node[label={left:\textcolor{blue}{$\A$}}] (ha'6) at (6,-1) {$a'_6$};
\end{scope}

\begin{scope}[every edge/.style={draw,minimum width = 0.04cm}]
    \path[red] (a1) edge (b2);
    \path[red] (b2) edge (a3);
    \path[red] (a3) edge[->-] (b4);
    \path[red] (b4) edge (a5);
    \path[red] (a5) edge (b6);
    \path[red] (a1) edge[->-] (b6);
    \path[red] (a5) edge[->-] (b2);
    
    \path[red] (b'1) edge (a1);
    \path[red] (a3) edge (b'3);
    \path[red] (b4) edge (a'4);
    \path[red] (b6) edge (a'6); 

    \path[blue] (hb'1) edge (a*);
    \path[blue] (hb'3) edge (a*);
    \path[blue] (ha'4) edge (b*);
    \path[blue] (ha'6) edge (b*);
\end{scope}
\end{tikzpicture}    
\end{minipage}
\begin{minipage}[b]{0.32\textwidth}
\centering
\begin{tikzpicture}[scale=0.75]
\begin{scope}[every node/.style={circle,draw,minimum size=0.65cm,inner sep = 2}]
    \node[red,label={above:\textcolor{red}{$\A$}}] (a1) at (2,4) {$a_1$};
    \node[red,label={left:\textcolor{red}{$\A$}}] (b2) at (2,2.5) {$b_2$};
    \node[red,label={below:\textcolor{red}{$\I$}}] (a3) at (2,1) {$a_3$};
    \node[red,label={below:\textcolor{red}{$\A$}}] (b4) at (4,1) {$b_4$};
    \node[red,label={right:\textcolor{red}{$\A$}}] (a5) at (4,2.5) {$a_5$};
    \node[red,label={above:\textcolor{red}{$\I$}}] (b6) at (4,4) {$b_6$};
    
    \node[label={right:\textcolor{blue}{$\A$}}] (b'1) at (0,5) {$b'_1$};
    \node[label={right:\textcolor{blue}{$\A$}}] (b'3) at (0,0) {$b'_3$};
    \node[label={left:\textcolor{blue}{$\A$}}] (a'4) at (6,0) {$a'_4$};
    \node[label={left:\textcolor{blue}{$\A$}}] (a'6) at (6,5) {$a'_6$};

    \node[draw=none] (arrow) at (3,-0.5) {$\big\updownarrow$};

    \node[blue,label={below:\textcolor{blue}{$\I$}}] (a*) at (2,-2) {$a^*$};
    \node[blue,label={below:\textcolor{blue}{$\I$}}] (b*) at (4,-2) {$b^*$};
    \node[label={right:\textcolor{blue}{$\A$}}] (hb'1) at (0,-1) {$b'_1$};
    \node[label={right:\textcolor{blue}{$\A$}}] (hb'3) at (0,-3) {$b'_3$};
    \node[label={left:\textcolor{blue}{$\A$}}] (ha'4) at (6,-3) {$a'_4$};
    \node[label={left:\textcolor{blue}{$\A$}}] (ha'6) at (6,-1) {$a'_6$};
\end{scope}

\begin{scope}[every edge/.style={draw,minimum width = 0.04cm}]
    \path[red] (a1) edge (b2);
    \path[red] (b2) edge (a3);
    \path[red] (a3) edge (b4);
    \path[red] (b4) edge (a5);
    \path[red] (a5) edge (b6);
    \path[red] (a1) edge (b6);
    \path[red] (a5) edge (b2);
    
    \path[red] (b'1) edge (a1);
    \path[red] (a3) edge (b'3);
    \path[red] (b4) edge (a'4);
    \path[red] (b6) edge (a'6); 

    \path[blue] (hb'1) edge (a*);
    \path[blue] (hb'3) edge (a*);
    \path[blue] (ha'4) edge (b*);
    \path[blue] (ha'6) edge (b*);

    \path (a'6) edge[decorate] (a'4);
    \path (b'1) edge[decorate] (b'3);
    \path (ha'6) edge[decorate] (ha'4);
    \path (hb'1) edge[decorate] (hb'3);
\end{scope}
\end{tikzpicture}    
\end{minipage}
\caption{Case 1 where $v=a_5$ in the second subcase.}
\label{fig:lem:2C4-1}
\end{subfigure}

\begin{subfigure}[b]{\textwidth}
\centering
\begin{minipage}[b]{0.49\textwidth}
\centering
\begin{tikzpicture}[scale=0.75]
\begin{scope}[every node/.style={circle,draw,minimum size=0.65cm,inner sep = 2}]
    \node[red,label={above:\textcolor{red}{$\A$}}] (a1) at (2,4) {$a_1$};
    \node[red,label={left:\textcolor{red}{$\A$}}] (b2) at (2,2.5) {$b_2$};
    \node[red,label={below:\textcolor{red}{$\A$}}] (a3) at (2,1) {$a_3$};
    \node[red,label={below:\textcolor{red}{$\I$}}] (b4) at (4,1) {$b_4$};
    \node[red,label={right:\textcolor{red}{$\A$}}] (a5) at (4,2.5) {$a_5$};
    \node[red,label={above:\textcolor{red}{$\I$}}] (b6) at (4,4) {$b_6$};
    
    \node (b'1) at (0,5) {$b'_1$};
    \node (b'3) at (0,0) {$b'_3$};
    \node[label={left:\textcolor{blue}{$\A$}}] (a'4) at (6,0) {$a'_4$};
    \node[label={left:\textcolor{blue}{$\A$}}] (a'6) at (6,5) {$a'_6$};

    \node[draw=none] (arrow) at (3,-0.5) {$\big\updownarrow$};

    \node[blue,label={below:\textcolor{blue}{$\A$}}] (a*) at (2,-2) {$a^*$};
    \node[blue,label={below:\textcolor{blue}{$\I$}}] (b*) at (4,-2) {$b^*$};
    \node (hb'1) at (0,-1) {$b'_1$};
    \node (hb'3) at (0,-3) {$b'_3$};
    \node[label={left:\textcolor{blue}{$\A$}}] (ha'4) at (6,-3) {$a'_4$};
    \node[label={left:\textcolor{blue}{$\A$}}] (ha'6) at (6,-1) {$a'_6$};
\end{scope}

\begin{scope}[every edge/.style={draw,minimum width = 0.04cm}]
    \path[red] (a1) edge (b2);
    \path[red] (b2) edge (a3);
    \path[red] (a3) edge (b4);
    \path[red] (b4) edge (a5);
    \path[red] (a5) edge (b6);
    \path[red] (b6) edge (a1);
    \path[red] (b2) edge (a5);
    
    \path[red] (b'1) edge (a1);
    \path[red] (a3) edge (b'3);
    \path[red] (b4) edge (a'4);
    \path[red] (b6) edge (a'6); 

    \path[blue] (a*) edge (b*);
    \path[blue] (hb'1) edge (a*);
    \path[blue] (hb'3) edge (a*);
    \path[blue] (ha'4) edge (b*);
    \path[blue] (ha'6) edge (b*);
\end{scope}
\end{tikzpicture}    
\end{minipage}
\begin{minipage}[b]{0.49\textwidth}
\centering
\begin{tikzpicture}[scale=0.75]
\begin{scope}[every node/.style={circle,draw,minimum size=0.65cm,inner sep = 2}]
    \node[red,label={above:\textcolor{red}{$\A$}}] (a1) at (2,4) {$a_1$};
    \node[red,label={left:\textcolor{red}{$\A$}}] (b2) at (2,2.5) {$b_2$};
    \node[red,label={below:\textcolor{red}{$\A$}}] (a3) at (2,1) {$a_3$};
    \node[red,label={below:\textcolor{red}{$\A$}}] (b4) at (4,1) {$b_4$};
    \node[red,label={right:\textcolor{red}{$\I$}}] (a5) at (4,2.5) {$a_5$};
    \node[red,label={above:\textcolor{red}{$\A$}}] (b6) at (4,4) {$b_6$};
    
    \node (b'1) at (0,5) {$b'_1$};
    \node (b'3) at (0,0) {$b'_3$};
    \node (a'4) at (6,0) {$a'_4$};
    \node (a'6) at (6,5) {$a'_6$};

    \node[draw=none] (arrow) at (3,-0.5) {$\big\updownarrow$};

    \node[blue,label={below:\textcolor{blue}{$\A$}}] (a*) at (2,-2) {$a^*$};
    \node[blue,label={below:\textcolor{blue}{$\A$}}] (b*) at (4,-2) {$b^*$};
    \node (hb'1) at (0,-1) {$b'_1$};
    \node (hb'3) at (0,-3) {$b'_3$};
    \node (ha'4) at (6,-3) {$a'_4$};
    \node (ha'6) at (6,-1) {$a'_6$};
\end{scope}

\begin{scope}[every edge/.style={draw,minimum width = 0.04cm}]
    \path[red] (a1) edge (b2);
    \path[red] (b2) edge (a3);
    \path[red] (a3) edge[->-] (b4);
    \path[red] (b4) edge (a5);
    \path[red] (a5) edge (b6);
    \path[red] (a1) edge[->-] (b6);
    \path[red] (a5) edge (b2);
    
    \path[red] (b'1) edge (a1);
    \path[red] (a3) edge (b'3);
    \path[red] (b4) edge (a'4);
    \path[red] (b6) edge (a'6); 

    \path[blue] (a*) edge[->-] (b*);
    \path[blue] (hb'1) edge (a*);
    \path[blue] (hb'3) edge (a*);
    \path[blue] (ha'4) edge (b*);
    \path[blue] (ha'6) edge (b*);
\end{scope}
\end{tikzpicture}    
\end{minipage}

\begin{minipage}[b]{0.49\textwidth}
\centering
\begin{tikzpicture}[scale=0.75]
\begin{scope}[every node/.style={circle,draw,minimum size=0.65cm,inner sep = 2}]
    \node[red,label={above:\textcolor{red}{$\A$}}] (a1) at (2,4) {$a_1$};
    \node[red,label={left:\textcolor{red}{$\A$}}] (b2) at (2,2.5) {$b_2$};
    \node[red,label={below:\textcolor{red}{$\A$}}] (a3) at (2,1) {$a_3$};
    \node[red,label={below:\textcolor{red}{$\A$}}] (b4) at (4,1) {$b_4$};
    \node[red,label={right:\textcolor{red}{$\I$}}] (a5) at (4,2.5) {$a_5$};
    \node[red,label={above:\textcolor{red}{$\A$}}] (b6) at (4,4) {$b_6$};
    
    \node (b'1) at (0,5) {$b'_1$};
    \node (b'3) at (0,0) {$b'_3$};
    \node (a'4) at (6,0) {$a'_4$};
    \node (a'6) at (6,5) {$a'_6$};

    \node[draw=none] (arrow) at (3,-0.5) {$\big\updownarrow$};

    \node[blue,label={below:\textcolor{blue}{$\A$}}] (a*) at (2,-2) {$a^*$};
    \node[blue,label={below:\textcolor{blue}{$\A$}}] (b*) at (4,-2) {$b^*$};
    \node (hb'1) at (0,-1) {$b'_1$};
    \node (hb'3) at (0,-3) {$b'_3$};
    \node (ha'4) at (6,-3) {$a'_4$};
    \node (ha'6) at (6,-1) {$a'_6$};
\end{scope}

\begin{scope}[every edge/.style={draw,minimum width = 0.04cm}]
    \path[red] (a1) edge (b2);
    \path[red] (b2) edge (a3);
    \path[red] (b4) edge[->-] (a3);
    \path[red] (b4) edge (a5);
    \path[red] (a5) edge (b6);
    \path[red] (a1) edge[->-] (b6);
    \path[red] (a5) edge[->-] (b2);
    
    \path[red] (b'1) edge (a1);
    \path[red] (a3) edge (b'3);
    \path[red] (b4) edge (a'4);
    \path[red] (b6) edge (a'6); 

    \path[blue] (b*) edge[->-] (a*);
    \path[blue] (hb'1) edge (a*);
    \path[blue] (hb'3) edge (a*);
    \path[blue] (ha'4) edge (b*);
    \path[blue] (ha'6) edge (b*);
\end{scope}
\end{tikzpicture}  

There are no $\A$-paths between $a'_6$ and $b'_1$, $b'_3$, or $a'_4$ in $G-\{a_1,b_2,a_3,b_4,a_5,b_6\}$.
\end{minipage}
\begin{minipage}[b]{0.49\textwidth}
\centering
\begin{tikzpicture}[scale=0.75]
\begin{scope}[every node/.style={circle,draw,minimum size=0.65cm,inner sep = 2}]
    \node[red,label={above:\textcolor{red}{$\A$}}] (a1) at (2,4) {$a_1$};
    \node[red,label={left:\textcolor{red}{$\A$}}] (b2) at (2,2.5) {$b_2$};
    \node[red,label={below:\textcolor{red}{$\A$}}] (a3) at (2,1) {$a_3$};
    \node[red,label={below:\textcolor{red}{$\A$}}] (b4) at (4,1) {$b_4$};
    \node[red,label={right:\textcolor{red}{$\A$}}] (a5) at (4,2.5) {$a_5$};
    \node[red,label={above:\textcolor{red}{$\I$}}] (b6) at (4,4) {$b_6$};
    
    \node (b'1) at (0,5) {$b'_1$};
    \node (b'3) at (0,0) {$b'_3$};
    \node[label={left:\textcolor{blue}{$\A$}}] (a'4) at (6,0) {$a'_4$};
    \node[label={left:\textcolor{blue}{$\A$}}] (a'6) at (6,5) {$a'_6$};

    \node[draw=none] (arrow) at (3,-0.5) {$\big\updownarrow$};

    \node[blue,label={below:\textcolor{blue}{$\A$}}] (a*) at (2,-2) {$a^*$};
    \node[blue,label={below:\textcolor{blue}{$\A$}}] (b*) at (4,-2) {$b^*$};
    \node (hb'1) at (0,-1) {$b'_1$};
    \node (hb'3) at (0,-3) {$b'_3$};
    \node[label={left:\textcolor{blue}{$\A$}}] (ha'4) at (6,-3) {$a'_4$};
    \node[label={left:\textcolor{blue}{$\A$}}] (ha'6) at (6,-1) {$a'_6$};
\end{scope}

\begin{scope}[every edge/.style={draw,minimum width = 0.04cm}]
    \path[red] (a1) edge (b2);
    \path[red] (b2) edge (a3);
    \path[red] (b4) edge[->-] (a3);
    \path[red] (b4) edge (a5);
    \path[red] (a5) edge (b6);
    \path[red] (a1) edge[->-] (b6);
    \path[red] (a5) edge[->-] (b2);
    
    \path[red] (b'1) edge (a1);
    \path[red] (a3) edge (b'3);
    \path[red] (b4) edge (a'4);
    \path[red] (b6) edge (a'6); 

    \path[blue] (b*) edge[->-] (a*);
    \path[blue] (hb'1) edge (a*);
    \path[blue] (hb'3) edge (a*);
    \path[blue] (ha'4) edge (b*);
    \path[blue] (ha'6) edge (b*);
\end{scope}
\end{tikzpicture}

There is an $\A$-path between $a'_6$ and $b'_1$, $b'_3$, or $a'_4$ in $G-\{a_1,b_2,a_3,b_4,a_5,b_6\}$.
\end{minipage}
\caption{Case 2.}
\label{fig:lem:2C4-2}
\end{subfigure}
\caption{\Cref{lem:2C4}.}
\label{fig:lem:2C4}
\end{figure}

\Cref{lem:2C4} is useful to prove that if there exists a 4-cycle in $G$, then it cannot be separating.

\begin{lemma}\label{lem:separating}
There are no separating $4$-cycles in $G$.
\end{lemma}

\begin{proof}
Suppose by contradiction that $G$ contains a separating $4$-cycle $C=a_0b_1a_2b_3$. Observe that $G-\{a_0,b_1,a_2,b_3\}$ has exactly two connected components since $G$ is subcubic and $2$-vertex-connected. Let $S_1$ and $S_2$ be the set of vertices of those two connected components. Let $b'_0,a'_1,b'_2,a'_3$ be the neighbors of $a_0,b_1,a_2,b_3$ respectively. See~\Cref{fig:lem:separating}. Since $G$ is $2$-vertex-connected, exactly two of $\{b'_0,a'_1,b'_2,a'_3\}$ are in the same component. Thus, w.l.o.g. we have the two cases below. By~\Cref{lem:2C4}, there are no edges between $b'_0$ and $a'_3$, between $b'_2$ and $a'_1$, between $b'_0$ and $a'_1$, and between $a'_3$ and $b'_2$. Therefore, the graphs that will be defined below are well-defined.

\textbf{Case 1:} $b'_0,a'_1\in S_1$ and $b'_2,a'_3\in S_2$. See~\Cref{fig:lem:separating-1}.\\
W.l.o.g. we assume $\ora{a_0b_1}$ is an arc in $G$.
\begin{itemize}
 \item Suppose that we have $\ora{a_2b_3}$ in $G$. Let $H=G-\{a_0,b_1,a_2,b_3\}+\ora{b'_0a'_3}+\ora{b'_2a'_1}$. Observe that $H\in\F$ since $C$ is separating in $G$. Let $(\A,\I)$ be a \CAI-partition of $H$. Since $\{\ora{b'_0a'_3},\ora{b'_2a'_1}\}$ is an edge-cut in $H$ and since $(\A,\I)$ is a \CAI-partition of $H$, there can be at most one vertex from $\{b'_0,a'_1,b'_2,a'_3\}$ in $\I$. Therefore, we distinguish two cases.
 \begin{itemize}
 \item Suppose w.l.o.g. that $b'_0\in\I$. In this case, $(\A\cup\{a_0,b_1,a_2\},\I\cup\{b_3\})$ is a \CAI-partition of $G$.
 \item Suppose that $\{b'_0,a'_1,b'_2,a'_3\}\subseteq\A$.
 Since $\A$ is connected, there must be an $\A$-path between $b'_0$ and $a'_1$ or between $a'_3$ and $b'_2$.
 Since neither $(\A\cup\{b_1,a_2,b_3\},\I\cup\{a_0\})$ nor $(\A\cup\{a_0,b_1,b_3\},\I\cup\{a_2\})$ are decompositions of $G$ and $C$ is a separating cycle of $G$, there must be a directed $\A$-path $\ora{P_2}$ from $a'_3$ to $b'_2$ in $S_2$ and a directed $\A$-path $\ora{P_1}$ from $a'_1$ to $b'_0$ in $S_1$. However, this is impossible because $\ora{P_1}\ora{b'_0a'_3}\ora{P_2}\ora{b'_2a'_1}$ is then a directed $\A$-cycle in $H$.
 \end{itemize}

 \item Suppose that we have $\ora{b_3a_2}$ in $G$. Let $H_1=G[S_1]+\ora{b'_0a'_1}$ and $H_2=G[S_2]+\ora{a'_3b'_2}$ be the two connected components of $G-\{a_0,b_1,a_2,b_3\}+\{\ora{b'_0a'_1}, \ora{a'_3b'_2}\}$. Observe that $H_1$ and $H_2$ are in $\F$. Let $(\A_i,\I_i)$ be a \CAI-partition of $H_i$, for $i\in\{1,2\}$. We claim that $(\A,\I)=(\A_1\cup\A_2\cup\{a_0,b_1,a_2,b_3\},\I_1\cup\I_2)$ is a \CAI-partition of $G$. Indeed, $C$ is not a directed cycle, $\A$ is connected, and any potential directed $\A$-cycle in $G$, would lead to a directed $\A_1$-cycle (resp. $\A_2$-cycle) in $H_1$ (resp. $H_2$) passing through the arc $\ora{b'_0a'_1}$ (resp. $\ora{a'_3b'_2}$).\\
\end{itemize}

\textbf{Case 2:} $b'_0,b'_2\in S_1$ and $a'_1,a'_3\in S_2$. See~\Cref{fig:lem:separating-2}.\\ 
Let $H=G-\{a_0,b_1,a_2,b_3\}+\ora{a'_3b'_0}+\ora{b'_2a'_1}$. Observe that $H\in\F$ since $C$ is separating in $G$. Let $(\A,\I)$ be a \CAI-partition of $H$. Since $\{\ora{a'_3b'_0},\ora{b'_2a'_1}\}$ is a cut in $H$ and $(\A,\I)$ is a \CAI-partition of $H$, there can be at most one vertex from $\{b'_0,a'_1,b'_2,a'_3\}$ in $\I$. Therefore, we distinguish two cases.
\begin{itemize}
 \item Suppose w.l.o.g. that $b'_0\in\I$. In this case, $(\A\cup\{a_0,b_1,a_2\},\I\cup\{b_3\})$ is a \CAI-partition of $G$.
 \item Suppose that $\{b'_0,a'_1,b'_2,a'_3\}\subseteq\A$. Since $\A$ is connected, suppose w.l.o.g. that there exists an $\A$-path between $b'_0$ and $b'_2$. Since $(\A\cup\{b_1,a_2,b_3\},\I\cup\{a_0\})$ and $(\A\cup\{a_0,b_1,b_3\},\I\cup\{a_2\})$ are not decompositions of $G$ and $C$ is a separating cycle of $G$, there must by a directed cycle in $(\A\cap S_2)\cup\{b_1,a_2,b_3\}$ and $(\A\cap S_2)\cup\{b_1,a_0,b_3\}$. Hence $b_1$ is either a source or a sink in the cycle $C$. Therefore, $(A\cup\{a_0,b_1,a_2\},\I\cup\{b_3\})$ is a \CAI-partition of $G$. \qedhere
\end{itemize}
\end{proof}

\begin{figure}[!htbp]
\centering
\begin{subfigure}[b]{\textwidth}
\centering
\begin{minipage}[b]{0.19\textwidth}
\centering
\begin{tikzpicture}[scale=0.67]
\begin{scope}[every node/.style={circle,draw,minimum size=0.65cm,inner sep = 2}]
    \node[red,label={below:\textcolor{red}{$\A$}}] (a0) at (1,1) {$a_0$};
    \node[red,label={below:\textcolor{red}{$\A$}}] (b1) at (3,1) {$b_1$};
    \node[red,label={above:\textcolor{red}{$\A$}}] (a2) at (3,2.5) {$a_2$};
    \node[red,label={above:\textcolor{red}{$\I$}}] (b3) at (1,2.5) {$b_3$};
    
    \node[label={below:\textcolor{blue}{$\I$}}] (b'0) at (0,0) {$b'_0$};
    \node[label={below:\textcolor{blue}{$\A$}}] (a'1) at (4,0) {$a'_1$};
    \node[label={above:\textcolor{blue}{$\A$}}] (b'2) at (4,3.5) {$b'_2$};
    \node[label={above:\textcolor{blue}{$\A$}}] (a'3) at (0,3.5) {$a'_3$};
\end{scope}

\begin{scope}[every edge/.style={draw,minimum width = 0.04cm}]
    \path[red] (a0) edge[->-] (b1);
    \path[red] (b1) edge (a2);
    \path[red] (a2) edge[->-] (b3);
    \path[red] (b3) edge (a0);
    
    \path[red] (a0) edge (b'0);
    \path[red] (b1) edge (a'1);
    \path[red] (b'2) edge (a2);
    \path[red] (b3) edge (a'3);

    \path[blue] (b'2) edge[->-] (a'1);
    \path[blue] (b'0) edge[->-] (a'3);

\end{scope}
\end{tikzpicture}

~\\
~\\
~
\end{minipage}
\begin{minipage}[b]{0.19\textwidth}
\centering
\begin{tikzpicture}[scale=0.67]
\begin{scope}[every node/.style={circle,draw,minimum size=0.65cm,inner sep = 2}]
    \node[red,label={below:\textcolor{red}{$\I$}}] (a0) at (1,1) {$a_0$};
    \node[red,label={below:\textcolor{red}{$\A$}}] (b1) at (3,1) {$b_1$};
    \node[red,label={above:\textcolor{red}{$\A$}}] (a2) at (3,2.5) {$a_2$};
    \node[red,label={above:\textcolor{red}{$\A$}}] (b3) at (1,2.5) {$b_3$};
    
    \node[label={below:\textcolor{blue}{$\A$}}] (b'0) at (0,0) {$b'_0$};
    \node[label={below:\textcolor{blue}{$\A$}}] (a'1) at (4,0) {$a'_1$};
    \node[label={above:\textcolor{blue}{$\A$}}] (b'2) at (4,3.5) {$b'_2$};
    \node[label={above:\textcolor{blue}{$\A$}}] (a'3) at (0,3.5) {$a'_3$};
\end{scope}

\begin{scope}[every edge/.style={draw,minimum width = 0.04cm}]
    \path[red] (a0) edge[->-] (b1);
    \path[red] (b1) edge (a2);
    \path[red] (a2) edge[->-] (b3);
    \path[red] (b3) edge (a0);
    
    \path[red] (a0) edge (b'0);
    \path[red] (b1) edge (a'1);
    \path[red] (b'2) edge (a2);
    \path[red] (b3) edge (a'3);

    \path[blue] (b'2) edge[->-] (a'1);
    \path[blue] (b'0) edge[->-] (a'3);

    \path (a'1) edge[bend left,decorate] (b'0);
\end{scope}
\end{tikzpicture}

No directed $\A$-path from $a'_3$ to $b'_2$ in $G-C$.
\end{minipage}
\begin{minipage}[b]{0.19\textwidth}
\centering
\begin{tikzpicture}[scale=0.67]
\begin{scope}[every node/.style={circle,draw,minimum size=0.65cm,inner sep = 2}]
    \node[red,label={below:\textcolor{red}{$\A$}}] (a0) at (1,1) {$a_0$};
    \node[red,label={below:\textcolor{red}{$\A$}}] (b1) at (3,1) {$b_1$};
    \node[red,label={above:\textcolor{red}{$\I$}}] (a2) at (3,2.5) {$a_2$};
    \node[red,label={above:\textcolor{red}{$\A$}}] (b3) at (1,2.5) {$b_3$};
    
    \node[label={below:\textcolor{blue}{$\A$}}] (b'0) at (0,0) {$b'_0$};
    \node[label={below:\textcolor{blue}{$\A$}}] (a'1) at (4,0) {$a'_1$};
    \node[label={above:\textcolor{blue}{$\A$}}] (b'2) at (4,3.5) {$b'_2$};
    \node[label={above:\textcolor{blue}{$\A$}}] (a'3) at (0,3.5) {$a'_3$};
\end{scope}

\begin{scope}[every edge/.style={draw,minimum width = 0.04cm}]
    \path[red] (a0) edge[->-] (b1);
    \path[red] (b1) edge (a2);
    \path[red] (a2) edge[->-] (b3);
    \path[red] (b3) edge (a0);
    
    \path[red] (a0) edge (b'0);
    \path[red] (b1) edge (a'1);
    \path[red] (b'2) edge (a2);
    \path[red] (b3) edge (a'3);

    \path[blue] (b'2) edge[->-] (a'1);
    \path[blue] (b'0) edge[->-] (a'3);

    \path (a'3) edge[bend left,decorate] (b'2);
\end{scope}
\end{tikzpicture}

No directed $\A$-path from $a'_1$ to $b'_0$ in $G-C$.
\end{minipage}
\begin{minipage}[b]{0.19\textwidth}
\centering
\begin{tikzpicture}[scale=0.67]
\begin{scope}[every node/.style={circle,draw,minimum size=0.65cm,inner sep = 2}]
    \node[red] (a0) at (1,1) {$a_0$};
    \node[red] (b1) at (3,1) {$b_1$};
    \node[red] (a2) at (3,2.5) {$a_2$};
    \node[red] (b3) at (1,2.5) {$b_3$};
    
    \node[label={below:\textcolor{blue}{$\A$}}] (b'0) at (0,0) {$b'_0$};
    \node[label={below:\textcolor{blue}{$\A$}}] (a'1) at (4,0) {$a'_1$};
    \node[label={above:\textcolor{blue}{$\A$}}] (b'2) at (4,3.5) {$b'_2$};
    \node[label={above:\textcolor{blue}{$\A$}}] (a'3) at (0,3.5) {$a'_3$};
\end{scope}

\begin{scope}[every edge/.style={draw,minimum width = 0.04cm}]
    \path[red] (a0) edge[->-] (b1);
    \path[red] (b1) edge (a2);
    \path[red] (a2) edge[->-] (b3);
    \path[red] (b3) edge (a0);
    
    \path[red] (a0) edge (b'0);
    \path[red] (b1) edge (a'1);
    \path[red] (b'2) edge (a2);
    \path[red] (b3) edge (a'3);

    \path[blue] (b'2) edge[->-] (a'1);
    \path[blue] (b'0) edge[->-] (a'3);

    \path[->] (a'1) edge[bend left,decorate] (b'0);
    \path[->] (a'3) edge[bend left,decorate] (b'2);
\end{scope}
\end{tikzpicture}

Impossible.\\
~\\
~
\end{minipage}
\begin{minipage}[b]{0.19\textwidth}
\centering
\begin{tikzpicture}[scale=0.67]
\begin{scope}[every node/.style={circle,draw,minimum size=0.65cm,inner sep = 2}]
    \node[red,label={below:\textcolor{red}{$\A$}}] (a0) at (1,1) {$a_0$};
    \node[red,label={below:\textcolor{red}{$\A$}}] (b1) at (3,1) {$b_1$};
    \node[red,label={above:\textcolor{red}{$\A$}}] (a2) at (3,2.5) {$a_2$};
    \node[red,label={above:\textcolor{red}{$\A$}}] (b3) at (1,2.5) {$b_3$};
    
    \node (b'0) at (0,0) {$b'_0$};
    \node (a'1) at (4,0) {$a'_1$};
    \node (b'2) at (4,3.5) {$b'_2$};
    \node (a'3) at (0,3.5) {$a'_3$};
\end{scope}

\begin{scope}[every edge/.style={draw,minimum width = 0.04cm}]
    \path[red] (a0) edge[->-] (b1);
    \path[red] (b1) edge (a2);
    \path[red] (b3) edge[->-] (a2);
    \path[red] (b3) edge (a0);
    
    \path[red] (a0) edge (b'0);
    \path[red] (b1) edge (a'1);
    \path[red] (a2) edge (b'2);
    \path[red] (b3) edge (a'3);

    \path[blue] (a'3) edge[->-,bend left] (b'2);
    \path[blue] (b'0) edge[->-,bend right] (a'1);
\end{scope}
\end{tikzpicture}

~\\
~\\
~\\
~
\end{minipage}
\caption{\label{fig:lem:separating-1}Case 1.}    
\end{subfigure}

\begin{subfigure}[b]{0.99\textwidth}
\centering
\begin{minipage}[b]{0.24\textwidth}
\centering
\begin{tikzpicture}[scale=0.6]
\begin{scope}[every node/.style={circle,draw,minimum size=0.65cm,inner sep = 2}]
    \node[red,label={below:\textcolor{red}{$\A$}}] (a0) at (0,0) {$a_0$};
    \node[red,label={below:\textcolor{red}{$\A$}}] (b1) at (4.5,0) {$b_1$};
    \node[red,label={above:\textcolor{red}{$\A$}}] (a2) at (4.5,4.5) {$a_2$};
    \node[red,label={above:\textcolor{red}{$\I$}}] (b3) at (0,4.5) {$b_3$};
    
    \node[label={below:\textcolor{blue}{$\I$}}] (b'0) at (1.5,1.5) {$b'_0$};
    \node[label={left:\textcolor{blue}{$\A$}}] (a'1) at (3,-1.5) {$a'_1$};
    \node[label={above:\textcolor{blue}{$\A$}}] (b'2) at (3,3) {$b'_2$};
    \node[label={right:\textcolor{blue}{$\A$}}] (a'3) at (1.5,6) {$a'_3$};
\end{scope}

\begin{scope}[every edge/.style={draw,minimum width = 0.04cm}]
    \path[red] (a0) edge (b1);
    \path[red] (b1) edge (a2);
    \path[red] (b3) edge (a2);
    \path[red] (b3) edge (a0);
    
    \path[red] (a0) edge (b'0);
    \path[red] (b1) edge (a'1);
    \path[red] (a2) edge (b'2);
    \path[red] (b3) edge (a'3);

    \path[blue] (a'3) edge[->-] (b'0);
    \path[blue] (b'2) edge[->-] (a'1);
\end{scope}
\end{tikzpicture}

~\\
~
\end{minipage}
\begin{minipage}[b]{0.24\textwidth}
\centering
\begin{tikzpicture}[scale=0.6]
\begin{scope}[every node/.style={circle,draw,minimum size=0.65cm,inner sep = 2}]
    \node[red,label={below:\textcolor{red}{$\I$}}] (a0) at (0,0) {$a_0$};
    \node[red,label={below:\textcolor{red}{$\A$}}] (b1) at (4.5,0) {$b_1$};
    \node[red,label={above:\textcolor{red}{$\A$}}] (a2) at (4.5,4.5) {$a_2$};
    \node[red,label={above:\textcolor{red}{$\A$}}] (b3) at (0,4.5) {$b_3$};
    
    \node[label={below:\textcolor{blue}{$\A$}}] (b'0) at (1.5,1.5) {$b'_0$};
    \node[label={left:\textcolor{blue}{$\A$}}] (a'1) at (3,-1.5) {$a'_1$};
    \node[label={above:\textcolor{blue}{$\A$}}] (b'2) at (3,3) {$b'_2$};
    \node[label={right:\textcolor{blue}{$\A$}}] (a'3) at (1.5,6) {$a'_3$};
\end{scope}

\begin{scope}[every edge/.style={draw,minimum width = 0.04cm}]
    \path[red] (a0) edge (b1);
    \path[red] (b1) edge (a2);
    \path[red] (b3) edge (a2);
    \path[red] (b3) edge (a0);
    
    \path[red] (a0) edge (b'0);
    \path[red] (b1) edge (a'1);
    \path[red] (a2) edge (b'2);
    \path[red] (b3) edge (a'3);

    \path[blue] (a'3) edge[->-] (b'0);
    \path[blue] (b'2) edge[->-] (a'1);

    \path (b'0) edge[decorate] (b'2);
\end{scope}
\end{tikzpicture}

No directed cycle in $(\A\cap S_2)\cup\{b_1,a_2,b_3\}$.
\end{minipage}
\begin{minipage}[b]{0.24\textwidth}
\centering
\begin{tikzpicture}[scale=0.6]
\begin{scope}[every node/.style={circle,draw,minimum size=0.65cm,inner sep = 2}]
    \node[red,label={below:\textcolor{red}{$\A$}}] (a0) at (0,0) {$a_0$};
    \node[red,label={below:\textcolor{red}{$\A$}}] (b1) at (4.5,0) {$b_1$};
    \node[red,label={above:\textcolor{red}{$\I$}}] (a2) at (4.5,4.5) {$a_2$};
    \node[red,label={above:\textcolor{red}{$\A$}}] (b3) at (0,4.5) {$b_3$};
    
    \node[label={below:\textcolor{blue}{$\A$}}] (b'0) at (1.5,1.5) {$b'_0$};
    \node[label={left:\textcolor{blue}{$\A$}}] (a'1) at (3,-1.5) {$a'_1$};
    \node[label={above:\textcolor{blue}{$\A$}}] (b'2) at (3,3) {$b'_2$};
    \node[label={right:\textcolor{blue}{$\A$}}] (a'3) at (1.5,6) {$a'_3$};
\end{scope}

\begin{scope}[every edge/.style={draw,minimum width = 0.04cm}]
    \path[red] (a0) edge (b1);
    \path[red] (b1) edge (a2);
    \path[red] (b3) edge (a2);
    \path[red] (b3) edge (a0);
    
    \path[red] (a0) edge (b'0);
    \path[red] (b1) edge (a'1);
    \path[red] (a2) edge (b'2);
    \path[red] (b3) edge (a'3);

    \path[blue] (a'3) edge[->-] (b'0);
    \path[blue] (b'2) edge[->-] (a'1);

    \path (b'0) edge[decorate] (b'2);
\end{scope}
\end{tikzpicture}

No directed cycle in $(\A\cap S_2)\cup\{a_0,b_1,b_3\}$.
\end{minipage}
\begin{minipage}[b]{0.24\textwidth}
\centering
\begin{tikzpicture}[scale=0.6]
\begin{scope}[every node/.style={circle,draw,minimum size=0.65cm,inner sep = 2}]
    \node[red,label={below:\textcolor{red}{$\A$}}] (a0) at (0,0) {$a_0$};
    \node[red,label={below:\textcolor{red}{$\A$}}] (b1) at (4.5,0) {$b_1$};
    \node[red,label={above:\textcolor{red}{$\A$}}] (a2) at (4.5,4.5) {$a_2$};
    \node[red,label={above:\textcolor{red}{$\I$}}] (b3) at (0,4.5) {$b_3$};
    
    \node[label={below:\textcolor{blue}{$\A$}}] (b'0) at (1.5,1.5) {$b'_0$};
    \node[label={left:\textcolor{blue}{$\A$}}] (a'1) at (3,-1.5) {$a'_1$};
    \node[label={above:\textcolor{blue}{$\A$}}] (b'2) at (3,3) {$b'_2$};
    \node[label={right:\textcolor{blue}{$\A$}}] (a'3) at (1.5,6) {$a'_3$};
\end{scope}

\begin{scope}[every edge/.style={draw,minimum width = 0.04cm}]
    \path[red] (a0) edge[->-] (b1);
    \path[red] (a2) edge[->-] (b1);
    \path[red] (b3) edge[->-] (a2);
    \path[red] (b3) edge[->-] (a0);
    
    \path[red] (a0) edge (b'0);
    \path[red] (b1) edge[->-] (a'1);
    \path[red] (a2) edge (b'2);
    \path[red] (a'3) edge[->-] (b3);

    \path[blue] (a'3) edge[->-] (b'0);
    \path[blue] (b'2) edge[->-] (a'1);

    \path (b'0) edge[decorate] (b'2);
    \path (a'1) edge[decorate, out=-20, in=-90] (6,2.25);
    \path[->] (6,2.25) edge[decorate, out = 70, in=40] (a'3);
\end{scope}
\end{tikzpicture}

Example with $b_1$ being a sink in $C$.
\end{minipage}
\caption{\label{fig:lem:separating-2}Case 2.}
\end{subfigure}
\caption{\Cref{lem:separating}.}
\label{fig:lem:separating}
\end{figure}

Finally, we prove a stronger result than~\Cref{prop:reducible}\ref{itm:no-4-faces}.
\begin{lemma}\label{lem:C4}
There are no $4$-cycles in $G$.
\end{lemma}

\begin{proof}
Suppose by contradiction that $G$ contains a $4$-cycle $C=a_0b_1a_2b_3$, which by \Cref{lem:separating} must be a $4$-face. Let $b'_0,a'_1,b'_2,a'_3$ be the neighbors of $a_0,b_1,a_2,b_3$ respectively. By~\Cref{lem:2C4}, there are no edges between $b'_0$ and $a'_3$, between $b'_2$ and $a'_1$, between $b'_0$ and $a'_1$, and between $a'_3$ and $b'_2$. Therefore, the graphs that will be defined below are well-defined. 
We begin by showing a useful claim.

\begin{claim}\label{claim:$2$-vertex-connected}
 The underlying undirected graph $G-C+b'_0a'_3+a'_1b'_2$ or $G-C+b'_0a'_1+a'_2b'_3$ is $2$-vertex-connected.
 By symmetry, we can assume that $G-C+b'_0a'_3+a'_1b'_2$ is $2$-vertex-connected.
\end{claim}

\begin{proof}
By contradiction, $G$ would contain two edge-cuts of size 3, say $\{a_0b_3,b_1a_2,w_1w_2\}$ and $\{b_3a_2,a_0b_1,u_0u_1\}$, where $u_0,u_1,w_0,w_1$ are some vertices of $G$ (see \Cref{fig:C4-cuts}).
W.l.o.g. suppose that $\{a_0b_3,b_1a_2,w_1w_2\}$ separates $G$ into two components with vertex sets $S_1\supseteq\{a_0,b_1,w_1,u_0,u_1\}$, $S_2\supseteq\{a_2,b_3,w_2\}$ and that $\{b_1a_2,a_0b_3,u_0u_1\}$ separates $G$ into two components with vertex sets $T_1\supseteq\{a_0,b_3,u_0\}$, $T_2\supseteq\{b_1,a_2,w_1,w_2,u_1\}$. In this case, $b_3$ is a cut vertex in $G$ ($a'_3\in T_1 \cap S_2 \setminus \{b_3\}$), which contradicts the 2-connectivity of $G$. 
\end{proof}

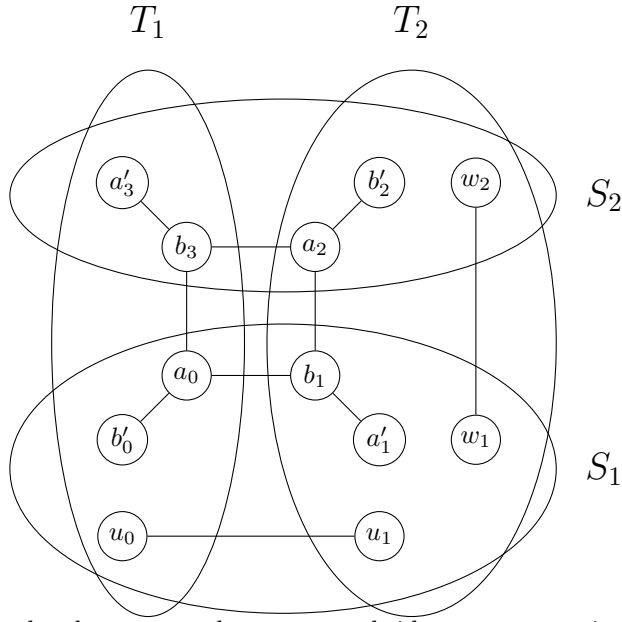
\begin{figure}[!htbp]
\centering
\begin{tikzpicture}[scale=0.85]
\begin{scope}[every node/.style={circle,draw,minimum size=0.65cm,inner sep = 2}]
    \node (a0) at (1,1) {$a_0$};
    \node (b1) at (3,1) {$b_1$};
    \node (a2) at (3,3) {$a_2$};
    \node (b3) at (1,3) {$b_3$};
    
    \node (b'0) at (0,0) {$b'_0$};
    \node (a'1) at (4,0) {$a'_1$};
    \node (b'2) at (4,4) {$b'_2$};
    \node (a'3) at (0,4) {$a'_3$};

    \node (w2) at (5.5,4) {$w_2$};
    \node (w1) at (5.5,0) {$w_1$};
    \node (u0) at (0,-1.5) {$u_0$};
    \node (u1) at (4,-1.5) {$u_1$};

    \node[draw=none] at (7.5,-0.45) {\Large $S_1$};
    \node[draw=none] at (7.5,3.8) {\Large $S_2$};
    \node[draw=none] at (0.4,6.5) {\Large $T_1$};
    \node[draw=none] at (4.5,6.5) {\Large $T_2$};
\end{scope}

\begin{scope}[every edge/.style={draw,minimum width = 0.04cm}]
    \path (a0) edge (b1);
    \path (b1) edge (a2);
    \path (a2) edge (b3);
    \path (b3) edge (a0);
    
    \path (a0) edge (b'0);
    \path (b1) edge (a'1);
    \path (b'2) edge (a2);
    \path (b3) edge (a'3);

    \path (u0) edge (u1);
    \path (w1) edge (w2);
\end{scope}

\draw (2.5,-0.45) ellipse (4.25cm and 2.25cm);
\draw (2.5,3.8) ellipse (4.25cm and 1.5cm);
\draw (4.5,1.5) ellipse (2.25cm and 4.25cm);
\draw (0.4,1.5) ellipse (1.5cm and 4.25cm);
\end{tikzpicture}
\caption{A 4-cycle whose removal creates two bridges must contain a cut-vertex ($b_3$).}
\label{fig:C4-cuts}
\end{figure}

Now, we proceed to the proof of \Cref{lem:C4}.

\textbf{Case 1:} Suppose that $G$ contains the following arcs $\ora{a'_3b_3}$, $\ora{b_3a_2}$, $\ora{a_2b'_2}$, $\ora{b'_0a_0}$, $\ora{a_0b_1}$, $\ora{b_1a'_1}$ and that $G-C$ is $2$-vertex-connected. Let $H=G-C+\ora{a'_3b'_2}+\ora{b'_0a'_1}$. See~\Cref{fig:lem:C4-1}.

By assumption, we have $H\in\F$. Let \AI~be a \CAI-partition of $H$. Observe that $|\{b'_0,a'_1,b'_2,a'_3\}\cap\A| \geq 2$ since $\I$ is an independent set in $H$. Thus, we have the following two cases.
\begin{itemize}
 \item Suppose $|\{b'_0,a'_1,b'_2,a'_3\}\cap\A|\leq 3$. W.l.o.g. we can assume that $b'_0\in\I$ and therefore $a'_1\in\A$. We claim that $(\A',\I')=(\A\cup\{a_0,a_2,b_3\},\I\cup\{b_1\})$ is a \CAI-partition of $G$. Indeed, $\A'$ is connected and any possibly directed $\A'$-cycle in $G$ would contain $\ora{b_3a_2}$, but then $H$ would contain a directed $\A$-cycle containing $\ora{a'_3b'_2}$.
 \item Suppose $|\{b'_0,a'_1,b'_2,a'_3\}\cap\A|=4$. Since $\A$ is connected, by symmetry, in $G-C$ there exists an $\A$-path from $b'_0$ to $b'_2$ or $a'_3$. Let $(\A',\I')=(\A\cup\{b_1,a_2,b_3\},\I\cup\{a_0\})$. If $(\A',\I')$ is a \CAI-partition of $G$, then we are done. 
 Otherwise, $G$ necessarily contains a directed $\A'$-cycle which consists of the arcs $\ora{a'_3b_3}$, $\ora{b_3a_2}$, $\ora{a_2b_1}$, $\ora{b_1a'_1}$ together with a directed $\A$-path from $a'_1$ to $a'_3$. In this case $(\A\cup\{a_0,b_1,a_2\},\I\cup\{b_3\})$ is a \CAI-partition of $G$.
\end{itemize}

\begin{figure}[!htbp]
\centering
\begin{minipage}[b]{0.32\textwidth}
\centering
\begin{tikzpicture}[scale=0.75]
\begin{scope}[every node/.style={circle,draw,minimum size=0.65cm,inner sep = 2}]
    \node[red,label={below:\textcolor{red}{$\A$}}] (a0) at (1,1) {$a_0$};
    \node[red,label={below:\textcolor{red}{$\I$}}] (b1) at (3,1) {$b_1$};
    \node[red,label={above:\textcolor{red}{$\A$}}] (a2) at (3,2.5) {$a_2$};
    \node[red,label={above:\textcolor{red}{$\A$}}] (b3) at (1,2.5) {$b_3$};
    
    \node[label={below:\textcolor{blue}{$\I$}}] (b'0) at (0,0) {$b'_0$};
    \node[label={below:\textcolor{blue}{$\A$}}] (a'1) at (4,0) {$a'_1$};
    \node[label={above:\textcolor{blue}{$\A$}}] (b'2) at (4,3.5) {$b'_2$};
    \node[label={above:\textcolor{blue}{$\A$}}] (a'3) at (0,3.5) {$a'_3$};
\end{scope}

\begin{scope}[every edge/.style={draw,minimum width = 0.04cm}]
    \path[red] (a0) edge[->-] (b1);
    \path[red] (b1) edge (a2);
    \path[red] (b3) edge[->-] (a2);
    \path[red] (b3) edge (a0);
    
    \path[red] (b'0) edge[->-] (a0);
    \path[red] (b1) edge[->-] (a'1);
    \path[red] (a2) edge[->-] (b'2);
    \path[red] (a'3) edge[->-] (b3);

    \path[blue] (a'3) edge[->-, bend left] (b'2);
    \path[blue] (b'0) edge[->-, bend right] (a'1);

\end{scope}
\end{tikzpicture}

~\\
~
\end{minipage}
\begin{minipage}[b]{0.32\textwidth}
\centering
\begin{tikzpicture}[scale=0.75]
\begin{scope}[every node/.style={circle,draw,minimum size=0.65cm,inner sep = 2}]
    \node[red,label={below:\textcolor{red}{$\I$}}] (a0) at (1,1) {$a_0$};
    \node[red,label={below:\textcolor{red}{$\A$}}] (b1) at (3,1) {$b_1$};
    \node[red,label={above:\textcolor{red}{$\A$}}] (a2) at (3,2.5) {$a_2$};
    \node[red,label={above:\textcolor{red}{$\A$}}] (b3) at (1,2.5) {$b_3$};
    
    \node[label={below:\textcolor{blue}{$\A$}}] (b'0) at (0,0) {$b'_0$};
    \node[label={below:\textcolor{blue}{$\A$}}] (a'1) at (4,0) {$a'_1$};
    \node[label={above:\textcolor{blue}{$\A$}}] (b'2) at (4,3.5) {$b'_2$};
    \node[label={above:\textcolor{blue}{$\A$}}] (a'3) at (0,3.5) {$a'_3$};
\end{scope}

\begin{scope}[every edge/.style={draw,minimum width = 0.04cm}]
    \path[red] (a0) edge[->-] (b1);
    \path[red] (b1) edge (a2);
    \path[red] (b3) edge[->-] (a2);
    \path[red] (b3) edge (a0);
    
    \path[red] (b'0) edge[->-] (a0);
    \path[red] (b1) edge[->-] (a'1);
    \path[red] (a2) edge[->-] (b'2);
    \path[red] (a'3) edge[->-] (b3);

    \path[blue] (a'3) edge[->-, bend left] (b'2);
    \path[blue] (b'0) edge[->-, bend right] (a'1);

    \path (b'0) edge[decorate] (a'3);
\end{scope}
\end{tikzpicture}

No directed $\A$-path from $a'_1$ to $a'_3$ in $G-C$.
\end{minipage}
\begin{minipage}[b]{0.32\textwidth}
\centering
\begin{tikzpicture}[scale=0.75]
\begin{scope}[every node/.style={circle,draw,minimum size=0.65cm,inner sep = 2}]
    \node[red,label={below:\textcolor{red}{$\A$}}] (a0) at (1,1) {$a_0$};
    \node[red,label={below:\textcolor{red}{$\A$}}] (b1) at (3,1) {$b_1$};
    \node[red,label={above:\textcolor{red}{$\A$}}] (a2) at (3,2.5) {$a_2$};
    \node[red,label={above:\textcolor{red}{$\I$}}] (b3) at (1,2.5) {$b_3$};
    
    \node[label={below:\textcolor{blue}{$\A$}}] (b'0) at (0,0) {$b'_0$};
    \node[label={below:\textcolor{blue}{$\A$}}] (a'1) at (4,0) {$a'_1$};
    \node[label={above:\textcolor{blue}{$\A$}}] (b'2) at (4,3.5) {$b'_2$};
    \node[label={above:\textcolor{blue}{$\A$}}] (a'3) at (0,3.5) {$a'_3$};
\end{scope}

\begin{scope}[every edge/.style={draw,minimum width = 0.04cm}]
    \path[red] (a0) edge[->-] (b1);
    \path[red] (a2) edge[->-] (b1);
    \path[red] (b3) edge[->-] (a2);
    \path[red] (b3) edge (a0);
    
    \path[red] (b'0) edge[->-] (a0);
    \path[red] (b1) edge[->-] (a'1);
    \path[red] (a2) edge[->-] (b'2);
    \path[red] (a'3) edge[->-] (b3);

    \path[blue] (a'3) edge[->-, bend left] (b'2);
    \path[blue] (b'0) edge[->-, bend right] (a'1);

    \path (b'0) edge[decorate] (a'3);
    \path (a'1) edge[decorate, bend right] (4.5,4);
    \path[->] (4.5,4) edge[decorate, bend right=50] (a'3);
\end{scope}
\end{tikzpicture}

~\\
~
\end{minipage}
\caption{Case 1 of~\Cref{lem:C4}.}
\label{fig:lem:C4-1}
\end{figure}

\textbf{Case 2:} Suppose that $G$ contains the following arcs $\ora{a'_3b_3}$, $\ora{b_3a_2}$, $\ora{a_2b'_2}$, $\ora{b'_0a_0}$, $\ora{a_0b_1}$, $\ora{b_1a'_1}$ and there exists a bridge in $G-C$. Together with \Cref{claim:$2$-vertex-connected}, we conclude that there exists an edge $e$ in $G$ such that $\{a_0b_3,b_1a_2,e\}$ is a 3-edge-cut of $G$. 
Let $H_1$ and $H_2$ be the two connected subgraphs of $G-C-e$ with $H_1$ containing $a'_0$ and $b'_1$ and $H_2$ containing $b'_2$ and $a'_3$. 
Let $H'_1$ be obtained by reversing every arc of $H_1$ and $H=H'_1+H_2+e+\ora{a'_3b'_0}+\ora{a'_1b'_2}$.
See~\Cref{fig:lem:C4-2}.\\
Observe that $H\in\F$ is smaller than $G$, so we have a \CAI-partition of $H$. Observe that $|\{b'_0,a'_1,b'_2,a'_3\}\cap\A|\geq 2$ since $\I$ is an independent set in $H$. Thus, we have the following two cases.
\begin{itemize}
 \item Suppose $|\{b'_0,a'_1,b'_2,a'_3\}\cap\A|\leq 3$. W.l.o.g. we can assume that $b'_0\in\I$ and therefore $a'_3\in\A$. Then $(\A\cup\{a_0,b_1,a_2\},\I\cup\{b_3\})$ is a \CAI-partition of $G$.
 \item Suppose $|\{b'_0,a'_1,b'_2,a'_3\}\cap\A|=4$. We claim that there cannot be simultaneously a directed $\A$-path from $b'_2$ to $a'_3$ in $H_2$ and a directed $\A$-path from $a'_1$ to $b'_0$ in $H_1$. Otherwise, there would be a directed $\A$-cycle in $H$ consisting of the following: a directed $\A$-path from $b'_2$ to $a'_3$, $\ora{a'_3b'_0}$, a directed $\A$-path from $b'_0$ to $a'_1$ (because $H'_1$ has all arcs reversed with respect to $H_1$). Therefore, we have the two following cases.
 \begin{itemize}
 \item There is a directed $\A$-path from $b'_2$ to $a'_3$ in $H_2$ and no directed $\A$-paths from $a'_1$ to $b'_0$ in $H_1$. If $(\A',\I')=(\A\cup\{a_0,b_1,b_3\},\I\cup\{a_2\})$ is a \CAI-partition of $G$, then we are done. Otherwise, there must be a directed $\A'$-cycle going through $a'_3$, $b_3$, $a_0$, $b_1$, $a'_1$, and the edge $e$ oriented from $H_1$ towards $H_2$. However, in this case, $(\A\cup\{a_0,b_1,a_2\},\I\cup\{b_3\})$ is a \CAI-partition of $G$.

 The same arguments give a \CAI-partition of $G$ when there is a directed $\A$-path from $a'_1$ to $b'_0$ in $H_1$ and no directed $\A$-paths from $b'_2$ to $a'_3$ in $H_2$.
 
 \item There are neither directed $\A$-paths from $b'_2$ to $a'_3$ in $H_2$, nor from $a'_1$ to $b'_0$ in $H_1$. If $(\A\cup\{a_0,b_1,a_2,b_3\},\I)$ is a \CAI-partition of $G$, then we are done. Otherwise, if $e$ is oriented from $H_1$ towards $H_2$, then there must be a directed $\A$-cycle going through $a'_3$, $b_3$, $b_1$, $a'_1$, and $e$. In this case, $(\A\cup\{a_0,b_1,a_2\},\I\cup\{b_3\})$ is a \CAI-partition of $G$. The case when $e$ is oriented from $H_2$ towards $H_1$ is symmetric.
 \end{itemize}
\end{itemize}

\begin{figure}[!htbp]
\centering
\begin{minipage}[b]{0.32\textwidth}
\centering
\begin{tikzpicture}[scale=0.75]
\begin{scope}[every node/.style={circle,draw,minimum size=0.65cm,inner sep = 2}]
    \node[red,label={below:\textcolor{red}{$\A$}}] (a0) at (1,1) {$a_0$};
    \node[red,label={below:\textcolor{red}{$\A$}}] (b1) at (3,1) {$b_1$};
    \node[red,label={above:\textcolor{red}{$\A$}}] (a2) at (3,2.5) {$a_2$};
    \node[red,label={above:\textcolor{red}{$\I$}}] (b3) at (1,2.5) {$b_3$};
    
    \node[label={below:\textcolor{blue}{$\I$}}] (b'0) at (0,0) {$b'_0$};
    \node[label={below:\textcolor{blue}{$\A$}}] (a'1) at (4,0) {$a'_1$};
    \node[label={above:\textcolor{blue}{$\A$}}] (b'2) at (4,3.5) {$b'_2$};
    \node[label={above:\textcolor{blue}{$\A$}}] (a'3) at (0,3.5) {$a'_3$};

\end{scope}

\begin{scope}[every edge/.style={draw,minimum width = 0.04cm}]
    \path[red] (a0) edge[->-] (b1);
    \path[red] (b1) edge (a2);
    \path[red] (b3) edge[->-] (a2);
    \path[red] (b3) edge (a0);
    
    \path[red] (b'0) edge[->-] (a0);
    \path[red] (b1) edge[->-] (a'1);
    \path[red] (a2) edge[->-] (b'2);
    \path[red] (a'3) edge[->-] (b3);

    \path[blue] (a'3) edge[->-] (b'0);
    \path[blue] (a'1) edge[->-] (b'2);

\end{scope}
\end{tikzpicture}

~\\
~
\end{minipage}
\begin{minipage}[b]{0.32\textwidth}
\centering
\begin{tikzpicture}[scale=0.75]
\begin{scope}[every node/.style={circle,draw,minimum size=0.65cm,inner sep = 2}]
    \node[red] (a0) at (1,1) {$a_0$};
    \node[red] (b1) at (3,1) {$b_1$};
    \node[red] (a2) at (3,2.5) {$a_2$};
    \node[red] (b3) at (1,2.5) {$b_3$};
    
    \node[label={below:\textcolor{blue}{$\A$}}] (b'0) at (0,0) {$b'_0$};
    \node[label={below:\textcolor{blue}{$\A$}}] (a'1) at (4,0) {$a'_1$};
    \node[label={above:\textcolor{blue}{$\A$}}] (b'2) at (4,3.5) {$b'_2$};
    \node[label={above:\textcolor{blue}{$\A$}}] (a'3) at (0,3.5) {$a'_3$};

\end{scope}

\begin{scope}[every edge/.style={draw,minimum width = 0.04cm}]
    \path[red] (a0) edge[->-] (b1);
    \path[red] (b1) edge (a2);
    \path[red] (b3) edge[->-] (a2);
    \path[red] (b3) edge (a0);
    
    \path[red] (b'0) edge[->-] (a0);
    \path[red] (b1) edge[->-] (a'1);
    \path[red] (a2) edge[->-] (b'2);
    \path[red] (a'3) edge[->-] (b3);

    \path[blue] (a'3) edge[->-] (b'0);
    \path[blue] (a'1) edge[->-] (b'2);


    \path[->] (b'2) edge[decorate,bend right] (a'3);
    \path[->] (a'1) edge[decorate,bend left] (b'0);
\end{scope}
\end{tikzpicture}

Impossible.\\
~
\end{minipage}
\begin{minipage}[b]{0.32\textwidth}
\centering
\begin{tikzpicture}[scale=0.75]
\begin{scope}[every node/.style={circle,draw,minimum size=0.65cm,inner sep = 2}]
    \node[red,label={below:\textcolor{red}{$\A$}}] (a0) at (1,1) {$a_0$};
    \node[red,label={below:\textcolor{red}{$\A$}}] (b1) at (3,1) {$b_1$};
    \node[red,label={above:\textcolor{red}{$\I$}}] (a2) at (3,2.5) {$a_2$};
    \node[red,label={above:\textcolor{red}{$\A$}}] (b3) at (1,2.5) {$b_3$};
    
    \node[label={below:\textcolor{blue}{$\A$}}] (b'0) at (0,0) {$b'_0$};
    \node[label={below:\textcolor{blue}{$\A$}}] (a'1) at (4,0) {$a'_1$};
    \node[label={above:\textcolor{blue}{$\A$}}] (b'2) at (4,3.5) {$b'_2$};
    \node[label={above:\textcolor{blue}{$\A$}}] (a'3) at (0,3.5) {$a'_3$};

    \node[draw=none] (e1) at (5,3) {};
    \node[draw=none] (e2) at (5,0.5) {}; 
\end{scope}

\begin{scope}[every edge/.style={draw,minimum width = 0.04cm}]
    \path[red] (a0) edge[->-] (b1);
    \path[red] (b1) edge (a2);
    \path[red] (b3) edge[->-] (a2);
    \path[red] (b3) edge (a0);
    
    \path[red] (b'0) edge[->-] (a0);
    \path[red] (b1) edge[->-] (a'1);
    \path[red] (a2) edge[->-] (b'2);
    \path[red] (a'3) edge[->-] (b3);

    \path[blue] (a'3) edge[->-] (b'0);
    \path[blue] (a'1) edge[->-] (b'2);

    \path (e1) edge[->-] node[left] {$e$} (e2);

    \path[->] (b'2) edge[decorate,bend right] (a'3);
\end{scope}
\end{tikzpicture}

No directed $\A$-path from $a'_1$ to $b'_0$ in $H_1$.
\end{minipage}

\begin{minipage}[b]{0.32\textwidth}
\centering
\begin{tikzpicture}[scale=0.75]
\begin{scope}[every node/.style={circle,draw,minimum size=0.65cm,inner sep = 2}]
    \node[red,label={below:\textcolor{red}{$\A$}}] (a0) at (1,1) {$a_0$};
    \node[red,label={below:\textcolor{red}{$\A$}}] (b1) at (3,1) {$b_1$};
    \node[red,label={above:\textcolor{red}{$\A$}}] (a2) at (3,2.5) {$a_2$};
    \node[red,label={above:\textcolor{red}{$\I$}}] (b3) at (1,2.5) {$b_3$};
    
    \node[label={below:\textcolor{blue}{$\A$}}] (b'0) at (0,0) {$b'_0$};
    \node[label={below:\textcolor{blue}{$\A$}}] (a'1) at (4,0) {$a'_1$};
    \node[label={above:\textcolor{blue}{$\A$}}] (b'2) at (4,3.5) {$b'_2$};
    \node[label={above:\textcolor{blue}{$\A$}}] (a'3) at (0,3.5) {$a'_3$};

    \node[draw=none] (e1) at (5,3) {};
    \node[draw=none] (e2) at (5,0.5) {}; 
\end{scope}

\begin{scope}[every edge/.style={draw,minimum width = 0.04cm}]
    \path[red] (a0) edge[->-] (b1);
    \path[red] (b1) edge (a2);
    \path[red] (b3) edge[->-] (a2);
    \path[red] (b3) edge (a0);
    
    \path[red] (b'0) edge[->-] (a0);
    \path[red] (b1) edge[->-] (a'1);
    \path[red] (a2) edge[->-] (b'2);
    \path[red] (a'3) edge[->-] (b3);

    \path[blue] (a'3) edge[->-] (b'0);
    \path[blue] (a'1) edge[->-] (b'2);

    \path (e2) edge[->-] node[left] {$e$} (e1);

    \path[->] (b'2) edge[decorate,bend right] (a'3);
\end{scope}
\end{tikzpicture}

No directed $\A$-path from $a'_1$ to $b'_0$ in $H_1$.\\
~
\end{minipage}
\begin{minipage}[b]{0.32\textwidth}
\centering
\begin{tikzpicture}[scale=0.75]
\begin{scope}[every node/.style={circle,draw,minimum size=0.65cm,inner sep = 2}]
    \node[red,label={below:\textcolor{red}{$\A$}}] (a0) at (1,1) {$a_0$};
    \node[red,label={below:\textcolor{red}{$\A$}}] (b1) at (3,1) {$b_1$};
    \node[red,label={above:\textcolor{red}{$\A$}}] (a2) at (3,2.5) {$a_2$};
    \node[red,label={above:\textcolor{red}{$\A$}}] (b3) at (1,2.5) {$b_3$};
    
    \node[label={below:\textcolor{blue}{$\A$}}] (b'0) at (0,0) {$b'_0$};
    \node[label={below:\textcolor{blue}{$\A$}}] (a'1) at (4,0) {$a'_1$};
    \node[label={above:\textcolor{blue}{$\A$}}] (b'2) at (4,3.5) {$b'_2$};
    \node[label={above:\textcolor{blue}{$\A$}}] (a'3) at (0,3.5) {$a'_3$};

    \node[draw=none] (e1) at (5,3) {};
    \node[draw=none] (e2) at (5,0.5) {}; 
\end{scope}

\begin{scope}[every edge/.style={draw,minimum width = 0.04cm}]
    \path[red] (a0) edge[->-] (b1);
    \path[red] (b1) edge (a2);
    \path[red] (b3) edge[->-] (a2);
    \path[red] (b3) edge (a0);
    
    \path[red] (b'0) edge[->-] (a0);
    \path[red] (b1) edge[->-] (a'1);
    \path[red] (a2) edge[->-] (b'2);
    \path[red] (a'3) edge[->-] (b3);

    \path[blue] (a'3) edge[->-] (b'0);
    \path[blue] (a'1) edge[->-] (b'2);

    \path (e2) edge[->-] node[left] {$e$} (e1);
\end{scope}
\end{tikzpicture}

No directed $\A$-path from $a'_1$ to $b'_0$ in $H_1$, from $b'_2$ to $a'_3$ in $H_2$, or from $a'_1$ to $a'_3$ going through $e$.
\end{minipage}
\begin{minipage}[b]{0.32\textwidth}
\centering
\begin{tikzpicture}[scale=0.75]
\begin{scope}[every node/.style={circle,draw,minimum size=0.65cm,inner sep = 2}]
    \node[red,label={below:\textcolor{red}{$\A$}}] (a0) at (1,1) {$a_0$};
    \node[red,label={below:\textcolor{red}{$\A$}}] (b1) at (3,1) {$b_1$};
    \node[red,label={above:\textcolor{red}{$\A$}}] (a2) at (3,2.5) {$a_2$};
    \node[red,label={above:\textcolor{red}{$\I$}}] (b3) at (1,2.5) {$b_3$};
    
    \node[label={below:\textcolor{blue}{$\A$}}] (b'0) at (0,0) {$b'_0$};
    \node[label={below:\textcolor{blue}{$\A$}}] (a'1) at (4,0) {$a'_1$};
    \node[label={above:\textcolor{blue}{$\A$}}] (b'2) at (4,3.5) {$b'_2$};
    \node[label={above:\textcolor{blue}{$\A$}}] (a'3) at (0,3.5) {$a'_3$};

    \node[draw=none] (e1) at (5,3) {};
    \node[draw=none] (e2) at (5,0.5) {}; 
\end{scope}

\begin{scope}[every edge/.style={draw,minimum width = 0.04cm}]
    \path[red] (a0) edge[->-] (b1);
    \path[red] (b1) edge (a2);
    \path[red] (b3) edge[->-] (a2);
    \path[red] (b3) edge (a0);
    
    \path[red] (b'0) edge[->-] (a0);
    \path[red] (b1) edge[->-] (a'1);
    \path[red] (a2) edge[->-] (b'2);
    \path[red] (a'3) edge[->-] (b3);

    \path[blue] (a'3) edge[->-] (b'0);
    \path[blue] (a'1) edge[->-] (b'2);

    \path (e2) edge[->-] node[left] {$e$} (e1);
    
    \path[->] (a'1) edge[decorate,bend right] (5,1);
    \path[->] (5,2.55) edge[decorate,bend right=90] (a'3);
\end{scope}
\end{tikzpicture}

No directed $\A$-path from $a'_1$ to $b'_0$ in $H_1$ or from $b'_2$ to $a'_3$ in $H_2$.\\
~
\end{minipage}
\caption{Case 2 of~\Cref{lem:C4}.}
\label{fig:lem:C4-2}
\end{figure}

\textbf{Case 3:} Suppose that we are not in Case 1, nor in Case 2. Suppose w.l.o.g. that $G$ contains $\ora{a_2b'_2}$. We define $H$ depending on the orientation of $a'_3b_3$ in $G$:
\begin{itemize}
 \item $\ora{a'_3b_3}$ : $H=G-\{a_0,b_1,a_2,b_3\}+\ora{a'_1b'_2}+\ora{a'_3b'_0}$,
 \item $\ora{b_3a'_3}$ : $H=G-\{a_0,b_1,a_2,b_3\}+\ora{a'_1b'_2}+\ora{b'_0a'_3}$.
\end{itemize} 
See~\Cref{fig:lem:C4-3}.
\begin{observation}\label{obs:H}
 Any directed cycle $C'$ in $G$ containing edges $b'_0a_0$, $a_0b_3$, $b_3a'_3$ (resp. $a'_1b_1$, $b_1a_2$, $a_2b'_2$) creates a directed cycle $C'-\{b'_0a_0,a_0b_3,b_3a'_3\}+b'_0a'_3$ (resp. $C'-\{a'_1b_1,b_1a_2,a_2b'_2\}+a'_1b'_2$) in $H$. 
\end{observation}

By \Cref{claim:$2$-vertex-connected}, $H$ is $2$-vertex-connected and is in $\F$. Let \AI~be a \CAI-partition of $H$. Observe that $|\{b'_0,a'_1,b'_2,a'_3\}\cap\A|\geq 2$ since $\I$ is an independent set in $H$. Thus, we have the following two cases.
\begin{itemize}
 \item Suppose $|\{b'_0,a'_1,b'_2,a'_3\}\cap\A|\leq 3$. W.l.o.g. we can assume that $b'_0\in\I$ and therefore $a'_3\in\A$. Then $(\A\cup\{a_0,b_1,a_2\},\I\cup\{b_3\})$ is a \CAI-partition of $G$.
 \item Suppose $|\{b'_0,a'_1,b'_2,a'_3\}\cap\A|= 4$. Now we distinguish the following cases.
 \begin{itemize}
 \item There exists an $\A$-path between $a'_1$ and $b'_0$ and an $\A$-path between $a'_3$ and $b'_2$ in $G-C$. If $\A$ is connected in $G-C$, then $(\A\cup\{a_0,a_2\},\I\cup\{b_1,b_3\})$ is a \CAI-partition of $G$. Therefore, there is no $\A$-path between $a'_3$ and $b'_0$ or $a'_1$, as well as between $b'_2$ and $b'_0$ or $a'_1$. Since $(\A',\I')=(\A\cup\{a_0,a_2,b_3\},\I\cup\{b_1\})$ is not a \CAI-partition of $G$, there must be a directed $\A'$-cycle containing $\ora{a'_3b_3}$, $\ora{b_3a_2}$, $\ora{a_2b'_2}$, and a directed $\A$-path from $b'_2$ to $a'_3$. Similarly, since $(\A'',\I'')=(\A\cup\{a_0,b_1,b_3\},\I\cup\{a_2\})$ is not a \CAI-partition of $G$, there must be an $\A''$-cycle containing $\ora{a'_1b_1}$, $\ora{b_1a_0}$, $\ora{a_0b'_0}$, and a directed $\A$-path from $b'_0$ to $a'_1$ (if the arcs were reversed then we would be in \textbf{Case 1} or \textbf{Case 2}). However, the directed $\A$-path from $b'_0$ to $a'_1$, $\ora{a'_1b'_2}$, the directed $\A$-path from $b'_2$ to $a'_3$, and $\ora{a'_3b'_0}$ form a directed $\A$-cycle in $H$, a contradiction.
 \item There exists either an $\A$-path between $a'_1$ and $b'_0$ or an $\A$-path between $a'_3$ and $b'_2$ in $G-C$. By symmetry, we assume that it is the latter. Since $(\A',\I')=(\A\cup\{a_0,b_1,a_2\},\I\cup\{b_3\})$ is not a \CAI-partition of $G$, there must be a directed $\A'$-cycle in $G$. This $\A'$-cycle cannot contain $a'_1$ and $b'_0$ since there is no $\A$-path between them in $G-C$. This cycle cannot contain $a'_1$, $b_1$, $a_2$, and $b'_2$ by \Cref{obs:H}. Therefore, this $\A'$-cycle contains an $\A$-path between $b'_2$ and $b'_0$. Using the same arguments, $G$ contains an $\A$-path between $a'_1$ and $a'_3$. Hence, we go back to the case where $\A$ is connected in $G-C$ and $(\A\cup\{a_0,a_2\},\I\cup\{b_1,b_3\})$ is a \CAI-partition of $G$.
 \item There is no $\A$-path between $a'_1$ and $b'_0$ and no $\A$-path between $a'_3$ and $b'_2$ in $G-C$. Since $\A$ must be connected in $H$, we can suppose w.l.o.g. that there exists an $\A$-path between $a'_1$ and $a'_3$. Since $(\A',\I')=(\A\cup\{a_0,b_1,a_2\},\I\cup\{b_3\})$ is not a \CAI-partition of $G$. There must be a directed $\A'$-cycle in $G$. This $\A'$-cycle cannot contain $a'_1$ and $b'_0$ since there is no $\A$-path between them in $G-C$. This cycle cannot contain $a'_1$, $b_1$, $a_2$, and $b'_2$ by \Cref{obs:H}. Therefore, this $\A'$-cycle contains $\ora{b'_0a_0}$, $\ora{a_0b_1}$, $\ora{b_1a_2}$, $\ora{a_2b'_2}$, and a directed $\A$-path from $b'_2$ to $b'_0$ (in particular, the orientations of $a_0b_1$, $b_1a_2$ are forced). Using the same arguments, since $(\A\cup\{a_0,a_2,b_3\},\I\cup\{b_1\})$ is not a \CAI-partition of $G$, we have that $G$ contains $\ora{a_0b_3}$ and $\ora{b_3a_2}$. Hence, $(\A\cup\{a_0,b_1,b_3\},\I\cup\{a_2\})$ is a \CAI-partition of $G$.\qedhere 
 \end{itemize}
\end{itemize}
\end{proof}

\begin{figure}[!htbp]
\centering
\begin{minipage}[b]{0.19\textwidth}
\centering
\begin{tikzpicture}[scale=0.67]
\begin{scope}[every node/.style={circle,draw,minimum size=0.65cm,inner sep = 2}]
    \node[red,label={below:\textcolor{red}{$\A$}}] (a0) at (1,1) {$a_0$};
    \node[red,label={below:\textcolor{red}{$\A$}}] (b1) at (3,1) {$b_1$};
    \node[red,label={above:\textcolor{red}{$\A$}}] (a2) at (3,2.5) {$a_2$};
    \node[red,label={above:\textcolor{red}{$\I$}}] (b3) at (1,2.5) {$b_3$};
    
    \node[label={below:\textcolor{blue}{$\I$}}] (b'0) at (0,0) {$b'_0$};
    \node (a'1) at (4,0) {$a'_1$};
    \node (b'2) at (4,3.5) {$b'_2$};
    \node[label={above:\textcolor{blue}{$\A$}}] (a'3) at (0,3.5) {$a'_3$};
\end{scope}

\begin{scope}[every edge/.style={draw,minimum width = 0.04cm}]
    \path[red] (a0) edge (b1);
    \path[red] (b1) edge (a2);
    \path[red] (b3) edge (a2);
    \path[red] (b3) edge (a0);
    
    \path[red] (b'0) edge (a0);
    \path[red] (b1) edge (a'1);
    \path[red] (a2) edge[->-] (b'2);
    \path[red] (a'3) edge[->-] (b3);

    \path[blue] (a'3) edge[->-] (b'0);
    \path[blue] (a'1) edge[->-] (b'2);

\end{scope}
\end{tikzpicture}

~\\
~\\
~
\end{minipage}
\begin{minipage}[b]{0.19\textwidth}
\centering
\begin{tikzpicture}[scale=0.67]
\begin{scope}[every node/.style={circle,draw,minimum size=0.65cm,inner sep = 2}]
    \node[red,label={below:\textcolor{red}{$\A$}}] (a0) at (1,1) {$a_0$};
    \node[red,label={below:\textcolor{red}{$\I$}}] (b1) at (3,1) {$b_1$};
    \node[red,label={above:\textcolor{red}{$\A$}}] (a2) at (3,2.5) {$a_2$};
    \node[red,label={above:\textcolor{red}{$\I$}}] (b3) at (1,2.5) {$b_3$};
    
    \node[label={below:\textcolor{blue}{$\A$}}] (b'0) at (0,0) {$b'_0$};
    \node[label={below:\textcolor{blue}{$\A$}}] (a'1) at (4,0) {$a'_1$};
    \node[label={above:\textcolor{blue}{$\A$}}] (b'2) at (4,3.5) {$b'_2$};
    \node[label={above:\textcolor{blue}{$\A$}}] (a'3) at (0,3.5) {$a'_3$};
\end{scope}

\begin{scope}[every edge/.style={draw,minimum width = 0.04cm}]
    \path[red] (a0) edge (b1);
    \path[red] (b1) edge (a2);
    \path[red] (b3) edge (a2);
    \path[red] (b3) edge (a0);
    
    \path[red] (b'0) edge (a0);
    \path[red] (b1) edge (a'1);
    \path[red] (a2) edge[->-] (b'2);
    \path[red] (a'3) edge[->-] (b3);

    \path[blue] (a'3) edge[->-] (b'0);
    \path[blue] (a'1) edge[->-] (b'2);

\end{scope}
\end{tikzpicture}

$\A$ is connected in $G-C$.\\
~
\end{minipage}
\begin{minipage}[b]{0.19\textwidth}
\centering
\begin{tikzpicture}[scale=0.67]
\begin{scope}[every node/.style={circle,draw,minimum size=0.65cm,inner sep = 2}]
    \node[red,label={below:\textcolor{red}{$\A$}}] (a0) at (1,1) {$a_0$};
    \node[red,label={below:\textcolor{red}{$\I$}}] (b1) at (3,1) {$b_1$};
    \node[red,label={above:\textcolor{red}{$\A$}}] (a2) at (3,2.5) {$a_2$};
    \node[red,label={above:\textcolor{red}{$\A$}}] (b3) at (1,2.5) {$b_3$};
    
    \node[label={below:\textcolor{blue}{$\A$}}] (b'0) at (0,0) {$b'_0$};
    \node[label={below:\textcolor{blue}{$\A$}}] (a'1) at (4,0) {$a'_1$};
    \node[label={above:\textcolor{blue}{$\A$}}] (b'2) at (4,3.5) {$b'_2$};
    \node[label={above:\textcolor{blue}{$\A$}}] (a'3) at (0,3.5) {$a'_3$};
\end{scope}

\begin{scope}[every edge/.style={draw,minimum width = 0.04cm}]
    \path[red] (a0) edge (b1);
    \path[red] (b1) edge (a2);
    \path[red] (b3) edge (a2);
    \path[red] (b3) edge (a0);
    
    \path[red] (b'0) edge (a0);
    \path[red] (b1) edge (a'1);
    \path[red] (a2) edge[->-] (b'2);
    \path[red] (a'3) edge[->-] (b3);

    \path[blue] (a'3) edge[->-] (b'0);
    \path[blue] (a'1) edge[->-] (b'2);

    \path (a'1) edge[decorate, bend left] (b'0);
    \path (a'3) edge[decorate, bend left] (b'2);
\end{scope}
\end{tikzpicture}

No directed $\A$-path from $b'_2$ to $a'_3$ in $G-C$.
\end{minipage}
\begin{minipage}[b]{0.19\textwidth}
\centering
\begin{tikzpicture}[scale=0.67]
\begin{scope}[every node/.style={circle,draw,minimum size=0.65cm,inner sep = 2}]
    \node[red,label={below:\textcolor{red}{$\A$}}] (a0) at (1,1) {$a_0$};
    \node[red,label={below:\textcolor{red}{$\A$}}] (b1) at (3,1) {$b_1$};
    \node[red,label={above:\textcolor{red}{$\I$}}] (a2) at (3,2.5) {$a_2$};
    \node[red,label={above:\textcolor{red}{$\A$}}] (b3) at (1,2.5) {$b_3$};
    
    \node[label={below:\textcolor{blue}{$\A$}}] (b'0) at (0,0) {$b'_0$};
    \node[label={below:\textcolor{blue}{$\A$}}] (a'1) at (4,0) {$a'_1$};
    \node[label={above:\textcolor{blue}{$\A$}}] (b'2) at (4,3.5) {$b'_2$};
    \node[label={above:\textcolor{blue}{$\A$}}] (a'3) at (0,3.5) {$a'_3$};
\end{scope}

\begin{scope}[every edge/.style={draw,minimum width = 0.04cm}]
    \path[red] (a0) edge (b1);
    \path[red] (b1) edge (a2);
    \path[red] (b3) edge (a2);
    \path[red] (b3) edge (a0);
    
    \path[red] (b'0) edge (a0);
    \path[red] (b1) edge (a'1);
    \path[red] (a2) edge[->-] (b'2);
    \path[red] (a'3) edge[->-] (b3);

    \path[blue] (a'3) edge[->-] (b'0);
    \path[blue] (a'1) edge[->-] (b'2);

    \path (a'1) edge[decorate, bend left] (b'0);
    \path (a'3) edge[decorate, bend left] (b'2);
\end{scope}
\end{tikzpicture}

No directed $\A$-path from $b'_0$ to $a'_1$ in $G-C$.
\end{minipage}
\begin{minipage}[b]{0.19\textwidth}
\centering
\begin{tikzpicture}[scale=0.67]
\begin{scope}[every node/.style={circle,draw,minimum size=0.65cm,inner sep = 2}]
    \node[red] (a0) at (1,1) {$a_0$};
    \node[red] (b1) at (3,1) {$b_1$};
    \node[red] (a2) at (3,2.5) {$a_2$};
    \node[red] (b3) at (1,2.5) {$b_3$};
    
    \node[label={below:\textcolor{blue}{$\A$}}] (b'0) at (0,0) {$b'_0$};
    \node[label={below:\textcolor{blue}{$\A$}}] (a'1) at (4,0) {$a'_1$};
    \node[label={above:\textcolor{blue}{$\A$}}] (b'2) at (4,3.5) {$b'_2$};
    \node[label={above:\textcolor{blue}{$\A$}}] (a'3) at (0,3.5) {$a'_3$};
\end{scope}

\begin{scope}[every edge/.style={draw,minimum width = 0.04cm}]
    \path[red] (a0) edge (b1);
    \path[red] (b1) edge (a2);
    \path[red] (b3) edge (a2);
    \path[red] (b3) edge (a0);
    
    \path[red] (b'0) edge (a0);
    \path[red] (b1) edge (a'1);
    \path[red] (a2) edge[->-] (b'2);
    \path[red] (a'3) edge[->-] (b3);

    \path[blue] (a'3) edge[->-] (b'0);
    \path[blue] (a'1) edge[->-] (b'2);

    \path[->] (b'0) edge[decorate, bend right] (a'1);
    \path[->] (b'2) edge[decorate, bend right] (a'3);
\end{scope}
\end{tikzpicture}

Impossible.\\
~\\
~
\end{minipage}

\begin{minipage}[b]{0.32\textwidth}
\centering
\begin{tikzpicture}[scale=0.75]
\begin{scope}[every node/.style={circle,draw,minimum size=0.65cm,inner sep = 2}]
    \node[red,label={below:\textcolor{red}{$\A$}}] (a0) at (1,1) {$a_0$};
    \node[red,label={below:\textcolor{red}{$\A$}}] (b1) at (3,1) {$b_1$};
    \node[red,label={above:\textcolor{red}{$\A$}}] (a2) at (3,2.5) {$a_2$};
    \node[red,label={above:\textcolor{red}{$\I$}}] (b3) at (1,2.5) {$b_3$};
    
    \node[label={below:\textcolor{blue}{$\A$}}] (b'0) at (0,0) {$b'_0$};
    \node[label={below:\textcolor{blue}{$\A$}}] (a'1) at (4,0) {$a'_1$};
    \node[label={above:\textcolor{blue}{$\A$}}] (b'2) at (4,3.5) {$b'_2$};
    \node[label={above:\textcolor{blue}{$\A$}}] (a'3) at (0,3.5) {$a'_3$};
\end{scope}

\begin{scope}[every edge/.style={draw,minimum width = 0.04cm}]
    \path[red] (a0) edge (b1);
    \path[red] (b1) edge (a2);
    \path[red] (b3) edge (a2);
    \path[red] (b3) edge (a0);
    
    \path[red] (b'0) edge (a0);
    \path[red] (b1) edge (a'1);
    \path[red] (a2) edge[->-] (b'2);
    \path[red] (a'3) edge[->-] (b3);

    \path[blue] (a'3) edge[->-] (b'0);
    \path[blue] (a'1) edge[->-] (b'2);

    \path (a'3) edge[decorate, bend left] (b'2);
\end{scope}
\end{tikzpicture}

No $\A$-path between $b'_0$ and $a'_1$ and no directed $\A$-path from $b'_2$ to $b'_0$ in $G-C$.
\end{minipage}
\begin{minipage}[b]{0.32\textwidth}
\centering
\begin{tikzpicture}[scale=0.75]
\begin{scope}[every node/.style={circle,draw,minimum size=0.65cm,inner sep = 2}]
    \node[red,label={below:\textcolor{red}{$\A$}}] (a0) at (1,1) {$a_0$};
    \node[red,label={below:\textcolor{red}{$\A$}}] (b1) at (3,1) {$b_1$};
    \node[red,label={above:\textcolor{red}{$\A$}}] (a2) at (3,2.5) {$a_2$};
    \node[red,label={above:\textcolor{red}{$\I$}}] (b3) at (1,2.5) {$b_3$};
    
    \node[label={below:\textcolor{blue}{$\A$}}] (b'0) at (0,0) {$b'_0$};
    \node[label={below:\textcolor{blue}{$\A$}}] (a'1) at (4,0) {$a'_1$};
    \node[label={above:\textcolor{blue}{$\A$}}] (b'2) at (4,3.5) {$b'_2$};
    \node[label={above:\textcolor{blue}{$\A$}}] (a'3) at (0,3.5) {$a'_3$};
\end{scope}

\begin{scope}[every edge/.style={draw,minimum width = 0.04cm}]
    \path[red] (a0) edge (b1);
    \path[red] (b1) edge (a2);
    \path[red] (b3) edge (a2);
    \path[red] (b3) edge (a0);
    
    \path[red] (b'0) edge (a0);
    \path[red] (b1) edge (a'1);
    \path[red] (a2) edge[->-] (b'2);
    \path[red] (a'3) edge[->-] (b3);

    \path[blue] (a'3) edge[->-] (b'0);
    \path[blue] (a'1) edge[->-] (b'2);

    \path (a'1) edge[decorate, bend right] (4.5,4);
    \path (4.5,4) edge[decorate, bend right=50] (a'3);
\end{scope}
\end{tikzpicture}

No $\A$-path between $b'_0$ and $a'_1$, or between $b'_2$ and $a'_3$, and no directed $\A$-path from $b'_2$ to $b'_0$ in $G-C$.
\end{minipage}
\begin{minipage}[b]{0.32\textwidth}
\centering
\begin{tikzpicture}[scale=0.75]
\begin{scope}[every node/.style={circle,draw,minimum size=0.65cm,inner sep = 2}]
    \node[red,label={below:\textcolor{red}{$\A$}}] (a0) at (1,1) {$a_0$};
    \node[red,label={below:\textcolor{red}{$\A$}}] (b1) at (3,1) {$b_1$};
    \node[red,label={above:\textcolor{red}{$\I$}}] (a2) at (3,2.5) {$a_2$};
    \node[red,label={above:\textcolor{red}{$\A$}}] (b3) at (1,2.5) {$b_3$};
    
    \node[label={below:\textcolor{blue}{$\A$}}] (b'0) at (0,0) {$b'_0$};
    \node[label={below:\textcolor{blue}{$\A$}}] (a'1) at (4,0) {$a'_1$};
    \node[label={above:\textcolor{blue}{$\A$}}] (b'2) at (4,3.5) {$b'_2$};
    \node[label={above:\textcolor{blue}{$\A$}}] (a'3) at (0,3.5) {$a'_3$};
\end{scope}

\begin{scope}[every edge/.style={draw,minimum width = 0.04cm}]
    \path[red] (a0) edge[->-] (b1);
    \path[red] (b1) edge[->-] (a2);
    \path[red] (b3) edge[->-] (a2);
    \path[red] (a0) edge[->-] (b3);
    
    \path[red] (b'0) edge[->-] (a0);
    \path[red] (b1) edge (a'1);
    \path[red] (a2) edge[->-] (b'2);
    \path[red] (a'3) edge[->-] (b3);

    \path[blue] (a'3) edge[->-] (b'0);
    \path[blue] (a'1) edge[->-] (b'2);

    \path (a'1) edge[decorate, bend right] (4.5,4);
    \path (4.5,4) edge[decorate, bend right=50] (a'3);
    \path (b'2) edge[decorate, bend right=50] (-0.5,4);
    \path (-0.5,4) edge[decorate, bend right] (b'0);
\end{scope}
\end{tikzpicture}

No $\A$-path between $b'_0$ and $a'_1$, or between $b'_2$ and $a'_3$ in $G-C$.\\
~
\end{minipage}
\caption{Case 3 of~\Cref{lem:C4} with $\ora{a'_3b_3}$.}
\label{fig:lem:C4-3}
\end{figure}

\begin{lemma}\label{lem:3-cuts}
    Let $u$ and $v$ be $2$-vertices of $G$, then $G-\{u,v\}$ is $2$-connected.
\end{lemma}

\begin{proof}
     Let $H=G-\{u,v\}$.
     
     First, we show that $H$ is connected. 
     Suppose by contradiction that $H$ is disconnected. 
     We will build $H'\in\F$ from $H$ such that $|V(H')|+|E(H')|<|V(G)|+|E(G)|$ and extend a \CAI-partition of $H'$ to $G$, thus obtaining a contradiction. 
     Let $t$ and $w$ be neighbors of $u$ in $G$.  
     Suppose that $u$ is incident to arcs $\ora{tu}$ and $\ora{uw}$. 
     In such case, we add $\ora{tw}$ to $H$, otherwise, we add $\ora{wt}$ to $H$.
     We do the same between neighbors of $v$ and obtain $H'$.
     Since $G\in \F$, $H'$ remains 2-connected, subcubic, oriented, and planar.
     Moreover, since $G$ is 2-connected, there are exactly two connected components $H_1$ and $H_2$ in $H$ and $\{u,v\}$ forms a cut-set of $G$.
     Let $(A,B)$ be the bipartition of $G$, let $(A_1,B_1)=(A\cap V(H_1),B\cap V(H_1))$ and $(A_2,B_2)=(A\cap V(H_2),B\cap V(H_2))$ be the bipartitions of $H_1$ and $H_2$ respectively.
     Observe that $(A_1\cup B_2,B_1\cup A_2)$ is a bipartition of $H'$.
     Therefore, $H'\in \F$.
     In addition, $|V(H')|+|E(H')|=|V(G)|+|E(G)|-4$.
     By minimality of $G$, there exists a \CAI-partition $(\A,\I)$ of $H'$.
     We claim that $(\A',\I')=(\A\cup\{u,v\},\I)$ is a \CAI-partition of $G$.
     Indeed, if it is not a \CAI-partition of $G$, then it must contain a directed $\A'$-cycle going through $u$ or $v$.
     However, by construction of $H'$, if such a directed cycle exists then it must exist in $\A$, which contradicts the fact that $\A$ is an acyclic set.

     Now, we show that $H$ is $2$-connected. 
     Suppose by contradiction that $H$ is not 2-connected. 
     Since $H$ is connected, it must contain a bridge $\ora{xy}$.
     Similarly to the previous case, we will build $H'\in\F$ from $H-\ora{xy}$ such that $|V(H')|+|E(H')|<|V(G)|+|E(G)|$ and extend a \CAI-partition of $H'$ to $G$.
     We add arcs between neighbors of $u$ and $v$ in the same fashion as when we proved that $H$ is connected. 
     Moreover, we add a vertex $z$ with arcs $\ora{xz}$ and $\ora{zy}$ to obtain $H'$.
     Moreover, 
     there are exactly two connected components $H_1$ and $H_2$ in $H-\ora{xy}$ and $\{u,v,x\}$ forms a cut-set of $G$.
     Let $(A,B)$ be the bipartition of $G$, let $(A_1,B_1)=(A\cap V(H_1),B\cap V(H_1))$ and $(A_2,B_2)=(A\cap V(H_2),B\cap V(H_2))$ be the bipartitions of $H_1$ and $H_2$ respectively.
     Suppose w.l.o.g. that $x\in A_1$.
     Observe that $(A_1\cup B_2,B_1\cup A_2\cup\{z\})$ is a bipartition of $H'$.
     Since $G\in \F$, all other properties of $G$ also remains in $H'$ so $H'\in\F$.
     In addition, $|V(H')|+|E(H')|=|V(G)|+|E(G)|-2$.
     By minimality of $G$, there exists a \CAI-partition $(\A,\I)$ of $H'$.\\
     Suppose that $z\in\A$. 
     We claim that $(\A',\I')=((\A\setminus\{z\})\cup\{u,v\},\I)$ is a \CAI-partition of $G$.
     Similarly to the proof of $H$ being connected, there is no directed $\A'$-cycle. 
     Moreover, losing $z$ does not disconnect $G[\A']$ since $x$ and $y$ would have been in $\A$, thus in $\A'$, and they are connected by $\ora{xy}$ in $G$.\\ 
     Suppose that $z\in\I$.
     As a consequence, $x,y\in\A$.
     Since $(\A',\I')=((\A\setminus\{z\})\cup\{u,v\},\I)$ cannot be a \CAI-partition of $G$, there must be a directed $\A'$-cycle going through $\ora{xy}$.
     Since $\{u,v,x\}$ forms a cut-set of $G$, such a cycle must go through $u$ and/or $v$.
     If such a cycle goes through $u$, then we put $u$ in $\I'$ instead.
     We do the same for $v$.
     We claim that the resulting partition $(\A'',\I'')$ is a \CAI-partition of $G$.
     Indeed, there are no $\A''$-directed cycle by construction of $\A''$.
     Moreover, $G[\A'']$ must be connected because whenever we put $u$ in $\I''$, the neighbors of $u$ are in $\A''$ and they are connected by the path remaining from a directed cycle going through $\ora{xy}$ and $u$ in $\A'$. 
     The same holds for $v$.
     This concludes the proof.
\end{proof}

\begin{lemma}\label{lem:23b}
$G$ cannot have two $2$-vertices at facial distance $3$ or less. 
\end{lemma}

\begin{proof}
Suppose by contradiction that it is not true, and so by~\Cref{lem:22} and~\Cref{lem:232} there exists a path $a_0b_1a_2b_3a_4b_5$ lying on some $k$-face of $G$ such that vertices $b_1$ and $a_4$ are of degree $2$ and vertices
$a_0,a_2,b_3,b_5$ have degree $3$. Moreover, by \Cref{lem:3-cuts}, $G-\{b_1,a_4\}$ is $2$-connected
and thus $G-\{b_1,a_4\}$ is in $\F$. Let $H=G-\{b_1,a_4\}$ to which we add the arc $\ora{a_0b_5}$ if these two vertices are not adjacent in $G$.
 Take a \CAI-partition $(\A,\I)$ of $H$.
 
 If $\{a_0,a_2\}\subset\I$, then all the neighbors of $a_0$ and $a_2$ are in $\A$. Now since $\A$ is connected in $H$, we get that $(\A\cup\{b_1,a_2,a_4\}\setminus\{b_3\},\I\cup\{b_3\}\setminus\{a_2\})$ is a \CAI-partition of $G$.
 If $a_0\in\I$ and $a_2\in\A$, then depending whether adding $a_4$ to $\A$ creates a cycle or not, either $(\A\cup\{b_1,a_4\},\I)$ or $(\A\cup\{b_1\},\I\cup\{a_4\})$ is a \CAI-partition of $G$. Therefore, we conclude that $a_0\in\A$.
 
 Suppose $b_5\in\I$. Observe that among $a_2$ and $b_3$, at least one must be in $\A$. 
 Moreover, since $\A$ is connected in $H$ and $b_5\in\I$, there is an $\A$-path in $H$ (and in $G$) from $a_0$ to every vertex in $\A$, in particular to either $a_2$ or $b_3$. Thus we build a \CAI-partition $(\A',\I')$ of $G$ as follows:
 \begin{itemize}
 \item If $b_3\in\I$, then $\A'=\A\setminus\{a_2\}\cup\{b_1,b_3,a_4\}$ and $\I'=\I\setminus\{b_3\}\cup\{a_2\}$. Note that since the last neighbor of $b_3$ is in $\A$, $\A'$ remains connected.
 \item If $b_3\in\A$, then $\A'=\A\cup\{a_4\}$. Now, if $a_2\in\I$ add $b_1$ to $\A'$ and otherwise add $b_1$ to $\I'$.
 \end{itemize}
 We conclude that $b_5\in\A$. Again, observe that among $a_2$ and $b_3$, at least one must be in $\A$. And since $\{a_0,b_5\}\subset\A$, w.l.o.g, we can assume that $a_2\in\A$. We build a \CAI-partition $(\A',\I')$ of $G$ as follows:
 \begin{itemize}
 
 \item If there is an $\A$-path between $a_0$ and $b_5$ in $G-\{b_1,a_4\}$, then :
 \begin{itemize}
 \item If $b_3\in\A$ then $\A'=\A$ and $\I'=\I\cup\{b_1,a_4\}$.
 \item If $b_3\in\I$, then $\A'=\A\cup\{a_4\}$ and $\I'=\I\cup\{b_1\}$
 \end{itemize}

 \item If all the $\A$-paths from $a_0$ to $b_5$ in $H$ contain $\ora{a_0b_5}$, then since $\A$ must be connected in $H$, either there is an $\A$-path in $G$ from $a_0$ to $a_2$ or from $b_5$ to $a_2$, but not both. Therefore:
 \begin{itemize}
 \item Suppose $b_3\in\A$. Then there is an $\A$-path in $G$ from $a_0$ to $b_3$ or from $b_5$ to $b_3$, but not both. For the former we fix $\A'=\A\cup\{a_4\}$ and $\I'=\I\cup\{b_1\}$, while for the latter fix $\A'=\A\cup\{b_1\}$ and $\I'=\I\cup\{a_4\}$
 \item Suppose $b_3\in\I$.
 \begin{itemize}
 \item If there is an $\A$-path in $G$ from $a_2$ to $b_5$, then there is no $\A$-path in $G$ from $a_0$ to $a_2$. Thus we can fix $\A'=\A\cup\{b_1,a_4\}$ and $\I'=\I$.
 \item So there is no $\A$-path in $G$ from $a_2$ to $b_5$, and therefore there is an $\A$-path in $G$ from $a_0$ to $a_2$ (since $\A$ is connected in $H$). Let $a'_3$ be the third neighbor of $b_3$ other than $a_2$ and $a_4$ and note that $a'_3\in\A$. If there is an $\A$-path in $G$ from $a_2$ to $a'_3$ then there is one from $a_0$ to $a'_3$. Hence we can fix $\A'=\A\setminus\{a_2\}\cup\{b_1,b_3,a_4\}$ and $\I'=\I\setminus\{b_3\}\cup\{a_2\}$. If there is no $\A$-path in $G$ from $a_2$ to $a'_3$ then there is one from $a'_3$ to $b_5$, because $\A$ must be connected in $H$. Hence we can fix $\A'=\A\cup\{b_3\}$ and $\I'=\I\setminus\{b_3\}\cup\{b_1,a_4\}$.
 \end{itemize} \qedhere
 \end{itemize} 
 \end{itemize} 
\end{proof}



\begin{lemma}\label{lem:C6with_deg2vx}
    Let $b_1a_2b_3a_4b_5a'_1$ be a $6$-face in $G$, all of whose vertices have degree $3$ except from $a'_1.$ Then a \CAI-partition $(\A,\I)$ of $H=G-\{a'_1\}$, assuming $H\in\F$, satisfies $\{a_2,b_3,a_4\}\subset\A$ and $\{b_1,b_5\} \subset \I$. 
\end{lemma}
    
\begin{proof}
    See~\Cref{fig:2C6}. We take $H=G-\{a'_1\}$ and since $H\in\F$, we can consider a \CAI-partition $(\A,\I)$ of $H$. Observe that if $\{b_1,b_5\}\subset\A$, then $(\A,\I\cup\{a'_1\})$ is a \CAI-partition of $G$. Also, if $|\{b_1,b_5\}\cap\A|=1$, then $(\A\cup\{a'_1\},\I)$ is a \CAI-partition of $G$. Hence $\{b_1,b_5\}\subset\I$ and therefore $\{a_0,a_2,a_4,a_6\}\subset\A$.
We claim that $b_3\in\A$. Indeed, if $b_3\in\I$ then necessarily $\{b'_2,b'_4\}\subset\A$. Therefore $((\A\setminus\{a_2,a_4\})\cup\{b_1,b_3,b_5\},(\I\setminus\{b_1,b_3,b_5\})\cup\{a_2,a_4,a'_1\})$ is a \CAI-partition of $G$. 
\end{proof}

\begin{figure}[H]
 \centering
 \begin{tikzpicture}
 \tikzset{VertexStyle/.append style={minimum size=0.6cm, inner sep=0.01cm, draw}}

\Vertex[style={minimum size=1.0cm,shape=circle},LabelOut=false,L=\hbox{$b_5$},x=2.5cm,y=4cm]{v0}
\Vertex[style={minimum size=1cm,shape=circle},LabelOut=false,L=\hbox{$a_6$},x=0.5,y=4.6454cm]{v1}
\Vertex[style={minimum size=1.0cm,shape=circle},LabelOut=false,L=\hbox{$a_0$},x=0.5,y=0cm]{v2}
\Vertex[style={minimum size=1.0cm,shape=circle},LabelOut=false,L=\hbox{$b_1$},x=2.5cm,y=0.75cm]{v3}
\Vertex[style={minimum size=1.0cm,shape=circle},LabelOut=false,L=\hbox{$a_2$},x=5cm,y=0.75cm]{v4}
\Vertex[style={minimum size=1.0cm,shape=circle},LabelOut=false,L=\hbox{$a_4$},x=5cm,y=4cm]{v5}
\Vertex[style={minimum size=1.0cm,shape=circle},LabelOut=false,L=\hbox{$a'_1$},x=2.5cm,y=2.5cm]{v6}
\Vertex[style={minimum size=1.0cm,shape=circle},LabelOut=false,L=\hbox{$b_3$},x=5cm,y=2.5cm]{v8}
\Vertex[style={minimum size=1.0cm,shape=circle},LabelOut=false,L=\hbox{$a'_3$},x=7cm,y=2.5cm]{v88}
\Vertex[style={minimum size=1.0cm,shape=circle},LabelOut=false,L=\hbox{$b'_2$},x=7cm,y=0.0cm]{v11}
\Vertex[style={minimum size=1.0cm,shape=circle},LabelOut=false,L=\hbox{$b'_4$},x=7cm,y=4.6454cm]{v12}
\Edge[lw=0.03cm,](v0)(v1)
\Edge[lw=0.03cm,](v0)(v5)
\Edge[lw=0.03cm,](v0)(v6)
\Edge[lw=0.03cm,](v2)(v3)
\Edge[lw=0.03cm,](v3)(v4)
\Edge[lw=0.03cm,](v3)(v6)
\Edge[lw=0.03cm,](v4)(v8)
\Edge[lw=0.03cm,](v4)(v11)
\Edge[lw=0.03cm,](v5)(v8)
\Edge[lw=0.03cm,](v5)(v12)
\Edge[lw=0.03cm,](v8)(v88)

\end{tikzpicture}
 \caption{A 2-vertex incident to a 6-face.}
 \label{fig:2C6}
\end{figure}
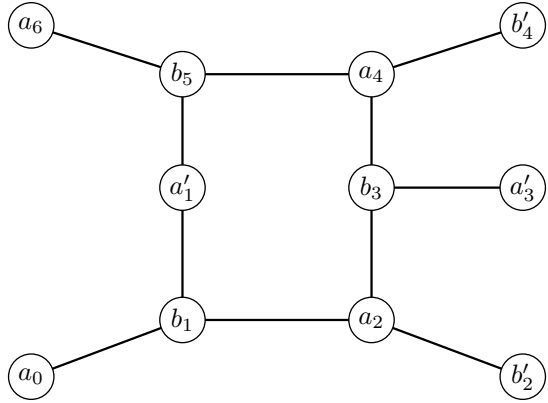

\begin{lemma}\label{lem:2C6}
A $2$-vertex cannot be incident to two $6$-faces in $G$.
\end{lemma}

\begin{proof}
    Let $b_1a_2b_3a_4b_5a'_1$ and $b_1a'_2b'_3a'_4b_5a'_1$ be the two $6$-faces in $G$, with $\deg(a'_1)=2$ and the other vertices having degree $3$ by \Cref{lem:23b}.
    Then, by \Cref{lem:C6with_deg2vx}, a \CAI-partition $(\A,\I)$ of $H=G-\{a'_1\}$ satisfies $\{a_2,b_3,a_4,a'_2,b'_3,a'_4\}\subset\A$ and $\{b_1,b_5\} \subset \I.$ We define $b'_2$, $a'_3$, and $b'_4$ as in \Cref{fig:2C6}.
    If we put $b_1$ in $\A$, $\A$ is not acyclic anymore.
    Consider the first edge $e$ among $a_2b'_2, b_3a'_3, a_4b'_4$ (in this order) for which a cycle in $\A \cup \{b_1\}$ exists using that edge.
    Let $x= e \cap \{a_2,b_3,a_4\}.$
    By the choice of $e$ and thus $x$, $\A \cup \{b_1\} \setminus x$ will be acyclic.
    If $x=a_4$ or $\A \cup \{b_1\} \setminus x$ is not connected (and thus $b'_4,a_4$ are not in the same connected component of $\A \cup \{b_1\} \setminus x$ as $b_1$), adding $b_5$ to $\A$ does not create a cycle.
    That is, $(\A \cup \{b_1, b_5\} \setminus x, \I \cup \{x,a'_1\} \setminus \{b_1, b_5\})$ is a $\CAI$-partition of $G$. 
    If $\A \cup \{b_1\} \setminus x$ is connected and $x\neq a_4$, then $(\A \cup \{b_1, a'_1\} \setminus x, \I \cup x \setminus b_1)$ is a $\CAI$-partition of $G.$ 
\end{proof}

\begin{lemma}\label{lem:C8}
If an $8$-face contains two $2$-vertices, then none of them is bad.
\end{lemma}

\begin{proof}
    Assume not. 
    Let $a_0b_1a_2b_3a_4b_5a_6b_7$ be an $8$-face containing two $2$-vertices, without loss of generality $a_2$ and $a_6$ (using~\Cref{lem:23b}) and assume that $a_2$ is bad, i.e. is also incident to a $6$-face $b_1a'_1b_2a'_3b_3a_2$. See \Cref{fig:C8C6}, for an illustration.

By~\Cref{lem:C6with_deg2vx},
a \CAI-partition $(\A, \I)$ of $H=G \setminus \{a_2\}$ (which belongs to $\F$) satisfies $\{a_0, a'_1,b_2,a'_3, a_4\} \subset \A$
and $\{b_1, b_3\} \subset \I.$
If there is an $\A$-path between $a_0$ and $a_4$, containing no vertex from $\{a'_1, b_2, a'_3\}$, we are done analogously as in the proof of~\Cref{lem:2C6}. 

By the previous and the definition of $(\A, \I)$, there is exactly one of $\{b_5, a_6, b_7\}$ belonging to $\I.$

If $b_5 \in \I$ (the case $b_7 \in \I$ is analogous), we can consider $(\A',\I')=(\A \cup \{a_2,b_3,b_5\} \setminus \{a_4\}, \I \setminus \{b_3,b_5\} \cup \{a_4\}). $
Here $\A'$ is connected and $\I'$ is an independent.
If $\A'$ contains a cycle, we can put $a_6$ in $\I'$,
i.e. either $(\A',\I')$ or $(\A' \setminus a_6,\I' \cup a_6)$ is a \CAI-partition of $G.$

Finally, we can assume that $a_6 \in \I$ and $b_5, b_7 \in \A,$ and recall that every $\A$-path from $a_0$ to $a_4$ uses at least two vertices out of $\{a'_1, b_2, a'_3\}$.
By planarity, this implies that there is an $\A$-path from $a_4$ to $a'_3$ avoiding $b_2$, or from $a_0$ to $a'_1$ avoiding $b_2$ (possibly both).
By symmetry, we can assume the first.
Now choose $(\A',\I')=(\A \cup \{b_3,a_2\} \setminus \{a'_3\}, \I \setminus \{b_3\} \cup \{a'_3\}). $
Now $(\A',\I')$ or $(\A' \cup a_6,\I' \setminus a_6)$ is a \CAI-partition of $G.$
\end{proof}

    \begin{figure}[H]
 \centering
\begin{tikzpicture}
 \tikzset{VertexStyle/.append style={minimum size=0.6cm, inner sep=0.01cm, draw}}

\Vertex[style={minimum size=1.0cm,shape=circle},LabelOut=false,L=\hbox{$a_0$},x=0.,y=0cm]{a0}
\Vertex[style={minimum size=1.0cm,shape=circle},LabelOut=false,L=\hbox{$b_1$},x=2cm,y=0cm]{b1}
\Vertex[style={minimum size=1.0cm,shape=circle},LabelOut=false,L=\hbox{$a_2$},x=2cm,y=2cm]{a2}
\Vertex[style={minimum size=1.0cm,shape=circle},LabelOut=false,L=\hbox{$b_3$},x=2.cm,y=4cm]{b3}

\Vertex[style={minimum size=1.0cm,shape=circle},LabelOut=false,L=\hbox{$a_4$},x=0cm,y=4cm]{a4}

\Vertex[style={minimum size=1.0cm,shape=circle},LabelOut=false,L=\hbox{$b_5$},x=-2.cm,y=4cm]{b5}
\Vertex[style={minimum size=1.0cm,shape=circle},LabelOut=false,L=\hbox{$a_6$},x=-2cm,y=2cm]{a6}

\Vertex[style={minimum size=1.0cm,shape=circle},LabelOut=false,L=\hbox{$b_7$},x=-2.,y=0cm]{b7}

\Vertex[style={minimum size=1.0cm,shape=circle},LabelOut=false,L=\hbox{$b'_4$},x=0cm,y=5cm]{b4}
\Vertex[style={minimum size=1.0cm,shape=circle},LabelOut=false,L=\hbox{$b'_0$},x=0.,y=-1cm]{b0}

\Vertex[style={minimum size=1.0cm,shape=circle},LabelOut=false,L=\hbox{$a'_5$},x=-3cm,y=5cm]{a5}

\Vertex[style={minimum size=1.0cm,shape=circle},LabelOut=false,L=\hbox{$a'_7$},x=-3.,y=-1cm]{a7}

\Vertex[style={minimum size=1.0cm,shape=circle},LabelOut=false,L=\hbox{$a'_1$},x=4cm,y=0cm]{a1}
\Vertex[style={minimum size=1.0cm,shape=circle},LabelOut=false,L=\hbox{$b_2$},x=4cm,y=2cm]{b2}
\Vertex[style={minimum size=1.0cm,shape=circle},LabelOut=false,L=\hbox{$a'_3$},x=4.cm,y=4cm]{a3}

\draw (a0)--(b1)--(a2)--(b3)--(a4)--(b5)--(a6)--(b7)--(a0)--(b1)--(a1)--(b2)--(a3)--(b3);
\draw (b7)--(a7);
\draw (a0)--(b0);
\draw (a4)--(b4);
\draw (a5)--(b5);

\end{tikzpicture}
 \caption{Configuration of~\Cref{lem:C8}}
 \label{fig:C8C6}
\end{figure}

\section{Proof of \texorpdfstring{\Cref{thm:tw2}}{Theorem 9}}
It is well-known that series-parallel graphs contain no subdivisions of a $K_4$, see e.g.~\cite{Bod98}.

Given an undirected graph $G=(V,E)$, a partition of its edges into a sequence of \emph{ears} $\mathrm{ED}=(E_0, \ldots, E_{\ell})$ is an \emph{open ear decomposition} (starting in $E_0$) if:
\begin{itemize} 
 \item[0.] $E_0$ is a cycle,
 \item[1.] $E_i$ is a path with endpoints $x_i,y_i$, for $1\leq i\leq \ell$,
 \item[2.] the internal vertices of $E_j$ do not appear in $E_i$ with $i<j$, but the endpoints $x_j,y_j$ appear in some $E_k$ and $E_m$, for $0\leq k,m<j\leq \ell$.
\end{itemize}
Further, $\mathrm{ED}$ is \emph{nested} if 
\begin{itemize}
 \item[3.] the endpoints $x_j,y_j$ of $E_j$ are interior vertices of exactly one ear $E_i$, for $0\leq i<j\leq \ell$. We call the $(x_j,y_j)$-subpath of $E_i$ the \emph{nest interval of $E_j$ on $E_i$},
 \item[4.] if $E_j$ and $E_k$ both have their endpoints on $E_i$, then their nest intervals on $E_i$ are contained in each other or are internally disjoint.
\end{itemize}
Additionally, $\mathrm{ED}$ is \emph{short} if
\begin{itemize}
 \item[5.] each $E_j$ is induced and the nest interval of $E_j$ on $E_i$ is not longer than {the path $E_j$}.
\end{itemize}

A classic result of Whitney~\cite{Whi32} shows that a $2$-vertex-connected graph on at least $3$ vertices has an open ear-decomposition. This has been adapted by Eppstein~\cite{Epp92} who showed that a $2$-vertex-connected graph is series-parallel if and only if it admits a nested open ear decomposition. We will show the following nice little lemma:
\begin{lemma}\label{lem:ED}
 If $G$ is a $2$-vertex-connected series-parallel then it has a short nested open ear decomposition.
\end{lemma}
\begin{proof}
Since $G$ is $2$-vertex-connected it has a cycle, take a shortest one and use it as $E_0$. Given a partial short nested open ear decomposition $\mathrm{ED}'=(E_0, \ldots, E_i)$ covering a subgraph $H\subset G$, pick any two vertices $x,y$ of $H$ such that they are connected with a path only using edges from $G\setminus E(H)$ and take a shortest such path $E_{i+1}$. To see that such $x,y$ exist is as usual: If there is a vertex $z\in G\setminus V(H)$ and since $G$ is $2$-vertex connected there must be two paths from $z$ to $H$ that only intersect in $z$. Their two endpoints are $x,y$. Otherwise any edge $E_{i+1}=\{x,y\}$ of $G\setminus E(H)$ will do. This yields an open ear decomposition.

Suppose that $E_k$ is the first ear that does not satisfying 3. or 4. Hence every prior ear has a unique predecessor. If $E_k$ violates 3., then it has endpoints as interior vertices $x_k\in E_j$ and $y_k\in E_i$ for $i\neq j$. Note that every vertex is an interior point of some ear, so the endpoints of $E_k$ must be interior of at least one ear. Let $E_{i\wedge j}$ be the first common predecessor ear of $E_i$ and $E_j$, in both cases $E_{i\wedge j}\in \{E_i, E_j\}$ and $E_{i\wedge j}\notin \{E_i, E_j\}$ it is easy to construct a $K_4$-minor, see the left two cases in Figure~\ref{fig:tw2cases}.

If $E_k$ violates 4., then there are $E_i, E_j$ such that the nest intervals of $E_k$ and $E_j$ on $E_i$ properly overlap. Also in this case it is easy to find a $K_4$-minor, see the right case in Figure~\ref{fig:tw2cases}. 

\begin{figure}
 \centering
 \includegraphics[width=.7\textwidth]{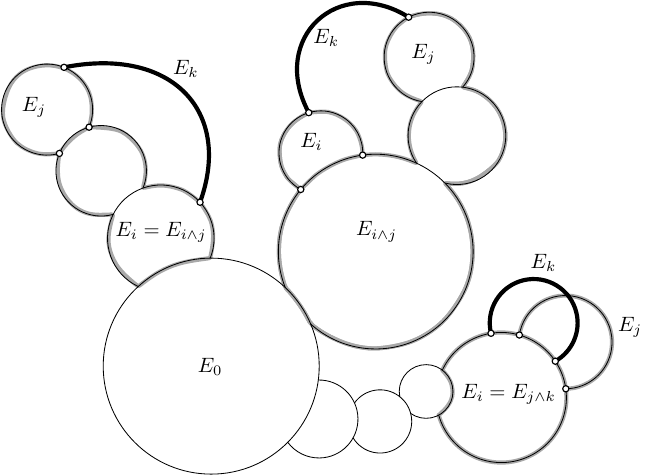}
 \caption{The three ways an ear $E_k$ can violate properties 3. or 4. and the resulting $K_4$-minors in grey.}
 \label{fig:tw2cases}
\end{figure}

Let us now prove 5.
First, note that by the choice of $E_k$ as shortest path (or cycle), it clearly is induced. Suppose now that the nest interval $I$ of $E_k$ on its unique predecessor ear $E_j$ is longer than $E_k$. But then at the time of constructing $E_j$ the shorter path $(E_j\setminus I)\cup E_i$ would have been available, contradicting the minimality in the choice of $E_j$.
\end{proof}

Khuller~\cite{Khu89} proposed the definition of \emph{tree ear decomposition}, which are those open ear decompositions additionally satisfying 3. We will call a tree ear decomposition \emph{short} if it furthermore satisfies 5. Clearly, short open nested ear decompositions are short tree ear decompositions. Hence, 
together with \Cref{lem:ED} the following yields~\Cref{thm:tw2}.
\begin{lemma}
 If $G$ is simple and has a short tree ear decomposition, then it has a \CAI-partition.
\end{lemma}
\begin{proof}
To prove the theorem go along a short tree ear decomposition $\mathrm{ED}$ of $G$ and construct a \CAI-partition with the property that $\mathcal{I}$ has at most one vertex on each ear. This is easy for $E_0$ by putting an arbitrary vertex of it into $\mathcal{I}$. Note that by 5. every $E_i$ has some interior vertex, because otherwise its nest interval must also have been an edge, contradicting simplicity. When $E_i$ is added, then by property 3. at most one of its endpoints is in $\mathcal{I}$. If it is exactly one, then just add the vertices of $E_i$ to $\mathcal{A}$. Otherwise choose an internal vertex of $E_i$ neighboring an endpoint of $E_i$ and add it to $\mathcal{I}$. Clearly, in both cases we maintain that $\mathcal{I}$ is independent and has at most one vertex on every ear. Moreover, in both cases we add one induced subpath of $E_i$ which is induced by 5. to $\mathcal{A}$. If there was an edge induced from a vertex of $E_i$ to some previous vertex in $\mathcal{A}$, then this must be a later ear, contradicting that the ears in a short tree ear decomposition are not edges. 
\end{proof}

We do not know if there are any interesting graphs apart from the series-parallel ones, that admit short tree ear decompositions. One source is to take a graph with a tree ear decomposition, e.g., any Hamiltonian graph, and subdivide edges sufficiently often so property 5. is satisfied.

\section{Tightness of \texorpdfstring{\Cref{thm:subcubic,thm:tw2}}{Theorems 8 and 9}}\label{s}
We discuss the tightness of the results obtained above. 

\begin{lemma}\label{lem:no-AI-graphs}
 Each of the graphs of \Cref{fig:tight-graphs} has no \CAI-partition.
\end{lemma}
\begin{proof}
We provide the proof for each figure separately.
\begin{itemize}
 \item[(a)] To show that the graph of \Cref{subfig:no-AI-oriented-planar-bipartite-2-conn} has no \CAI-partition we show some properties of the left (resp. right) part of the figure induced by vertices $\{0,\ldots,7\}$ (resp. $\{0',\ldots,7'\}$). More precisely, we show that for any \CAI-partition of the left part, vertices $\{1,2\}\not\subset\A$. Indeed, suppose that there is a \CAI-partition such that $\{1,2\}\subset\A$. Since $0,3,2,1$ is a directed cycle, either $0\in\I$ or $3\in\I$. If $0\in\I$, then $3\in\A$ and since $1,4,3,2$ is a directed cycle, we conclude that $4\in\I$. But then vertex 6 has both neighbors in $\I$ and at the same time $6\in\A$ which contradicts the connectivity of $\A$. If $3\in\I$, then $4\in\A$ and since $1,4,5,2$ is a directed cycle, we conclude that $5\in\I$. But then vertex 7 has both neighbors in $\I$ and at the same time $7\in\A$ which contradicts the connectivity of $\A$.

 Therefore, for any \CAI-partition of the left (resp. right) part, either vertex 1 or 2 (resp. $1'$ or $2'$) must be in $\I$. Hence we obtain a contradiction because $\A$ is not connected.
 
 \item[(b)] 
 Let $(\A, \I)$ be a \CAI-partition of the hypercube.
 If $|\I|\le 2$, then $\A$ cannot be acyclic. 
 But if $|\I|\geq 3$, then since $\I$ is independent, there would be an isolated vertex in $\A$ contradicting the connectivity of $\A$. 
 Thus no \CAI-partition of the hypercube exists.
 \item[(c),(d)] The proofs for \Cref{subfig:no-AI-digons,subfig:no-AI-1-connected} are straightforward.
 \item[(e)] Observe that for every directed triangle of \Cref{subfig:truncated_tetrahedron}, exactly one vertex must be in $\I$. Since the graph is symmetric, it is easy to observe that for any choice of these four vertices in $\I$, the other vertices form a disconnected graph. \qedhere
\end{itemize}
\end{proof}

\Cref{thm:tw2} is best possible in the sense that removing any of the restrictions on the graph class provides a counter-example. For instance, the graph of \Cref{subfig:hypercube} is $2$-vertex-connected but of treewidth 3, hence just above $2$-vertex-connected series-parallel, which coincides with $2$-vertex-connected and treewidth $2$. The graph of \Cref{subfig:no-AI-1-connected} has treewidth 2 but is not $2$-vertex-connected.

As of \Cref{thm:subcubic}, we provide a counterexample whenever one of the following restrictions is removed : maximum degree 3 (\Cref{subfig:no-AI-oriented-planar-bipartite-2-conn}), oriented (\Cref{subfig:hypercube,subfig:no-AI-digons}), $2$-vertex-connected (\Cref{subfig:no-AI-1-connected}), bipartite (\Cref{subfig:truncated_tetrahedron}).

\begin{figure}[h]
\centering

\begin{subfigure}[b]{0.67\textwidth}
 \centering
 \begin{tikzpicture}[scale=0.85]
 \tikzset{->-/.style={
 decoration={
 markings,
 mark=at position #1 with {\arrow{>}}
 },
 postaction={decorate}
 },
 ->-/.default=0.55 
 }
 
 \tikzset{VertexStyle/.append style={minimum size=0.57cm, inner sep=0.01cm}}
 
 \tikzset{EdgeStyle/.style = {->-,>=stealth'}}

 \Vertex[style={minimum size=1.0cm,shape=circle},LabelOut=false,L=\hbox{$0$},x=0cm,y=5cm]{v0}
 \Vertex[style={minimum size=1.0cm,shape=circle},LabelOut=false,L=\hbox{$1$},x=5cm,y=5cm]{v1}
 \Vertex[style={minimum size=1.0cm,shape=circle},LabelOut=false,L=\hbox{$2$},x=5cm,y=0cm]{v2}
 \Vertex[style={minimum size=1.0cm,shape=circle},LabelOut=false,L=\hbox{$3$},x=0cm,y=0cm]{v3}
 \Vertex[style={minimum size=1.0cm,shape=circle},LabelOut=false,L=\hbox{$4$},x=2.5cm,y=2.513cm]{v4}
 \Vertex[style={minimum size=1.0cm,shape=circle},LabelOut=false,L=\hbox{$5$},x=3.75cm,y=1.25cm]{v5}
 \Vertex[style={minimum size=1.0cm,shape=circle},LabelOut=false,L=\hbox{$6$},x=1.2cm,y=3.8cm]{v6}
 \Vertex[style={minimum size=1.0cm,shape=circle},LabelOut=false,L=\hbox{$7$},x=2.5cm,y=0.8cm]{v7}

 \foreach \from/\to in {0/3, 1/0, 1/4, 2/1, 3/2, 4/3, 4/5, 5/2}
 \Edge(v\from)(v\to);
 
 \draw[-,line width=0.8mm] (v0)--(v6)--(v4) (v3)--(v7)--(v5);
 
 \Vertex[style={minimum size=1.0cm,shape=circle},LabelOut=false,L=\hbox{$0'$},x=13cm,y=5cm]{v00}
 \Vertex[style={minimum size=1.0cm,shape=circle},LabelOut=false,L=\hbox{$1'$},x=8cm,y=5cm]{v11}
 \Vertex[style={minimum size=1.0cm,shape=circle},LabelOut=false,L=\hbox{$2'$},x=8cm,y=0cm]{v22}
 \Vertex[style={minimum size=1.0cm,shape=circle},LabelOut=false,L=\hbox{$3'$},x=13cm,y=0cm]{v33}
 \Vertex[style={minimum size=1.0cm,shape=circle},LabelOut=false,L=\hbox{$4'$},x=10.5cm,y=2.513cm]{v44}
 \Vertex[style={minimum size=1.0cm,shape=circle},LabelOut=false,L=\hbox{$5'$},x=9.25cm,y=1.25cm]{v55}
 \Vertex[style={minimum size=1.0cm,shape=circle},LabelOut=false,L=\hbox{$6'$},x=11.7cm,y=3.8cm]{v66}
 \Vertex[style={minimum size=1.0cm,shape=circle},LabelOut=false,L=\hbox{$7'$},x=10.5cm,y=0.8cm]{v77}

 \foreach \from/\to in {00/33, 11/00, 11/44, 22/11, 33/22, 44/33, 44/55, 55/22}
 \Edge(v\to)(v\from);
 \draw[-,line width=0.8mm] (v00)--(v66)--(v44) (v33)--(v77)--(v55);

 \Edge(v22)(v2);
 \Edge(v1)(v11);
 
 \end{tikzpicture}
\caption{Oriented planar bipartite $2$-vertex-connected (with maximum degree 4).\\The bold edges can be oriented arbitrarily.\label{subfig:no-AI-oriented-planar-bipartite-2-conn}}
\end{subfigure}
\hfill
\begin{subfigure}[b]{0.3\textwidth}
 \centering
 \begin{tikzpicture}[join=bevel,inner sep=0.5mm, scale=1.5]
 \tikzstyle{graphnode}=[draw,shape=circle,draw=black,minimum size=0.1cm,inner sep=0.07cm]
 \node[graphnode](000) at (0,0) {};
 \node[graphnode](001) at (0,1) {};
 \node[graphnode](010) at (1,0) {};
 \node[graphnode](011) at (1,1) {};

 \node[graphnode](100) at (-1,-1) {};
 \node[graphnode](101) at (-1,2) {};
 \node[graphnode](110) at (2,-1) {};
 \node[graphnode](111) at (2,2) {};

 \draw[-] (000)--(001)--(011)--(010)--(000) (100)--(101)--(111)--(110)--(100);
 \foreach \I in {00,01,10,11}
 \draw[-] (1\I)--(0\I);
 \end{tikzpicture}
 
\caption{Planar bipartite $2$-vertex-connected\\subcubic (undirected).\label{subfig:hypercube}}
\end{subfigure}

\begin{subfigure}[b]{0.2\textwidth}
\centering
\begin{tikzpicture}

\tikzset{->-/.style={
 decoration={
 markings,
 mark=at position #1 with {\arrow{>}}
 },
 postaction={decorate}
 },
 ->-/.default=0.55 
}

\tikzset{VertexStyle/.append style={minimum size=0.1cm, inner sep=0.07cm}}

\tikzset{EdgeStyle/.style = {->-,>=stealth'}}

\Vertex[NoLabel,style={minimum size=0.4cm,shape=circle},x=1cm,y=4cm]{v2}
\Vertex[NoLabel,style={minimum size=0.4cm,shape=circle},x=1cm,y=2cm]{v3}
\Vertex[NoLabel,style={minimum size=0.4cm,shape=circle},x=3cm,y=2cm]{v4}
\Vertex[NoLabel,style={minimum size=0.4cm,shape=circle},x=3cm,y=4cm]{v5}
\Edge(v3)(v2)
\Edge(v5)(v4)

\tikzset{EdgeStyle/.append style = {bend left}}

\Edge(v2)(v5)
\Edge(v4)(v3)

\tikzset{EdgeStyle/.append style = {bend left}}
\Edge(v5)(v2)
\Edge(v3)(v4)
\end{tikzpicture}

 \caption{Planar bipartite \\$2$-vertex-connected subcubic \\(non-oriented).\label{subfig:no-AI-digons}}
\end{subfigure}
\qquad
\begin{subfigure}[b]{0.25\textwidth}
\centering
\begin{tikzpicture}[scale=0.8]

\tikzset{->-/.style={
 decoration={
 markings,
 mark=at position #1 with {\arrow{>}}
 },
 postaction={decorate}
 },
 ->-/.default=0.55 
}

\tikzset{VertexStyle/.append style={minimum size=0.1cm, inner sep=0.07cm}}

\tikzset{EdgeStyle/.style = {->-,>=stealth'}}

\Vertex[NoLabel,style={minimum size=0.4cm,shape=circle},x=0.0cm,y=5cm]{v0}
\Vertex[NoLabel,style={minimum size=0.4cm,shape=circle},x=4.0cm,y=5cm]{v1}
\Vertex[NoLabel,style={minimum size=0.4cm,shape=circle},x=1cm,y=4cm]{v2}
\Vertex[NoLabel,style={minimum size=0.4cm,shape=circle},x=1cm,y=2cm]{v3}
\Vertex[NoLabel,style={minimum size=0.4cm,shape=circle},x=3cm,y=2cm]{v4}
\Vertex[NoLabel,style={minimum size=0.4cm,shape=circle},x=3cm,y=4cm]{v5}
\Vertex[NoLabel,style={minimum size=0.4cm,shape=circle},x=0cm,y=1cm]{v8}
\Vertex[NoLabel,style={minimum size=0.4cm,shape=circle},x=4cm,y=1cm]{v9}
\Edge(v3)(v2)
\Edge(v2)(v5)
\Edge(v2)(v0)
\Edge(v4)(v3)
\Edge(v5)(v4)
\Edge(v5)(v1)
\Edge(v4)(v9)
\Edge(v3)(v8)
\end{tikzpicture}

 \caption{Oriented planar bipartite subcubic (not $2$-vertex-connected).\label{subfig:no-AI-1-connected}}
\end{subfigure}
\qquad
\begin{subfigure}[b]{0.45\textwidth}
\centering
\begin{tikzpicture}[scale=0.6]

\tikzset{->-/.style={
 decoration={
 markings,
 mark=at position #1 with {\arrow{>}}
 },
 postaction={decorate}
 },
 ->-/.default=0.55 
}

\tikzset{VertexStyle/.append style={minimum size=0.1cm, inner sep=0.1cm}}
\tikzset{EdgeStyle/.style = {->-,>=stealth'}}

\Vertex[NoLabel,style={minimum size=1.0cm,shape=circle},x=0cm,y=0cm]{A}
\Vertex[NoLabel,style={minimum size=1.0cm,shape=circle},x=5cm,y=0cm]{B}
\Vertex[NoLabel,style={minimum size=1.0cm,shape=circle},x=7.5cm,y=4.33cm]{C}
\Vertex[NoLabel,style={minimum size=1.0cm,shape=circle},x=5cm,y=8.66cm]{D}
\Vertex[NoLabel,style={minimum size=1.0cm,shape=circle},x=0cm,y=8.66cm]{E}
\Vertex[NoLabel,style={minimum size=1.0cm,shape=circle},x=-2.5cm,y=4.33cm]{F}
\Vertex[NoLabel,style={minimum size=1.0cm,shape=circle},x=2.5cm,y=2.98cm]{H}
\Vertex[NoLabel,style={minimum size=1.0cm,shape=circle},x=3.67cm,y=5cm]{I}
\Vertex[NoLabel,style={minimum size=1.0cm,shape=circle},x=1.33cm,y=5cm]{J}
\Vertex[NoLabel,style={minimum size=1.0cm,shape=circle},x=2.5cm,y=1.83cm]{K}
\Vertex[NoLabel,style={minimum size=1.0cm,shape=circle},x=0.33cm,y=5.58cm]{L}
\Vertex[NoLabel,style={minimum size=1.0cm,shape=circle},x=4.67cm,y=5.58cm]{M}
\Edge(A)(K)
\Edge(K)(B)
\Edge(B)(A)
\Edge(M)(D)
\Edge(D)(C)
\Edge(C)(M)
\Edge(J)(I)
\Edge(I)(H)
\Edge(H)(J)
\Edge(E)(L)
\Edge(L)(F)
\Edge(F)(E)

\draw[-,line width=0.8mm] (E)--(D) (L)--(J) (I)--(M) (H)--(K) (C)--(B) (F)--(A);
\end{tikzpicture}

\caption{Oriented planar $2$-vertex-connected subcubic (not bipartite). The bold edges can be oriented arbitrarily.\label{subfig:truncated_tetrahedron}}
\end{subfigure}
\caption{Graphs with no \CAI-partition.\label{fig:tight-graphs}}
\end{figure}
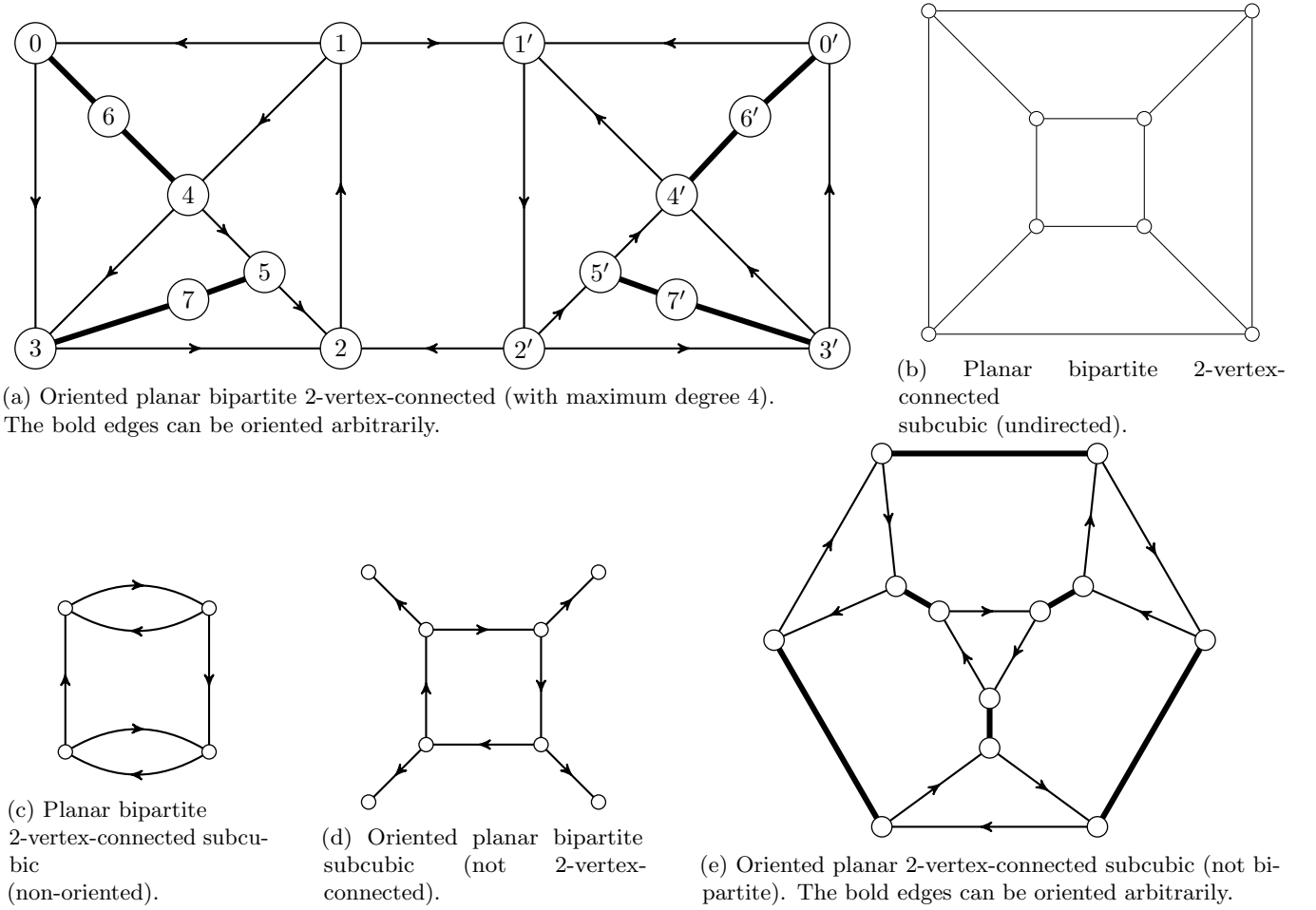
\section{Proofs of \texorpdfstring{\Cref{thm:finalboss}}{Theorem 10} and \texorpdfstring{\Cref{cor:finalboss}}{Corollary 11}} \label{sec:neg}
Note that with the specific properties of a partition resulting from \Cref{obs:main} the following show that this strategy will not resolve~\Cref{conj:EulerianNL} or \Cref{conj:Barnette}, i.e., it yields \Cref{thm:finalboss}.

\begin{theorem} \label{thm:oriented-triang}
There exists a Eulerian oriented planar triangulation $G$ with tripartition $I_1,I_2,I_3$, such that every partition of $G$ into two acyclic sets $A_1,A_2$ has $I_i\not\subseteq A_j$ for all $i\in\{1,2,3\}$ and $j\in\{1,2\}$.
\end{theorem}

In order to build the graph of \Cref{thm:oriented-triang}, we first provide two useful gadgets.

\begin{lemma}\label{lem:gadgets}
 Let $G_1(0,1,2,3)$ and $G_2(1,2,13)$ be the oriented triangulations of \Cref{fig:gadget1,fig:gadget2}. We have the following properties :
 \begin{enumerate}
 \item $\forall i\in\{4,\ldots,12\}, d^+_{G_1}(i)=d^-_{G_1}(i), d^+_{G_2}(i)=d^-_{G_2}(i)$.
 \item $d^+_{G_1}(0)=3$, $d^-_{G_1}(0)=2$.
 \item $d^+_{G_1}(1)=3$, $d^-_{G_1}(1)=2$.
 \item $d^+_{G_1}(2)=2$, $d^-_{G_1}(2)=3$.
 \item $d^+_{G_1}(3)=3$, $d^-_{G_1}(3)=4$.
 \item $d^+_{G_2}(1)=4$, $d^-_{G_2}(1)=2$.
 \item $d^+_{G_2}(2)=3$, $d^-_{G_2}(2)=3$.
 \item $d^+_{G_2}(13)=1$, $d^-_{G_2}(13)=3$.
 \item For every partition of $G_1(0,1,2,3)$ into two acyclic sets $\A_1$ and $\A_2$, if $\{1,2\}\subset\A_1$, then $\{8,9,10,11,12\}\not\subset\A_2$.
 \item For every partition of $G_2(1,2,13)$ into two acyclic sets $\A_1$ and $\A_2$, if $\{1,2\}\subset\A_1$, then $\{8,9,10,11,12,13\}\not\subset\A_2$.
 \end{enumerate}
\end{lemma}

\begin{proof}
 The first eight items can be easily checked on \Cref{fig:gadget1,fig:gadget2}.

 To prove item 9, we proceed by contradiction. Consider a vertex-partition of $G_1(0,1,2,3)$ into two acyclic sets $\A_1$ and $\A_2$ such that $\{1,2\}\subset\A_1$ and $\{8,9,10,11,12\}\subset\A_2$. We have the two following cases:
 \begin{itemize}
 \item Suppose $0\in\A_2$. Since $0,11,4,12$ induce a directed cycle, we know that $4\in\A_1$. But then since $1,4,5,2$ induce a directed cycle, we know that $5\in\A_2$. Therefore, since $3,8,5,9$ induce a directed cycle, we know that $3\in\A_1$. This is a contradiction because $\A_1$ contains the directed cycle $1,4,3,2$.
 \item Suppose $0\in\A_1$. Since $0,3,2,1$ induce a directed cycle, we know that $3\in\A_2$. Similarly to the previous paragraph we conclude that $\{4,5\}\subset\A_1$. This is a contradiction because $\A_1$ contains the directed cycle $1,4,5,2$.
 \end{itemize}
 The proof of item 10 follows the same arguments.

 Consider a vertex-partition of $G_2(1,2,13)$ into two acyclic sets $\A_1$ and $\A_2$ such that $\{1,2\}\subset\A_1$ and $\{8,9,10,11,12,13\}\subset\A_2$. We have the two following cases:
  \begin{itemize}
 \item Suppose $0\in\A_2$. Since $0,3,13$ induce a directed cycle, we know that $3\in\A_1$. 
 Considering $\{1,2,3,4\}$ implies that $4 \in \A_2$, leading to $0,12,4,11$ inducing a directed cycle and thus $\A_2$ not being acyclic, contradiction.
 \item Suppose $0\in\A_1$. Since $0,3,2,1$ induce a directed cycle, we know that $3\in\A_2$. Similarly to the previous paragraph we conclude that $\{4,5\}\subset\A_1$. This is a contradiction because $\A_1$ contains the directed cycle $1,4,5,2$. \qedhere
 \end{itemize}
\end{proof}
\begin{proof}[Proof of \Cref{thm:oriented-triang}]
 We build $G$ by gluing the gadgets of \Cref{fig:gadget1,fig:gadget2} on a Eulerian orientation of the octahedron. See~\Cref{fig:eulerian-counterexample}. More precisely we have the following gadgets in $G$: 
 \begin{itemize}
 \item $G_1(v_6,v_3,v_4,v_7)$, $G_1(v_{10},v_4,v_5,v_{11})$, $G_1(v_{14},v_5,v_3,v_{15})$, $G_1(v_8,v_0,v_1,v_9)$, $G_1(v_{12},v_1,v_2,v_{13})$, $G_1(v_{16},v_2,v_0,v_{17})$,
 \item $G_2(v_3,v_0,v_6)$, $G_2(v_0,v_4,v_8)$, $G_2(v_4,v_1,v_{10})$, $G_2(v_1,v_5,v_{12})$, $G_2(v_5,v_2,v_{14})$, $G_2(v_2,v_3,v_{16})$.
 \end{itemize}

 Observe that $G$ is a triangulation. We show that $G$ is Eulerian, that is $d^+(v)=d^-(v)$ for every vertex $v$. By item 1 of \Cref{lem:gadgets}, we have $d^+(v)=d^-(v)$ for every internal vertex $v$ (which is not on the outerface of the gadgets). We show that $d^+(v_i)=d^-(v_i)$ for every $i\in\{0,\ldots,17\}$:

 \begin{itemize}
 \item For $i\in\{6,8,10,12,14,16\}$, by items 2 and 8 of \Cref{lem:gadgets}, we have $d^+(v_i)=d^+_{G_1}(0)+d^+_{G_2}(13)=3+1=4=2+3-1=d^-_{G_1}(0)+d^-_{G_2}(13)-1$.
 \item For $i\in\{7,9,11,13,15,17\}$, by item 5 of \Cref{lem:gadgets}, we have $d^+(v_i)=d^+_{G_1}(3)+1=3+1=4=d^-_{G_1}(3)$.
 \item For $i\in\{0,1,2,3,4,5\}$, by items 3, 4, 6, and 7 of \Cref{lem:gadgets}, we have $d^+(v_i)=d^+_{G_1}(1)+d^+_{G_2}(1)-1+d^+_{G_2}(2)+d^+_{G_1}(2)=3+4-1+3+2=11=2+2+1+3+3=d^-_{G_1}(1)+d^-_{G_2}(1)+1+d^-_{G_2}(2)+d^-_{G_1}(2)$.
 \end{itemize}

 Let $I_1,I_2,I_3$ be the tripartition of $G$. It remains to prove that for every partition of $G$ into two acyclic sets, none of these sets contains $I_j$ for every $j\in\{1,2,3\}$. W.l.o.g. let $\{v_0,v_5,v_9,v_{10},v_{15},v_{16}\}\subset I_1$, let $\{v_1,v_3,v_7,v_8,v_{13},v_{14}\}\subset I_2$, let $\{v_2,v_4,v_6,v_{11},v_{12},v_{17}\}\subset I_3$. Let $\A_1,\A_2$ be a vertex-partition of $G$ into two acyclic sets. By contradiction and by symmetry, we can assume that $I_1\subset\A_1$. Observe that the five internal vertices of $G_1(v_6,v_3,v_4,v_7)$ corresponding to $\{8,9,10,11,12\}$ in \Cref{fig:gadget1} must all be in $I_1$. Hence by item 9 of \Cref{lem:gadgets} applied to $G_1(v_6,v_3,v_4,v_7)$, we conclude that $\{v_3,v_4\}\not\subset\A_2$ and thus $\{v_3,v_4\}\cap\A_1\neq\emptyset$. Thus, we distinguish the two cases:
 \begin{itemize}
 \item Suppose $v_3\in\A_1$. Since $v_5\in\A_1$ and $v_3,v_5,v_4$ induce a directed triangle, we have $v_4\in\A_2$. Since $v_0,v_3,v_5,v_1$ induce a directed cycle, we know that vertex $v_1\in\A_2$. This is a contradiction with item 10 of \Cref{lem:gadgets} applied to $G_2(v_4,v_1,v_{10})$. Indeed, since the five vertices internal vertices of $G_2(v_4,v_1,v_{10})$ corresponding to $\{8,9,10,11,12,13\}$ in \Cref{fig:gadget2} all belong to $I_1$, by hypothesis we know that they all belong to $\A_1$. On the other hand, since $\{v_1,v_4\}\subset\A_2$, by item 10 of \Cref{lem:gadgets} we know that at least one of these five vertices must be in $\A_2$.
 \item Suppose $v_4\in\A_1$. The proof is very similar to the previous case, due to the symmetry of $G$. Since $v_5\in\A_1$ and $v_3,v_5,v_4$ induce a directed triangle, we have $v_3\in\A_2$. Since $v_0,v_2,v_5,v_4$ induce a directed cycle, vertex $v_2\in\A_2$. This is a contradiction with item 10 of \Cref{lem:gadgets} applied to $G_2(v_2,v_3,v_{16})$, because $v_{16}\in I_1$ but when $\{v_2,v_3\}\subset\A_2$ we know that $I_1\not\subset\A_1$. \qedhere
 \end{itemize}
\end{proof}

\begin{figure}[H]
\begin{subfigure}[b]{0.4\textwidth}
 \centering
 
 \begin{tikzpicture}[scale=1.1]
 \tikzset{->-/.style={
 decoration={
 markings,
 mark=at position #1 with {\arrow{>}}
 },
 postaction={decorate}
 },
 ->-/.default=0.55 
 }
 
 \tikzset{VertexStyle/.append style={minimum size=0.5cm, inner sep=0.01cm, draw}}
 
 \tikzset{EdgeStyle/.style = {->-,>=stealth'}}
 
 \Vertex[LabelOut=false,L=\hbox{$0$},x=0cm,y=5cm]{v0}
 \Vertex[LabelOut=false,L=\hbox{$3$},x=0cm,y=0cm]{v3}
 \Vertex[LabelOut=false,L=\hbox{$4$},x=2.5cm,y=2.513cm]{v4}
 \Vertex[LabelOut=false,L=\hbox{$5$},x=3.675cm,y=1.25cm]{v5}
 \Vertex[LabelOut=false,L=\hbox{$6$},x=1.2cm,y=3.8cm]{v6}
 \Vertex[LabelOut=false,L=\hbox{$7$},x=2.5cm,y=0.8cm]{v7}

 \tikzset{VertexStyle/.append style={minimum size=0.5cm, inner sep=0.01cm, color=blue!35, text=black}}
 \Vertex[LabelOut=false,L=\hbox{$1$},x=5cm,y=5cm]{v1}
 \Vertex[LabelOut=false,L=\hbox{$2$},x=5cm,y=0cm]{v2}
 
 \tikzset{VertexStyle/.append style={minimum size=0.5cm, inner sep=0.01cm, color=red!25, text=black}}
 
 \Vertex[LabelOut=false,L=\hbox{$12$},x=1.2cm,y=2.513cm]{v12}
 \Vertex[LabelOut=false,L=\hbox{$11$},x=2.5cm,y=3.8cm]{v11}
 \Vertex[LabelOut=false,L=\hbox{$10$},x=3.8cm,y=2.513cm]{v10}
 \Vertex[LabelOut=false,L=\hbox{$8$},x=2.5cm,y=1.7cm]{v8}
 \Vertex[LabelOut=false,L=\hbox{$9$},x=3.6cm,y=0.5cm]{v9}

 \foreach \from/\to in {0/3, 0/6, 1/0, 1/4, 2/1, 3/2, 3/7, 4/3, 4/5, 4/6, 5/2, 7/5}
 \Edge(v\from)(v\to);
 \foreach \from/\to in {0/11, 12/0, 11/4, 11/1, 4/12, 12/3, 8/4, 3/8, 7/8, 8/5, 5/9, 9/3, 2/9, 9/7, 10/4, 1/10, 5/10, 10/2, 6/12, 6/11}
 \Edge(v\from)(v\to);


 
 \end{tikzpicture}
\caption{$G_1(0,1,2,3)$ - For every partition into two acyclic sets $\A_1$ and $\A_2$, if $\{1,2
\}\subset\A_1$, then $\{8, 9, 10, 11, 12\}\not\subset\A_2$.}
\label{fig:gadget1}
\end{subfigure}
\hfill
\begin{subfigure}[b]{0.52\textwidth}
 \centering
 \begin{tikzpicture}[scale=1.1]
 \tikzset{->-/.style={
 decoration={
 markings,
 mark=at position #1 with {\arrow{>}}
 },
 bezier bounding box=true,
 postaction={decorate}
 },
 ->-/.default=0.55 
 }
 
 \tikzset{VertexStyle/.append style={minimum size=0.5cm, inner sep=0.01cm, draw}}
 
 \tikzset{EdgeStyle/.style = {->-,>=stealth'}}
 
 \Vertex[LabelOut=false,L=\hbox{$0$},x=0cm,y=5cm]{v0}
 \Vertex[LabelOut=false,L=\hbox{$3$},x=0cm,y=0cm]{v3}
 \Vertex[LabelOut=false,L=\hbox{$4$},x=2.5cm,y=2.513cm]{v4}
 \Vertex[LabelOut=false,L=\hbox{$5$},x=3.675cm,y=1.25cm]{v5}
 \Vertex[LabelOut=false,L=\hbox{$6$},x=1.2cm,y=3.8cm]{v6}
 \Vertex[LabelOut=false,L=\hbox{$7$},x=2.5cm,y=0.8cm]{v7}

 \tikzset{VertexStyle/.append style={minimum size=0.5cm, inner sep=0.01cm, color=blue!35, text=black}}
 \Vertex[LabelOut=false,L=\hbox{$1$},x=5cm,y=5cm]{v1}
 \Vertex[LabelOut=false,L=\hbox{$2$},x=5cm,y=0cm]{v2}
 
 \tikzset{VertexStyle/.append style={minimum size=0.5cm, inner sep=0.01cm, color=red!25, text=black}}
 
 \Vertex[LabelOut=false,L=\hbox{$12$},x=1.2cm,y=2.513cm]{v12}
 \Vertex[LabelOut=false,L=\hbox{$11$},x=2.5cm,y=3.8cm]{v11}
 \Vertex[LabelOut=false,L=\hbox{$10$},x=3.8cm,y=2.513cm]{v10}
 \Vertex[LabelOut=false,L=\hbox{$8$},x=2.5cm,y=1.7cm]{v8}
 \Vertex[LabelOut=false,L=\hbox{$9$},x=3.6cm,y=0.5cm]{v9}

 \foreach \from/\to in {0/3, 0/6, 1/0, 1/4, 2/1, 3/2, 7/3, 4/3, 4/5, 6/4, 5/2, 5/7}
 \Edge(v\from)(v\to);
 
 \foreach \from/\to in {11/0, 0/12, 4/11, 11/1, 12/4, 3/12, 4/8, 8/3, 8/7, 5/8, 9/5, 3/9, 9/2, 7/9, 10/4, 1/10, 10/5, 2/10, 12/6, 6/11}
 \Edge(v\from)(v\to);

 \Vertex[style={minimum size=1.0cm,shape=circle},LabelOut=false,L=\hbox{$13$},x=-1.7cm,y=2.5cm]{v13}
 \Edge(v13)(v0)
 \Edge(v3)(v13)
 \tikzset{EdgeStyle/.append style = {bend left=70}}
 \Edge(v2)(v13)
 \tikzset{EdgeStyle/.append style = {bend right=70}}
 \Edge(v1)(v13)
 
 \end{tikzpicture}

\caption{$G_2(1,2,13)$ - For every partition into two acyclic sets $\A_1$ and $\A_2$, if $\{1,2
\}\subset\A_1$, then $\{8, 9, 10, 11, 12, 13\}\not\subset\A_2$. }
\label{fig:gadget2}
\end{subfigure}

\begin{subfigure}{\textwidth}
 \centering
 \includegraphics[scale=1.2]{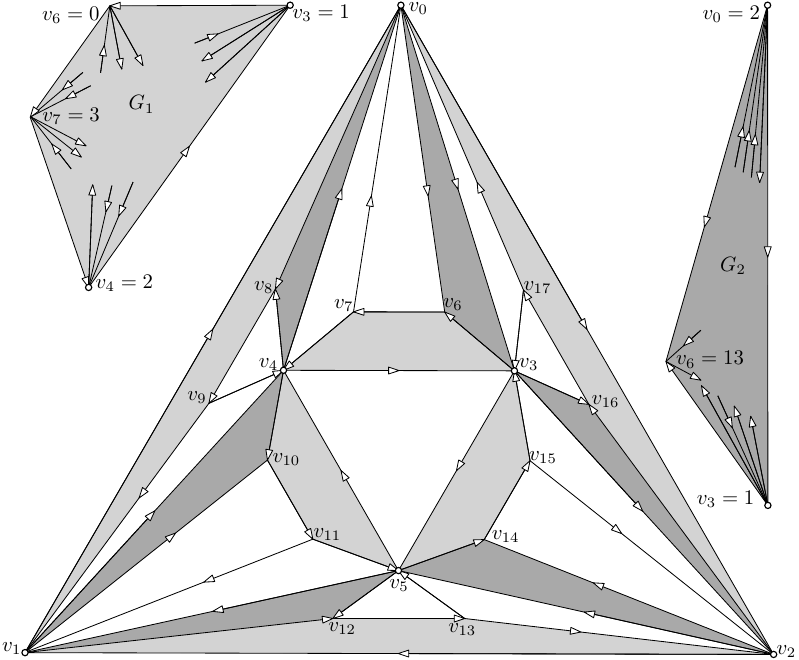}
 \caption{A Eulerian oriented triangulation where the light (resp. dark) gray face is isomorphic to $G_1$ (resp. $G_2$.)}
 \label{fig:eulerian-counterexample}
\end{subfigure}
\caption{The construction of the counterexample in \Cref{thm:oriented-triang}.}
\end{figure}
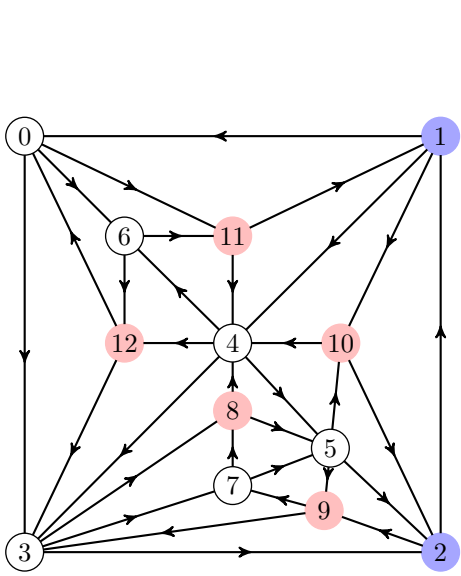
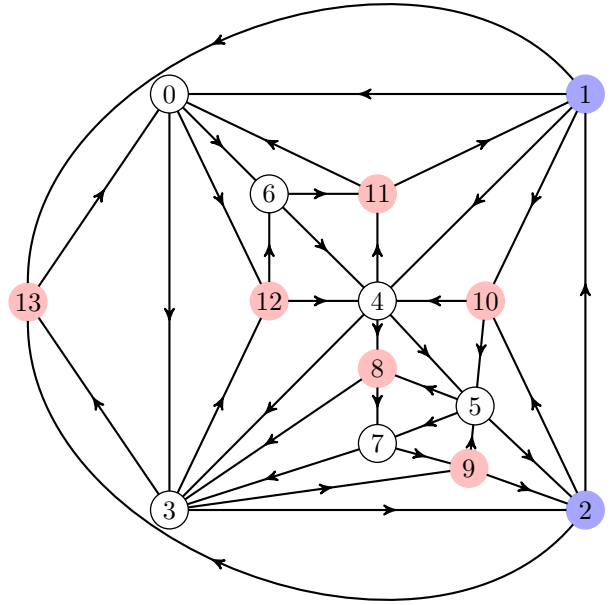
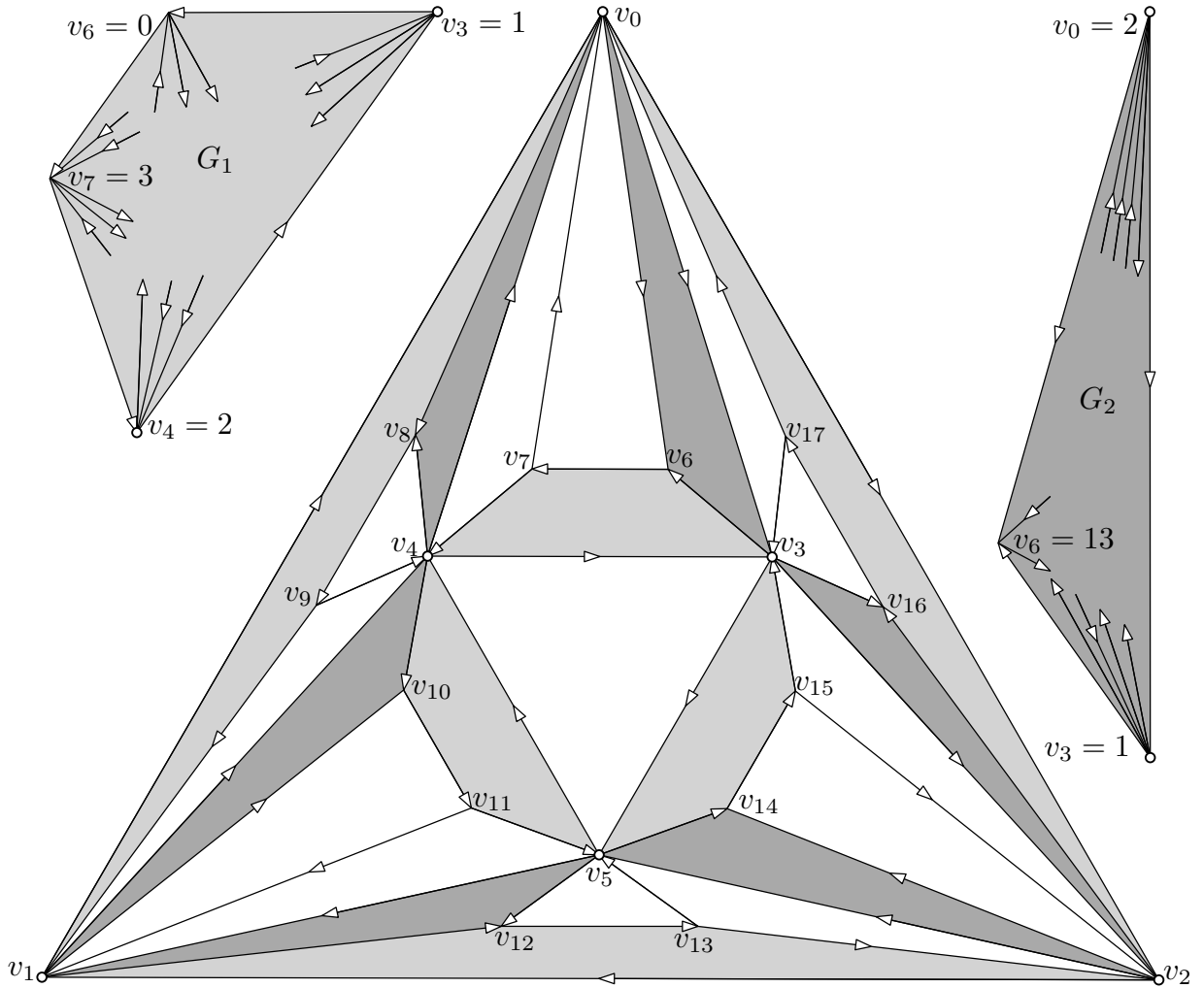

\begin{proof}[Proof of~\Cref{cor:finalboss}]
    Take $2k-1$ copies $G_1, \ldots G_{2k-1}$ copies of the graph $G$ from~\Cref{thm:oriented-triang} and identify the inner triangle $v_5, v_4, v_3$ of $G_i$ with the outer triangle $v_0, v_1, v_2$ of $G_{i+1}$ for $1\leq i\leq 2k-2$. The resulting graph $H$ is a Eulerian oriented planar triangulation. Let $I_1, I_2, I_3$ be its tripartition and suppose that $\mathcal{A}$ is a connected acyclic permeating subgraph such that $|\mathcal{A}\cap I_1|<k$. 
    By the pigeonhole principle there is an $1\leq i\leq 2k-1$ such that $G_i\cap I_1\cap \mathcal{A}=\emptyset$. Since $\mathcal{A}$ is connected then for any two vertices $u,v\in \mathcal{A}\cap G_i$ there is a $(u,v)$-path $P$. If $P$ leaves $G_i$, then  $P$ traverses one of the gluing triangles towards $G_{i\pm 1}$ on two adjacent vertices of the triangles and can be shortened so it remains in $G_i$. Hence, $\mathcal{A}\cap G_i$ is connected. But by~\Cref{obs:permeat} $\mathcal{A}$ and $\mathcal{I}=G_i\setminus I_1\setminus \mathcal{A}$ are a \CAI-partition of $G_i\setminus I_1$, which by~\Cref{thm:oriented-triang} implies that $\mathcal{A}\cap G_i$ is disconnected. Contradiction. 
\end{proof}

As a final remark of this section, we note that the underlying undirected graph of the construction obtained in \Cref{fig:eulerian-counterexample} is not a counterexample to \Cref{conj:Barnette} (and thus is not a counterexample to Conjectures~\ref{conj:EulerianNL} and \ref{conj:NL}). To see this, let $G_1$, $G_2$ be the underlying undirected graphs of $G_1(0,1,2,3)$, $G_2(1,2,13)$ respectively. An easy case analysis shows that every partition into two forests $A_1$ and $A_2$ of vertices $\{0,1,2,3\}$ of $G_1$, can be extended to a partition into two forests $A'_1\supset A_1$ and $A'_2\supset A_2$ of $G_1$, such that in the subgraph induced by $A_1$ in $G_1 - \{(0,1),(1,2),(2,3),(3,0)\}$, vertices of $A_1$ (resp. $A_2$) are not connected. A similar property can be shown for vertices $\{1,2,13\}$ of $G_2$. With this in hand, it is enough to give a valid partition into two forests of the undirected subgraph of \Cref{fig:eulerian-counterexample} induced by vertices $\{v_0,\ldots,v_{17}\}$ and  extend this partition to each of the light and dark faces.

\section{Conclusion}\label{conclusions}
Concerning \Cref{thm:subcubic}, each of the graphs of \Cref{fig:tight-graphs} has one less restriction and no \CAI-partition as shown in \Cref{lem:no-AI-graphs}. There is only one missing case that we leave as an open question:

\begin{question}\label{quest}
 Does every oriented bipartite or triangle-free $2$-vertex-connected subcubic graph admit a \CAI-partition?
\end{question}
\noindent Furthermore, we believe that~\Cref{thm:tw2} can be generalized in the following way:
\begin{conjecture}
The vertices of a graph $G$ of treewidth at most $k$, and connectivity at least $k$ can be partitioned into an induced graph $\mathcal{T}$ of treewidth at most $k-1$ and connectivity at least $k-1$ and an independent set $\mathcal{I}$.
\end{conjecture}
Considering treewidth $0$ graphs as independent sets, the case $k=1$ just says that trees are bipartite. \Cref{thm:tw2} corresponds to $k=2$ since $2$-vertex-connected simple series-parallel graphs are the $2$-vertex-connected graphs of treewidth $2$. Further, the conjecture holds for $k$-trees: just construct $G,\mathcal{I},\mathcal{T}$ along an elimination-ordering. Start with $K_{k+1},\{v\},K_{k+1}\setminus v$, for any $v\in K_{k+1}$. If a new vertex $u$ gets added and is adjacent to no element of $\mathcal{I}$, then add $u$ to $\mathcal{I}$ and add $u$ to $\mathcal{T}$ otherwise.

\subsubsection*{Acknowledgments:} 
We thank Franti\v{s}ek Kardo\v{s} for the initial discussion on the problems of this paper and for pointing out the result of Payan and Sakarovitch~\cite{PS75}. We further thank Marthe Bonamy for contributing to the proof of~\Cref{thm:subcubic}, and an anonymous referee for noticing that part of the proof could be shortened.

S.C. was supported by a FWO grant with grant number 1225224N and was supported during a research visit in 2021 by a Van Gogh grant, reference VGP.19/00015.
K.K was supported by the Spanish State Research Agency
through grants RYC-2017-22701, PID2022-137283NB-C22 and the Severo Ochoa and María de Maeztu Program for Centers and Units of Excellence in R\&D (CEX2020-001084-M) and the grant of The Natural Science Foundation of Hebei Province (project No. A2023205045). P.V. was partially supported by Agence Nationale de la Recherche (France) under research grant ANR DIGRAPHS ANR-19-CE48-0013-01. Moreover K.K. and P.V. were partially supported by Agence Nationale de la Recherche (France) under the JCJC program (ANR-21-CE48-0012).


\begin{thebibliography}{10}

\bibitem{APSW16}
{\sc H.~Alt, M.~S. Payne, J.~M. Schmidt, and D.~R. Wood}, {\em Thoughts on
  {Barnette}'s conjecture}, Australas. J. Comb., 64 (2016), pp.~354--365.

\bibitem{BFFS22}
{\sc B.~Bagheri~Gh., T.~Feder, H.~Fleischner, and C.~Subi}, {\em On finding
  {Hamiltonian} cycles in {Barnette} graphs}, Fundam. Inform., 188 (2022),
  pp.~1--14.

\bibitem{B68}
{\sc D.~Barnette}, {\em Conjecture 5}, in Recent Progress in Combinatorics:
  Proceedings of the Third Waterloo Conference on Combinatorics, W.~T. Tutte,
  ed., Academic Press, New York, 1968.

\bibitem{Bod98}
{\sc H.~L. Bodlaender}, {\em A partial k-arboretum of graphs with bounded
  treewidth}, Theoretical Computer Science, 209 (1998), pp.~1--45.

\bibitem{BCD24}
{\sc {\'E}.~Bonnet, D.~Chakraborty, and J.~Duron}, {\em Cutting {Barnette}
  graphs perfectly is hard}, Theor. Comput. Sci., 1010 (2024), p.~15.
\newblock Id/No 114701.

\bibitem{BGM22}
{\sc G.~Brinkmann, J.~Goedgebeur, and B.~D. Mckay}, {\em The minimality of the
  {G}eorges-{K}elmans graph}, Mathematics of Computation, 91 (2022),
  pp.~1483--1500.

\bibitem{CO02}
{\sc G.~L. Chia and S.-H. Ong}, {\em On {Barnette}'s conjecture and {CBP}
  graphs with given numbers of {Hamilton} cycles}, in Proceedings of the third
  Asian mathematical conference 2000, University of the Philippines, Diliman,
  Philippines, October 23--27, 2000, Singapore: World Scientific, 2002,
  pp.~94--111.

\bibitem{CBOS24}
{\sc M.~L.~L. da~Cruz, R.~S.~F. Bravo, R.~A. Oliveira, and U.~S. Souza}, {\em
  Near-bipartiteness, connected near-bipartiteness, independent feedback vertex
  set and acyclic vertex cover on graphs having small dominating sets}, in
  Combinatorial Optimization and Applications, W.~Wu and J.~Guo, eds., Cham,
  2024, Springer Nature Switzerland, pp.~82--93.

\bibitem{Epp92}
{\sc D.~Eppstein}, {\em Parallel recognition of series-parallel graphs},
  Information and Computation, 98 (1992), pp.~41--55.

\bibitem{Flo10}
{\sc J.~Florek}, {\em On {Barnette}'s conjecture}, Discrete Math., 310 (2010),
  pp.~1531--1535.

\bibitem{Flo16}
{\sc J.~Florek}, {\em On {Barnette}'s conjecture and the {{\(H^{+-}\)}}
  property}, J. Comb. Optim., 31 (2016), pp.~943--960.

\bibitem{Flo20b}
{\sc J.~Florek}, {\em Graphs with multi-$4$-cycles and the {Barnette}'s
  conjecture}.
\newblock Preprint, {arXiv}:2002.05288 [math.{CO}] (2020), 2020.

\bibitem{Flo20a}
{\sc J.~Florek}, {\em Remarks on {Barnette}'s conjecture}, J. Comb. Optim., 39
  (2020), pp.~149--155.

\bibitem{F24}
{\sc J.~Florek}, {\em A sufficient condition for cubic 3-connected plane
  bipartite graphs to be hamiltonian}, 2024.

\bibitem{Har13}
{\sc J.~Harant}, {\em A note on {Barnette}'s conjecture}, Discuss. Math., Graph
  Theory, 33 (2013), pp.~133--137.

\bibitem{Hoc17}
{\sc W.~Hochst{\"a}ttler}, {\em A flow theory for the dichromatic number},
  European Journal of Combinatorics, 66 (2017), pp.~160--167.

\bibitem{HMM85}
{\sc D.~Holton, B.~Manvel, and B.~McKay}, {\em Hamiltonian cycles in cubic
  3-connected bipartite planar graphs}, Journal of Combinatorial Theory, Series
  B, 38 (1985), pp.~279--297.

\bibitem{Hor82}
{\sc J.~D. Horton}, {\em On two-factors of bipartite regular graphs}, Discrete
  Mathematics, 41 (1982), pp.~35--41.

\bibitem{K20}
{\sc F.~Kardo\v{s}}, {\em A computer-assisted proof of the {B}arnette--{G}oodey
  {C}onjecture: Not only fullerene graphs are hamiltonian}, SIAM Journal on
  Discrete Mathematics, 34 (2020), pp.~62--100.

\bibitem{KT09}
{\sc K.-i. Kawarabayashi and C.~Thomassen}, {\em Decomposing a planar graph of
  girth 5 into an independent set and a forest}, J. Comb. Theory, Ser. B, 99
  (2009), pp.~674--684.

\bibitem{Khu89}
{\sc S.~Khuller}, {\em Ear decompositions}, abstract, SIGACT News 20, 128,
  1989.

\bibitem{knauer2024partitioning}
{\sc K.~Knauer, C.~Rambaud, and T.~Ueckerdt}, {\em Partitioning a planar graph
  into two triangle-forests}, arXiv:2401.15394,  (2024).

\bibitem{KV19}
{\sc K.~Knauer and P.~Valicov}, {\em Cuts in matchings of 3-connected cubic
  graphs}, European Journal of Combinatorics, 76 (2019), pp.~27--36.

\bibitem{LM17}
{\sc Z.~Li and B.~Mohar}, {\em Planar digraphs of digirth four are
  2-colorable}, SIAM Journal on Discrete Mathematics, 31 (2017),
  pp.~2201--2205.

\bibitem{Lu11}
{\sc X.~Lu}, {\em A note on {Barnette}'s conjecture}, Discrete Math., 311
  (2011), pp.~2711--2715.

\bibitem{MT01}
{\sc B.~Mohar and C.~Thomassen}, {\em Graphs on surfaces}, Baltimore, MD: Johns
  Hopkins University Press, 2001.

\bibitem{NSS23}
{\sc R.~Nedela, M.~Seifrtová, and M.~Škoviera}, {\em Decycling cubic graphs},
  Discrete Mathematics, 347 (2024), p.~114039.

\bibitem{NS22}
{\sc R.~Nedela and M.~Škoviera}, {\em Cyclic connectivity, edge-elimination,
  and the twisted {I}saacs graphs}, Journal of Combinatorial Theory, Series B,
  155 (2022), pp.~17--44.

\bibitem{NL85}
{\sc V.~Neumann-Lara}, {\em Vertex colourings in digraphs. some problems},
  technical report, University of Waterloo, July 8 1985.

\bibitem{PS75}
{\sc C.~Payan and M.~Sakarovitch}, {\em Ensembles cycliquement stables et
  graphes cubiques}, Cah. centr. et. rech. operat. (Colloq. theor. graphes,
  Paris, 1974), 17 (1975), pp.~319--343.

\bibitem{RW08}
{\sc A.~Raspaud and W.~Wang}, {\em On the vertex-arboricity of planar graphs},
  Eur. J. Comb., 29 (2008), pp.~1064--1075.

\bibitem{steiner2018}
{\sc R.~Steiner}, {\em Neumann-{L}ara-flows and the two-colour-conjecture},
  master's thesis, FernUniversit\"at in Hagen, Fakult\"at f\"ur Mathematik und
  Informatik, 2018.

\bibitem{Tait1884}
{\sc P.~G. Tait}, {\em Listing's topologie}, Philosophical Magazine, 17 (1884),
  pp.~30--46.
\newblock Reprinted in Scientific Papers, Vol. II, pp. 85--98.

\bibitem{Tho95}
{\sc C.~Thomassen}, {\em Decomposing a planar graph into an independent set and
  a 3-degenerate graph}, J. Comb. Theory, Ser. B, 83 (2001), pp.~262--271.

\bibitem{TW11}
{\sc M.-T. Tsai and D.~B. West}, {\em A new proof of 3‑colorability of
  {E}ulerian triangulations}, Ars Mathematica Contemporanea, 4 (2011),
  pp.~73--77.

\bibitem{Tutte46}
{\sc W.~T. Tutte}, {\em On {H}amiltonian circuits}, Journal of the London
  Mathematical Society, s1-21 (1946), pp.~98--101.

\bibitem{Tut71}
{\sc W.~T. Tutte}, {\em On the 2-factors of bicubic graphs}, Discrete
  Mathematics, 1 (1971), pp.~203--208.

\bibitem{Whi32}
{\sc H.~Whitney}, {\em Non-separable and planar graphs}, Transactions of the
  American Mathematical Society, 34 (1932), pp.~339--362.

\end{thebibliography}
\end{document}